%% file: main.tex
\documentclass[11pt]{article}
\usepackage[utf8]{inputenc}
\usepackage{amsmath,amssymb,amsthm}
\usepackage{graphicx}
\usepackage[margin=1.1in]{geometry}
\usepackage[hidelinks]{hyperref}
\usepackage{placeins}
\usepackage{longtable}
\usepackage{tikz}
\usepackage{array}

\newtheorem{theorem}{Theorem}
\newtheorem{lemma}[theorem]{Lemma}
\newtheorem{proposition}[theorem]{Proposition}
\newtheorem{corollary}[theorem]{Corollary}

\theoremstyle{definition}
\newtheorem{definition}[theorem]{Definition}
\newtheorem{remark}[theorem]{Remark}
\newtheorem{openproblem}[theorem]{Open Problem}

\newcommand{\R}{\mathbb{R}}

\title{Greedy Packing of Nested Rings:\\ The Golden Threshold,
Placement Rules, and a Tribonacci Floor}
\author{Javier Aguilar Mart\'in\thanks{AGILabs.
\texttt{javiecija96@gmail.com}. Computational verification scripts for
every numerical claim, together with extended proofs and the verification
reports cited in the appendix, are available in the accompanying repository,
\url{https://github.com/JaviMaligno/calamares}. This version of the
manuscript is the integrated v2 working revision of 15 September 2026, extending
the archived release \texttt{v1-arxiv} of that repository;
\texttt{python code/run\_all.py --campaign} reproduces every
verification in the map (band-by-band, several hours);
\texttt{--full} covers the core-theorem scripts in about $35$
minutes, and \texttt{lake build} in \texttt{lean/} rechecks the
formalized certificates. \texttt{requirements.txt} pins the exact
environment.}}
\date{}

\begin{document}
\maketitle

\begin{abstract}
We study packings of annuli of a common width, allowing each ring to
nest inside the hole of a larger one. The objectives of maximizing
contact area and the number of placed rings diverge: area is
superadditive in the radius, cardinality is not. Under superincreasing
radii, every descending greedy maximizes every positive, strictly
increasing, superadditive objective. More strongly, any choice among
feasible containers yields the lexicographically maximal feasible set,
for containers of arbitrary shape in every dimension. This placement
irrelevance holds unconditionally for at most three rings and fails
at four in disks and squares; twin instances exclude every universal
rule based only on the observable state.

Write $\rho=\max_i(\sum_{j>i}r_j)/r_i$. The additive model has
threshold exactly $1$. For disks we prove the exact global threshold
$\tau=\varphi$, with no failure at $\rho\le\varphi$, for every finite
inventory, even with independent hole radii. The key geometric theorem
states that, under golden tail bounds, an entire disk list fits a
circular container if and only if its three largest disks fit. It
supplies a uniform exchange of parents, matching the four-ring golden
counterexamples approaching $\varphi$ from above. The Tribonacci
constant $T\approx1.83929$ remains the exact floor of a rigid subfamily.

A dimension-reduction lemma transfers spherical sharpness results to
all dimensions $d\ge2$, and a separate argument proves the golden
threshold for at most five rings in those dimensions. For square pans,
a Cartesian confinement criterion gives twins and a family proving
$\tau_{\square}\le Y\approx1.684487745872346$; its optimality is open.
For independent holes, the exact universal area guarantee under
$\rho\le\kappa<1$ is $\min(1,\kappa^{-2}-1)$, with optimality threshold
$1/\sqrt2$. The repository contains 122 Lean algebraic and Cartesian
theorems. Euclidean geometry, forest assembly and continuity remain
written proofs; numerical checks do not substitute for them.
\end{abstract}

\section{Introduction}

Drop a batch of squid rings into a frying pan and a natural optimization
problem appears: place as many rings as possible flat on the pan -- or, if one
cares about searing, maximize the total ring--pan contact area -- exploiting the
fact that a sufficiently small ring fits inside the hole of a larger one.
Abstracted, this is the problem of packing annuli of common width into a disk
with recursive nesting, a circular-container relative of the
\emph{Recursive Circle Packing Problem} (RCPP) introduced by Pedroso, Cunha and
Tavares~\cite{PedrosoCunhaTavares2016} to model the ``telescoping'' of tubes in
shipping containers, and solved exactly by Gleixner, Maher, M\"uller and
Pedroso~\cite{Gleixner2020}. That literature is algorithmic: it packs a fixed
set of rings into as few rectangular containers as possible, and its
methods -- like the greedy and GRASP procedures for RCPP and the greedy
heuristics for equal circles in a circular
container~\cite{PedrosoCunhaTavares2016,ChenEtAl2018} -- are
\emph{heuristics}: they propose placements without guarantees of
optimality or of independence from the placement rule. Here we ask
structural questions instead: which objectives the natural greedy
algorithm \emph{provably} optimizes, when the choice of \emph{where} to
put each ring is provably irrelevant, and which conditions on the radii
govern the answers; to our knowledge, no comparable
placement-independence guarantee appears in the circle-packing literature.

Our starting point is a classical phenomenon in one dimension. For bin packing
with \emph{divisible} item sizes, Coffman, Garey and
Johnson~\cite{CoffmanGareyJohnson1987} proved that First Fit Decreasing is
optimal. We are not aware of an analogue for geometric packing of disks, let
alone with nesting. We prove one, with \emph{superincreasing} radii (each radius exceeding
the sum of all smaller ones) playing the role of divisibility, and we
determine exactly how far the hypothesis can be relaxed in the additive
model, bracketing the geometric threshold. The answer is a
surprise: the additive relaxation of the model has threshold exactly $1$,
while the disk model has threshold exactly the
\emph{golden ratio} -- an explicit golden family fails at
$\rho=\varphi+3\varepsilon$. The Tribonacci constant -- the natural
conjectured threshold, which our own program long targeted -- survives as
the exact floor of the rigid nested family, the proved floor of a
hierarchy of nested-exchange templates, and the conjectured universal
floor of nested exchanges.

\paragraph{Contributions: main theorems.}
(1)~A \emph{superadditivity dichotomy}: under superincreasing radii the
selection greedy is optimal for every positive, strictly increasing,
superadditive value function of the radius -- contact area is one -- while
cardinality (not superadditive) admits counterexamples even under
superincreasing radii (Section~\ref{sec:superadd}). The selection statement
is best seen as a general fact about downward-closed set systems with
superincreasing weights, in the lineage of greedy analyses on independence
systems~\cite{Edmonds1971,KorteHausmann1978} and of the superincreasing
knapsack~\cite{Gupte2016}; the geometric content of this paper begins with
the next item.
(2)~A \emph{placement-obliviousness theorem}: under superincreasing radii,
\emph{any} rule for choosing among feasible containers yields the
lexicographically maximal feasible set; the proof is container-shape- and
dimension-agnostic (Section~\ref{sec:oblivious}).
(3)~\emph{Sharpness}: placement obliviousness holds unconditionally for
$n\le3$ rings and fails at $n=4$; twin instances with an identical prefix rule
out every state-based placement rule, deterministic or randomized
(Section~\ref{sec:sharp}).
(The positive statement holds in every compact container and dimension;
Section~\ref{sec:dimension} transfers the spherical negative results
to every $d\ge2$.)
(4)~\emph{Thresholds}: the additive model has universal threshold exactly
$\rho=1$; the geometric four-ring rigid subfamily has infimum exactly the
Tribonacci constant $T$, proved with no idealization
(Theorem~\ref{thm:rigidfloor}, full proof in
Appendix~\ref{app:rigidproof}); and an explicit \emph{golden family}
refutes the natural conjecture that the geometric threshold equals $T$:
placement obliviousness fails at $\rho=\varphi+3\varepsilon$ for every
small $\varepsilon>0$, which \emph{proves} $\tau\le\varphi<T$
(Theorem~\ref{thm:golden}), and the family realizes every
$\rho\in(\varphi,T)$ at a fixed width
(Corollary~\ref{cor:goldencover}). The matching universal guarantee
is Theorem~\ref{thm:golden-global}: every finite disk inventory with
$\rho\le\varphi$ is placement-oblivious, including independent holes.
Its proof uses the three-largest-disks theorem, rather than a
classification by inventory size or a computational coverage claim.
The earlier four-ring proof is retained as a direct special-case analysis.

The status of the golden value, at a glance:
\begin{center}
\begin{tabular}{ll}
\hline
Statement & Status \\
\hline
$\tau\le\varphi$ & proved (golden family, Theorem~\ref{thm:golden}) \\
$\tau_4=\varphi$, infimum unattained & proved
(Theorem~\ref{thm:fourfloor}) \\
Pan-exchange floor for pair profiles equals $\varphi$ & proved
(Theorem~\ref{thm:DP}) \\
Universal $\tau=\varphi$, infimum unattained & proved
(Theorem~\ref{thm:golden-global}) \\
\hline
\end{tabular}
\end{center}

\paragraph{Specialized exchange floors.}
(5)~\emph{The width program}: the blocked-exchange infimum with relative
width $\omega>0$ is governed by the deformed Tribonacci cubic
$\alpha^3=(1+\omega)(\alpha^2+\alpha+1)$; the full lower-bound curve is in
closed form and its global minimum is exactly $13/7$, attained at a rational
corner (Section~\ref{sec:width}); the combinatorial profile threshold is
in closed form for every profile size, with golden plateau and exact Tribonacci
crossing $\omega_T=1/T-1/2$.
(6)~\emph{Generic containers}: an exact corona criterion for wall-tangent
packings (with two corrective counterexamples and a $k=5$ ``pentagram''
certificate), and a hierarchy of blocking walls -- occupant holes, ``blocking
costs the slack'', and the double pocket -- yielding metallic-mean floors
(silver ratio $1+\sqrt2$, golden line $\varphi^2-(\varphi/2)\omega$) that
push every blocked exchange with extra occupants above $T$ -- throughout
the width convention $\omega\in(0,1)$ of the program: for widths
$\omega<0.9626$ with one extra occupant, up to $0.624$ and $0.896$ with two
and three, and for all $\omega\in(0,1)$ with four or more
(Section~\ref{sec:generic}). The remaining gaps of the program are stated
explicitly there and in Open Problem~\ref{op:assembly}; they concern
stronger template bounds, not the global golden guarantee.
(7)~A phase diagram for the divergence of the two objectives, and a
decoupling of hardness into a geometric and a combinatorial layer, the
latter with an explicit \textsc{Partition} reduction
(Sections~\ref{sec:divergence}--\ref{sec:hardness}).
(8)~\emph{Shape and dimension}: sibling queries involving $k$ balls in
a ball reduce to dimension $k-1$. This preserves the sharpness results
and the rigid Tribonacci floor for every spherical dimension $d\ge2$.
For square pans an exact quadrant-confinement criterion gives four-ring
failures, state-indistinguishable twins, and a family proving
$\tau_{\square}\le Y\approx1.684487745872346$
(Sections~\ref{sec:dimension}--\ref{sec:squarelimit}); these results
make no claim of an exact global square threshold.
(9)~\emph{Independent holes and area}: placement obliviousness under
$\rho\le1$ remains valid for arbitrary compact containers, but contact
area is no longer a common function of outer radius. We prove the exact
factor $\min(1,\kappa^{-2}-1)$ under $\rho\le\kappa<1$ and the sharp
area-optimality threshold $1/\sqrt2$ (Section~\ref{sec:variablewidth}).

\section{Model and preliminaries}\label{sec:model}

\begin{definition}[Rings and placements]
Fix a width $w>0$. A \emph{ring} of outer radius $r>0$ is an annulus with
inner (hole) radius $\max(0,\,r-w)$; if $r\le w$ it is a solid disk with no
hole. Let $K\subset\R^d$ be a compact \emph{container} (the \emph{pan}; our
motivating case is a disk of radius $R$ in the plane, $d=2$, but
Theorems~\ref{thm:selection} and~\ref{thm:oblivious} hold for arbitrary $K$ and
$d$, with rings read as spherical shells). A \emph{placement} of a subset $S$
of the rings is a forest: each placed ring has as parent either the pan or a
placed ring with $r_{\mathrm{child}}\le r_{\mathrm{parent}}-w$ (it sits inside
the parent's hole); the children of a common parent must be packable as
balls (by outer radius) with pairwise disjoint interiors inside the parent's region --- $K$
itself, or the hole ball of radius $r-w$. Feasibility is a property of the
assignment (siblings may be rearranged freely inside their container).
\end{definition}

\begin{definition}[Objectives; superincreasing radii]
For radii $r_1>r_2>\dots>r_n$, the \emph{contact area} of a placed ring is
$a(r)=\pi\bigl(r^2-\max(0,r-w)^2\bigr)$; the two objectives are
$N=|S|$ and $A=\sum_{i\in S}a(r_i)$ over feasible $S$. The radii are
\emph{superincreasing} if $r_i>\sum_{j>i}r_j$ for all $i$; the violation is
measured by $\rho=\max_i\bigl(\sum_{j>i}r_j\bigr)/r_i$.
We call $\bigl(\sum_{j>i}r_j\bigr)/r_i$ the \emph{tail} (ratio) of ring
$i$, so that $\rho$ is the largest tail; ``the tail of $x$'' always
refers to this ratio.
The \emph{lex-max} set $L$ is defined greedily: scanning $i=1,\dots,n$,
include $i$ whenever $(L\cap\{1,\dots,i-1\})\cup\{i\}$ is feasible.
Feasibility is \emph{downward closed}: removing a ring from a feasible
placement leaves a feasible placement of the remaining set --- the
removed ring's disk becomes free space, and its nested children (each
of radius at most $r-w$, less than the freed disk) stay where
they are, re-parented to the removed ring's container, so no piece
moves. Hence $L$ is the unique lexicographically
maximal feasible set, and the \emph{selection greedy} (add ring $i$ if the
resulting \emph{set} remains feasible) computes it. Note that the selection
greedy queries a \emph{set-feasibility oracle}; it is an information bound,
not an efficient algorithm -- deciding sibling packability is itself hard
(Section~\ref{sec:hardness}). The placement greedies below, by contrast,
only ever attempt concrete insertions.
A \emph{descending greedy with placement rule} processes rings in decreasing
radius and places each into \emph{some} feasible container (the pan or the hole
of an already placed ring), skipping it only if no container admits it.
\end{definition}

\begin{lemma}[Row lemma]\label{lem:row}
Balls with radii $x_1,\dots,x_k$ satisfying $\sum_j x_j\le C$ pack
inside a ball of radius $C$: place them tangent in a row along a diameter.
Consequently, if a ball of radius $m$ is removed from any packing, any
collection of balls of total radius at most $m$ can be inserted into the
vacated ball without disturbing the rest.
\end{lemma}
\begin{proof}
With centers on a diameter at coordinates
$p_j=-C+2\sum_{l<j}x_l+x_j$, each satisfies $|p_j|\le C-x_j$.
\end{proof}

\section{The superadditivity dichotomy}\label{sec:superadd}

\begin{lemma}\label{lem:superadd}
The contact area $a$ is strictly increasing and superadditive on $(0,\infty)$:
$a(x)+a(y)\le a(x+y)$.
\end{lemma}
\begin{proof}
Monotonicity is clear. For $x,y\ge w$:
$\pi w(2x-w)+\pi w(2y-w)=\pi w\bigl(2(x{+}y)-2w\bigr)\le\pi w\bigl(2(x{+}y)-w\bigr)$.
For $x,y\le w$ with $x+y\le w$: $x^2+y^2\le(x+y)^2$.
For $x,y\le w<x+y$: $x^2+y^2\le w(x+y)\le 2w(x+y)-w^2$, using
$x^2\le wx$, $y^2\le wy$ and $x+y\ge w$.
For $x\le w\le y$: $\pi x^2+\pi w(2y-w)\le\pi w(2x+2y-w)$ since $x^2\le 2wx$.
\end{proof}

\begin{lemma}[Lexicographic dominance]\label{lem:lexdom}
Let the radii be superincreasing and let $v$ be any positive, increasing,
superadditive function of the radius. If $S,T$ are feasible sets and the first
index (in decreasing radius order) at which they differ belongs to $S$, then
$\sum_S v(r_i)\ge\sum_T v(r_i)$, with equality only if $T\subseteq S$.
\end{lemma}
\begin{proof}
Let $i$ be the first differing index and $P=S\cap\{1,\dots,i-1\}
=T\cap\{1,\dots,i-1\}$. If $T\setminus P=\emptyset$ then $T\subseteq S$ and
monotonicity of the sum under inclusion suffices. Otherwise
$T\setminus P\subseteq\{i{+}1,\dots,n\}$, so
$\sum_{T\setminus P}r_j<r_i$ by superincreasingness, and by iterated
superadditivity and monotonicity
$\sum_{T\setminus P}v(r_j)\le v\bigl(\sum_{T\setminus P}r_j\bigr)<v(r_i)$.
Hence $\sum_T v=\sum_P v+\sum_{T\setminus P}v<\sum_P v+v(r_i)\le\sum_S v$.
\end{proof}

\begin{theorem}[Selection greedy; arbitrary container, any dimension]
\label{thm:selection}
With superincreasing radii, the selection greedy computes the lex-max feasible
set and therefore maximizes $\sum_{i\in S}v(r_i)$ for every positive,
strictly increasing, superadditive $v$; in particular it maximizes the
contact area $A$.
\end{theorem}
\begin{proof}
Feasibility is downward closed, so the selection greedy computes $L$
(if it could add $i\notin L$, any maximal feasible extension would
lexicographically beat $L$). Lemma~\ref{lem:lexdom} shows $L$ dominates every
feasible set.
\end{proof}

\begin{remark}[Positioning]
Theorem~\ref{thm:selection} is at heart a statement about arbitrary
downward-closed set systems with superincreasing weights, and as such it
belongs to the classical circle of ideas around greedy algorithms on
independence systems~\cite{Edmonds1971,KorteHausmann1978} and the
superincreasing knapsack, whose lexicographic structure was analyzed by
Gupte~\cite{Gupte2016}. The hypotheses on $v$ are all needed: positivity
rules out $v\equiv-1$ (superadditive and non-decreasing, yet maximized by
the empty set), and strict monotonicity is what makes the dominance in
Lemma~\ref{lem:lexdom} strict. The geometric novelty of this paper lies in
Theorem~\ref{thm:oblivious}, where the exchange argument must respect the
geometry of nested containers.
\end{remark}

\begin{proposition}[Cardinality is not rescued]\label{prop:count}
There are superincreasing instances on which every descending greedy is
suboptimal for $N$. Example (verified; script \texttt{superinc.py}): pan $R=10$, $w=4.8$, radii
$\{9.95,\,5.0,\,4.3,\,0.6\}$. Greedy places $\{9.95,5.0\}$ ($N=2$, optimal
area; no step offers a choice of container, so all placement rules coincide),
while $\{5.0,4.3,0.6\}$ packs in a row in the pan ($N=3$).
The obstruction is structural: $N$ corresponds to $v\equiv1$, which is not
superadditive, and Lemma~\ref{lem:lexdom} fails for it.
\end{proposition}

\section{Placement obliviousness under superincreasing radii}
\label{sec:oblivious}

\begin{theorem}[Placement rule irrelevance]\label{thm:oblivious}
Let the radii be superincreasing (weakly: $r_i\ge\sum_{j>i}r_j$ suffices;
under the weak hypothesis the dominance of Lemma~\ref{lem:lexdom} holds with
non-strict inequality, which is enough for optimality) and
let $K\subset\R^d$ be an arbitrary container. Every descending greedy, with an
\emph{arbitrary} rule for choosing among feasible containers, places exactly
the lex-max feasible set. In particular best fit (place in the feasible
container of smallest capacity, defined for a general container as the radius
of its largest inscribed ball), worst fit (largest capacity) and random
placement are all optimal for every positive, strictly increasing,
superadditive objective.
\end{theorem}
\begin{proof}
Let $F$ denote the greedy's placement just before processing ring $i$, with
placed set $G$, and suppose $G\cup\{i\}$ is feasible via a witness placement
$P$; it suffices to show some container of $F$ admits $i$ (the converse
direction is immediate, and then greedy's set equals $L$ by induction as in
Theorem~\ref{thm:selection}). The available containers -- the pan plus one hole
per ring of $G$ -- coincide in $F$ and $P$, and in $P$ the smallest ring $i$
parents nobody.

Let $m$ be the largest ring of $G$ assigned different containers by $F$ and
$P$ (if none, $P$ restricted to $G$ equals $F$ and $i$'s container in $P$
witnesses admission). Let $u$ be $F$'s container for $m$ and $v$ be $P$'s.
Since the greedy processes rings in decreasing order, when $F$ placed $m$ the
occupants of $u$ were exactly the rings larger than $m$ that $F$ assigns to
$u$; by maximality of $m$ these coincide with the rings larger than $m$ that
$P$ assigns to $u$, and $F$ certified that this set together with $m$ packs in
$u$. Build $P'$ from $P$: (i) every ring smaller than $m$ that $P$ kept in $u$
is moved into the ball of radius $r_m$ vacated by $m$ inside $v$ -- they fit in
a row by Lemma~\ref{lem:row}, because their total radius is at most
$\sum_{j:\,r_j<r_m}r_j\le r_m$, and they touch nothing else in $v$;
(ii) $m$ moves to $u$, joining exactly the larger-than-$m$ rings whose
compatibility $F$ certified. Nesting constraints persist: $u$ and $v$ are the
pan or holes of rings larger than $m$, which do not move, and every ring moved
into $v$ is smaller than $r_m$, which fitted in $v$. Thus $P'$ is a feasible
witness agreeing with $F$ on all rings of radius $\ge r_m$: the first
disagreement strictly shrinks. Iterating at most $|G|$ times yields a witness
$P^{*}$ agreeing with $F$ on all of $G$ and placing $i$ in some container
$c^{*}$; since the occupants of $c^{*}$ in $P^{*}$ are exactly those in $F$,
the greedy admits $i$ there.
\end{proof}

\begin{remark}
The proof uses only (a) downward closure of packings, (b) the row lemma
applied \emph{inside the vacated ball}, and (c) the total-radius bound from
superincreasingness. Nothing refers to the shape of $K$ or the dimension:
the theorem covers rectangular pans (griddles) and spherical shells nested in
arbitrary containers in $\R^d$ verbatim. Computational corroboration of the
strongest consequence: on $100$ random superincreasing instances, best-fit,
worst-fit and random placement produced optimal (hence, by
Lemma~\ref{lem:lexdom}, identical) outcomes without exception (script
\texttt{test\_oblivious.py}; beyond the exact one-circle, two-circle and row
criteria its feasibility oracle is a heuristic packer with restarts -- a
sampling caveat of the corroboration, not an ingredient of the proof).
\end{remark}

\section{Sharpness: the \texorpdfstring{$n=4$}{n=4} transition and twin instances}\label{sec:sharp}

\begin{proposition}[$n\le3$, arbitrary container]\label{prop:n3}
For any compact container $K\subset\R^d$, $d\ge1$, and arbitrary
strictly decreasing radii of common positive width, every descending
greedy on at most three rings places the lex-max set.
\end{proposition}
\begin{proof}
A ring skipped before any has been admitted cannot occur in any
feasible set: its outer ball would have to lie in $K$, even if nested.
Discard such rings. With at most two remaining there is no previous
assignment choice that can obstruct admission. With three, if the second
ring is skipped, the prefix pair is infeasible because its only possible
parents are the root and the first ring, and both were tested. The last
decision then involves only one previously admitted ring.

It remains to consider $r_1>r_2>r_3$ with the first two admitted. Let
$F$ be their greedy assignment and suppose a witness $P$ places all three.
If the parent of $r_2$ agrees in $F$ and $P$, the latter's home for $r_3$
works in $F$, since siblings may be rearranged. If $F$ nests $r_2$ but
$P$ keeps it at the root, shrink that root ball concentrically to radius
$r_3$: it fits beside $r_1$ in $K$, so $F$ admits $r_3$ at the root.
If $F$ keeps $r_2$ at the root but $P$ nests it, then
$r_3<r_2\le r_1-w$, so the empty hole of $r_1$ in $F$ admits $r_3$.
Thus the third ring is admitted whenever the full prefix is feasible.
Conversely every legal admission is a feasibility witness. No convexity
or special shape of $K$ was used.
\end{proof}

\begin{lemma}[Cap confinement]\label{lem:cap}
Circles of radii $10,\,4.9,\,4.8$ do not pack in a disk of radius $15$.
\end{lemma}
\begin{proof}
Write $c_A,c_X,c_Y$ for the centers, measured from the pan center. Then
$|c_A|\le5$, $|c_X|\le10.1$, $|c_Y|\le10.2$, $|c_A-c_X|\ge14.9$,
$|c_A-c_Y|\ge14.8$, $|c_X-c_Y|\ge9.7$. From $|c_A|+|c_X|\ge14.9$,
$a:=|c_A|\in[4.8,5]$. Let $u=-c_A/|c_A|$. Then
$c_X\!\cdot\!u=\bigl(|c_X-c_A|^2-|c_X|^2-a^2\bigr)/(2a)
\ge(120-a^2)/(2a)\ge9.5$ (decreasing in $a$), so the transverse part of
$c_X$ has norm at most $\sqrt{10.1^2-9.5^2}\le3.43$; similarly
$c_Y\!\cdot\!u\ge(115-a^2)/(2a)\ge9.0$ with transverse norm at most
$\sqrt{10.2^2-9^2}=4.8$. Hence
$|c_X-c_Y|^2\le10.1^2+10.2^2-2(9.5\cdot9.0-3.43\cdot4.8)\le67.98$,
i.e.\ $|c_X-c_Y|\le8.25<9.7$, a contradiction.
\end{proof}

\begin{theorem}[Failure at $n=4$]\label{thm:n4}
Placement obliviousness fails for four rings: with $R=15$, $w=0.3$ and radii
$\{10,\,5,\,4.9,\,4.8\}$ (all four form the lex-max set, witnessed by the pan
pair $\{10,5\}$ at exact tangency and the hole pair $\{4.9,4.8\}$ filling the
hole $9.7$ exactly), best fit nests the $5$, forces the $4.9$ into the pan,
and blocks the $4.8$ everywhere by Lemma~\ref{lem:cap} and the exact two-circle
conditions; worst fit places all four (Figure~\ref{fig:n4}). (The worst-fit
run also uses that $\{10,5,z\}$ does not pack in the pan for $z\in\{4.9,4.8\}$:
the pan pair is diametrically rigid and the rescaled pocket of
Proposition~\ref{prop:S5} is $10\,b(1/2)=30/7\approx4.286<z$; the count of
four is in any case robust to either oracle answer.)
\end{theorem}

\begin{figure}[ht]
\centering
\includegraphics[width=0.95\textwidth]{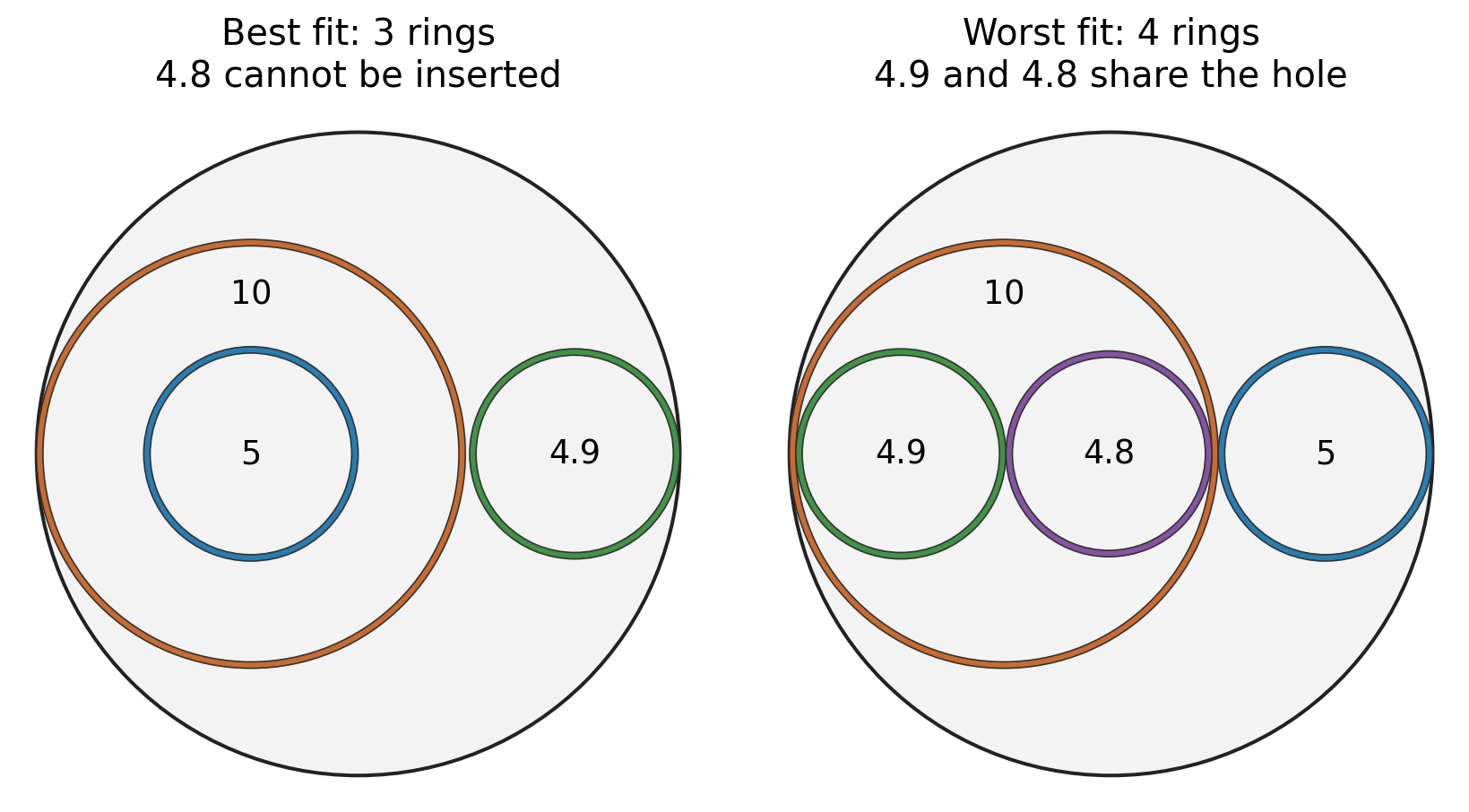}
\caption{The $n=4$ counterexample of Theorem~\ref{thm:n4}: best fit (left)
nests the $5$ and loses the $4.8$; the witness (right), realized by worst fit,
places all four rings.}
\label{fig:n4}
\end{figure}

\begin{theorem}[Twin instances; no state-based rule]\label{thm:twins}
Let $R=15$, $r_1=10$, $r_2=5$, $w=0.505$ (hole $9.495$), and
\[
I_1=\{10,\,5,\,4.99,\,4.50\},\qquad I_2=\{10,\,5,\,4.76,\,4.74\}.
\]
In $I_1$ the correct placement of the $5$ is the pan (best fit fails, worst
fit succeeds); in $I_2$ it is the hole (worst fit fails, best fit succeeds).
At the decisive step the observable state -- containers, capacities, occupants,
incoming ring, $R$ and $w$ -- is identical in $I_1$ and $I_2$. Consequently no
deterministic placement rule that is a function of the state attains the
lex-max set on all instances, and for every randomized rule some instance is
failed with probability at least $1/2$.
\end{theorem}
\begin{proof}[Proof ingredients]
Feasibilities are two-circle-exact except for two triples, both settled
rigorously: $\{10,4.99,4.50\}$ does not pack in the disk of radius $15$ (cap
confinement as in Lemma~\ref{lem:cap} yields $|c_X-c_Y|\le8.80<9.49$), and the
pan pair $\{10,5\}$ is diametrically rigid, whence no third circle of radius
$\ge4.7$ fits (confinement again: for $z=4.76$,
$|c_z-c_B|\le\sqrt{204.86-20\cdot8.8}\le5.38<9.76$; in general
$|c_z-c_B|^2\le z^2-130z+625<(5+z)^2$ for every $z>30/7$, so the claim holds
for all $z\ge4.7$). Packability of
$\{10,4.76,4.74\}$ in radius $15$ is exhibited by an explicit boundary
configuration (angles $180^\circ$, $30.8^\circ$, $-32.0^\circ$; all pairwise
constraints verified). The key phenomenon enabling the twins is that
triple-packability is \emph{not monotone in the sum}: the unbalanced pair
$\{4.99,4.50\}$ (sum $9.49$) fails beside the $10$ while the balanced pair
$\{4.76,4.74\}$ (sum $9.50$) succeeds.
The four executions, every step after the decisive one forced (a single
feasible container or none): in $I_1$, nesting the $5$ sends the $4.99$ to the
pan ($5+4.99>9.495$) and blocks the $4.50$ everywhere (pan by the triple
fact, hole by $5+4.50>9.495$, nesting by capacity) -- three rings -- while
the $5$ in the pan sends both $4.99$ and $4.50$ to the hole
($4.99+4.50=9.49\le9.495$) -- all four; in $I_2$, nesting the $5$ sends
$4.76$ and $4.74$ to the pan (the explicit boundary configuration) -- all
four -- while the $5$ in the pan blocks the pan for both ($\ge4.7$), nests
the $4.76$ alone and blocks the $4.74$ ($4.76+4.74=9.50>9.495$) -- three
rings. See Appendix~\ref{app:verif}.
\end{proof}

\begin{remark}
The obstruction is informational, not existential: following any witness of
the lex-max set is itself a legal greedy execution (a ring outside the lex-max
set is never admissible), so some execution always succeeds; what cannot exist
is a state-based rule that finds it. For rules with access to the full input
the question becomes one of computational complexity and remains open
(Section~\ref{sec:open}).
\end{remark}

\section{Thresholds: the Tribonacci floor and the golden counterexample}
\label{sec:threshold}

\begin{theorem}[Additive model: threshold exactly $1$]\label{thm:additive}
Replace sibling feasibility by its additive surrogate (a set of siblings fits
iff their radii sum to at most the container capacity; exact for at most two
circles and sufficient in general by Lemma~\ref{lem:row}). Then the
\emph{universal} threshold for placement obliviousness is exactly $1$:
obliviousness holds for \emph{all} instances with $\rho\le1$ (the proof of
Theorem~\ref{thm:oblivious} applies verbatim), while for every
$\varepsilon>0$ \emph{some} instance with $\rho\le1+\varepsilon$ fails --
individual instances with larger $\rho$ may of course still succeed.
Indeed, fix $r_1$, set $s=r_1/2$ and $R=r_1+s=\tfrac32r_1$, pick any
$w\in\bigl(\tfrac{r_1}4,\,\tfrac{r_1}2-\delta\bigr)$, and take $r_2=s$
and the exact pair
$r_3=\tfrac{r_1}4+\tfrac{2\delta}3$, $r_4=\tfrac{r_1}4+\tfrac\delta3$
(so $r_3+r_4=s+\delta$ and $r_4<r_3<r_2<r_1$ strictly), with any
$0<\delta<\min(\tfrac{r_1}4,\ \tfrac{\varepsilon r_1}2)$: the witness
fills the pan with $\{r_1,r_2\}$ and
the hole of $r_1$ with the pair ($r_3+r_4=s+\delta\le r_1-w$), while
best fit nests $r_2$ ($r_2=s\le r_1-w$ since $w\le r_1/2$), sends $r_3$
to the pan, and blocks $r_4$ additively ($r_1+r_3+r_4>R$ since
$r_3+r_4>s$; $r_2+r_4>r_1-w$ since $w>r_1/4$; $r_4>r_2-w$ likewise).
The three tail ratios are $(2s+\delta)/r_1=1+\delta/r_1$,
$(s+\delta)/s=1+2\delta/r_1$ and $r_4/r_3<1$, so
$\rho=1+2\delta/r_1\le1+\varepsilon$; verified instances reach
$\rho=1.03$ and $1.02$ (script \texttt{umbral.py}).
\end{theorem}

The rigid four-ring counterexample family has its floor at the Tribonacci
constant. An earlier version of this statement relied on a tangency
idealization ($r_3\to r_2$, $w\to0$); the following theorem removes it
entirely: the rigid configuration is not assumed but \emph{emerges} as the
extremal case of a concavity argument.

\begin{theorem}[Rigid-subfamily Tribonacci floor, no idealization]
\label{thm:rigidfloor}
Normalize $r_1=1$ and write $t=r_2$, $p=r_3$, $q=r_4$, $\omega=w>0$.
Let $\mathcal F$ be the family of four-ring instances with
\textup{(F1)} diametral pan $R=1+t$,
\textup{(F2)} $p+q\le 1-\omega$ (the witness pair fits the hole of $r_1$), and
\textup{(F3)} the triple $\{1,p,q\}$ does not pack in the disk of radius
$1+t$. Then every instance of $\mathcal F$ satisfies $\rho>T$, strictly, for
every width $\omega>0$ and all strict radii $q<p<t<1$; moreover
$\inf_{\mathcal F}\rho=T$ and the infimum is not attained. The family is
nonempty only for $t<t^{*}=1/T$, and it contains the $n=4$ counterexample and
the twin $I_1$ of Theorem~\ref{thm:twins}.
\end{theorem}

\begin{proof}[Proof sketch (complete proof in Appendix~\ref{app:rigidproof})]
For two circles of radii $a,b$ internally tangent to the wall of a disk of
radius $R$, the minimal angular separation $\theta(a,b)$ obeys the half-angle
identity $\sin^2\bigl(\theta/2\bigr)=f(a)f(b)$ with $f(x)=x/(R-x)$.
Placing the triple $\{1,p,q\}$ wall-tangent with the unit ring in the middle
packs it whenever $\theta(1,p)+\theta(1,q)+\theta(p,q)\le2\pi$ (the two
non-adjacent arguments are separated by monotonicity of $\theta$), and this
angular condition reduces \emph{algebraically} to
\[
\psi(p)+\psi(q)\;\ge\;\tau_t,\qquad
\psi(x)=\sqrt{\tfrac{t-x}{x}},\qquad \tau_t=\tfrac{t}{\sqrt{1+t}} ,
\]
(an algebraic \emph{sufficient} condition for packing, which is all the
contrapositive needs: (F3) forces $\psi(p)+\psi(q)<\tau_t$). In the
coordinates $\alpha=\psi(p)$, $\beta=\psi(q)$ the sum $p+q=U(\alpha)+U(\beta)$
with $U(z)=t/(1+z^2)$, and $U$ is concave on $[0,\tau_t]$ exactly when
$\tau_t\le1/\sqrt3$, i.e.\ $t\le(1+\sqrt{13})/6=0.7676\dots$, which covers the
relevant range with room to spare; hence the infimum of $p+q$ over the
blocked (open) region sits at the \emph{corner} of its closure, $p=t$,
$q=b(t)$, where $b(t)=t(1+t)/(1+t+t^2)$ is the Descartes pocket of the rigid
pair. This gives the two pressures $t+b(t)<p+q\le1-\omega<1$ (the first one
strict, whence the strict floor and the non-attainment), whose compatibility
is precisely the cubic $t^3+t^2+t<1$, and then
$\rho\ge(p+q)/t>1+(1+t)/(1+t+t^2)=:L(t)$ with $L$ strictly decreasing and
$L(t^{*})=T$. Exactness of the pocket at the rigid pair (needed only for the
direction $\inf\le T$) is a rigidity computation: the constraint left-hand
side factors exactly as $4(t^2+t+1)\bigl(b(t)-q\bigr)$; and the approximating
family lives genuinely inside $\mathcal F$ by a closure-and-monotonicity
lemma for the packing region (compactness of witnesses), which shows the
blocked interval in $p$ is half-open with an attained endpoint.
\end{proof}

\begin{remark}[Slack pans]\label{rem:slack}
The hypothesis $R=r_1+r_2$ can be relaxed to $R\ge r_1+r_2$ at no cost:
non-packability is antitone in the container (a packing in a smaller disk is
a packing in a larger one after translation), so (F3) at radius $R$ implies
(F3) at radius $r_1+r_2$ and Theorem~\ref{thm:rigidfloor} applies verbatim;
the infimum remains $T$. Slack in the pan never helps the blocker --- what it
changes is the set of possible \emph{occupants}, the subject of
Section~\ref{sec:generic}.
\end{remark}

Define the \emph{geometric threshold} as
\[
\tau\;:=\;\inf\bigl\{\rho(I)\ :\ \text{some greedy execution fails to
place the lex-max set on }I\bigr\},
\]
so that placement obliviousness holds for every instance with
$\rho<\tau$ and fails at values of $\rho$ arbitrarily close to $\tau$
(possibly at $\tau$ itself, if the infimum is attained); ``floor''
always refers to an infimum of $\rho$ over a family of blocked
exchanges. The
natural conjecture -- which was tenable until late in this
project -- is that $\tau=T$. It is \emph{false}. The witness capacity that lifts the floor from $\varphi$ to
$T$ in Theorem~\ref{thm:rigidfloor} exists only when the greedy's
container for the pivotal ring is a \emph{hole}; when it is the pan
itself, the floor drops to the first rung of the ladder, the Descartes
pocket, and the golden ratio appears:

\begin{theorem}[The golden family: the threshold is at most $\varphi$]
\label{thm:golden}
For every $\omega\in(1-\varphi/2,\ \varphi-1)=(0.1910,0.6180)$ and every
\[
0<\varepsilon<\varepsilon^\ast:=
\min\Bigl(\tfrac{T-\varphi}{3},\ \tfrac{\varphi/2-\omega}{2},\
\tfrac{4\sqrt5-8-\varphi/2}{2}\Bigr)
\]
(all three bounds positive on the window; the middle one enforces
$s_1\le\varphi-\omega$), the instance
\[
R=\varphi+1,\qquad w=\omega,\qquad
r=\bigl\{\varphi,\ 1,\ \tfrac\varphi2+2\varepsilon,\
\tfrac\varphi2+\varepsilon\bigr\}
\]
has $\rho=\varphi+3\varepsilon<T$, its lex-max set is all four rings, and
the descending greedy under worst fit places only three: placement
obliviousness fails at $\rho=\varphi+3\varepsilon$.
\end{theorem}
\begin{proof}
Write $s_1=\tfrac\varphi2+2\varepsilon$, $s_2=\tfrac\varphi2+\varepsilon$.
Every feasibility below is exact. \emph{The value of $\rho$}: the tails
are $(1+s_1+s_2)/\varphi=\varphi+3\varepsilon/\varphi$ (by the golden
identity $(1+\varphi)/\varphi=\varphi$), $s_1+s_2=\varphi+3\varepsilon$
(dominant) and $s_2/s_1<1$, and $\varphi+3\varepsilon<T$ for
$\varepsilon<(T-\varphi)/3$. \emph{The witness}: place $1$ in the hole of
$\varphi$ (capacity $\varphi-\omega\ge1\iff\omega\le\varphi-1$) and
$\{\varphi,s_1,s_2\}$ as a corona in the pan; the corona exists by an
\emph{exact} certificate. By monotonicity of $\theta$ it suffices that
the symmetric trio $\{\varphi,\sigma,\sigma\}$ with $\sigma=s_1$ has
angular sum below $2\pi$, i.e.\ $2A+B<\pi$ with
$A=\theta(\varphi,\sigma)/2$, $B=\theta(\sigma,\sigma)/2$. Here
$f(\varphi)=\varphi$, and when $2A>\pi/2$ (otherwise the claim is
immediate from $B<\pi/2$),
\[
2A+B<\pi\iff\sin B<\sin2A
\iff f(\sigma)^2<4\varphi f(\sigma)\bigl(1-\varphi f(\sigma)\bigr)
\iff f(\sigma)\,(1+4\varphi^2)<4\varphi ,
\]
which with $f(\sigma)=\sigma/(\varphi^2-\sigma)$,
$8\varphi+5=(2+\sqrt5)^2$ and $\varphi^3=2+\sqrt5$ reduces to
$\sigma(8\varphi+5)<4\varphi^3$, i.e.\
$\sigma<4(\sqrt5-2)=0.9442\dots$; this holds since
$s_1=\varphi/2+2\varepsilon<4(\sqrt5-2)$ by the third bound in
$\varepsilon^\ast$. (At $\varepsilon=0$ the certificate is the exact
identity $4\varphi(\sqrt5-\varphi)=4>1$.) So all four rings are
placeable.
\emph{The greedy under worst fit}: $\varphi\to$ pan; the ring $1$ fits
both the pan ($\varphi+1\le R$, exact tangency, as in
Theorem~\ref{thm:n4}) and the hole of $\varphi$, and worst fit -- the
container of largest capacity, $R=\varphi+1>\varphi-\omega$ -- chooses
the \emph{pan}. Now $\{\varphi,1\}$ fills the pan diametrally: the pair
is rigid and, by the exact necessity of Proposition~\ref{prop:S5}
(rescaled), any third circle has radius at most
$b_2(\varphi,1)=\varphi/2$ -- the golden pocket identity. Hence neither
$s_1$ nor $s_2$ fits the pan; the hole of $\varphi$ admits one of them
but not both ($s_1+s_2=\varphi+3\varepsilon>\varphi-\omega$, two-circle
exact); the hole of the ring $1$ admits neither
($s_2>1-\omega\iff\omega>1-\tfrac\varphi2-\varepsilon$); and
$s_2$ cannot nest in $s_1$ (nesting would require
$\omega\le s_1-s_2=\varepsilon$, while
$\omega>1-\tfrac\varphi2>\varepsilon^\ast$); and $s_1$ fits alone in the hole of
$\varphi$ ($s_1\le\varphi-\omega$, via $\varepsilon^\ast$), which is
where the greedy puts it before $s_2$ is stranded. An exhaustive
placement tree confirms the count: three rings with $1$ in the pan,
four with $1$ in the hole (\texttt{code/aureo.py}).
\end{proof}

\begin{corollary}[The family realizes every $\rho\in(\varphi,T)$]
\label{cor:goldencover}
For $\rho_0\in(\varphi,T)$ fix the width $\omega=0.3$ (any value in
the window works) and take $s_2=\varphi/2+\varepsilon$, $s_1=s_2+\eta$
with $2\varepsilon+\eta=\rho_0-\varphi$ and $0<\eta<\rho_0-\varphi$
small. The dominant tail is $s_1+s_2=\rho_0$ exactly. All walls hold
explicitly: $s_1<1$ and the corona certificate $s_1<4(\sqrt5-2)$, since
$s_1=\varphi/2+\tfrac{\rho_0-\varphi}{2}+\tfrac\eta2
<\varphi/2+\tfrac{T-\varphi}{2}+\tfrac\eta2=0.9196\ldots+\tfrac\eta2
<0.9443$ for $\eta<0.049$; $s_2>\varphi/2=0.809>0.7=1-\omega$ ($H_m$
blocked); $s_1+s_2=\rho_0>\varphi>\varphi-\omega$ (the hole of
$\varphi$ rejects the pair); $s_2>s_1-\omega$ (their gap is
$\eta<0.3=\omega$, so no nesting); and $s_1\le\varphi-\omega=1.318$
($s_1$ fits alone in the hole, where the greedy strands it). \qed
\end{corollary}

The instance is golden in four distinct ways: the occupant is
$\varphi$ (the unique positive fixed point of $2b(A)\,A=1+2b(A)$), the
pocket is $b_2(\varphi,1)=\varphi/2$, the tail of $\varphi$ is
$(1+\varphi)/\varphi=\varphi$, and $\rho\to\varphi$. The family is not
degenerate: the $\omega$-window has width $0.427$, $\varepsilon$ ranges
over an interval, and the diametral tangency survives pan slack up to
$\delta^\ast=0.0248$ (numerically, via the angular criterion). It
does not appear in structured searches over generic radii because it lives at a
codimension-one corner (pan tangency) that random radii never visit.

\begin{theorem}[Exact threshold for four rings]\label{thm:fourfloor}
Let $\tau_4$ denote the infimum defining $\tau$ restricted to inventories
of at most four rings in a disk. Every failing execution on such an
inventory has $\rho>\varphi$. Consequently $\tau_4=\varphi$, and this
infimum is not attained at positive width and strictly decreasing radii.
\end{theorem}
\begin{proof}
We first isolate the combinatorial reduction. Consider the first index
where a greedy execution differs from the lex-max set. It skips a ring
that is feasible together with its previously admitted set. By
Proposition~\ref{prop:n3}, there must be at least three previously
admitted rings: otherwise the same execution on the subinventory of
these rings and the skipped ring would contradict that proposition.
Previously omitted rings leave no occupants or containers in $F$,
so the restricted execution is legal and its full inventory is feasible.
Thus all four rings are feasible, and the execution admits the first
three and rejects the fourth. Normalize the second radius to $1$,
and write the radii as $A>1>p>q>0$, with scaled width $w>0$.
Let $F$ be the greedy assignment of $A,1,p$ and $P$ a full witness.

If $F$ and a full witness agree on the parent of $1$, they can also
be made to agree on the parent of $p$: use the exchange step of
Theorem~\ref{thm:oblivious} at $p$. The only smaller ring is $q<p$,
so any displaced child fits in the vacated outer ball of $p$;
if $q$ is a child of $p$, carry it with $p$. The larger parents and
their assignments stay fixed, and the destination was certified
when the greedy placed $p$. The resulting witness admits $q$ in
$F$, a contradiction. Therefore every full witness places $1$
in the opposite parent from $F$: the only choices are the root and
the hole of $A$.

Apply the exchange step at $1$, and let $S$ be the immediate
children smaller than $1$ in its greedy destination, as assigned by
$P$. If $S$ has at most one member, its subtree fits in the vacated
unit ball. If $S=\{p,q\}$ and $p+q\le1$, the row lemma does the
same. In either operation $1$ travels to its destination with its
own subtree. This gives a full witness agreeing with $F$ on
$1$, which was just ruled out. Hence $S=\{p,q\}$ and $p+q>1$.
There are now exactly two cases.

\emph{The greedy nests $1$ in $A$.} The witness has root pair
$\{A,1\}$ and hole pair $\{p,q\}$ in $A$, so
$R\ge A+1$ and $p+q\le A-w$. If the root triple $\{A,p,q\}$
were packable as root children, nesting $1$ in $A$ (legal since
$F$ certified $1\le A-w$) would give a full witness with
the required parent of $1$. Thus that triple is infeasible.
Divide all lengths by $A$: $t=1/A$, $\omega=w/A$ and
$R/A\ge1+t$, while $\rho$ is unchanged. These are the hypotheses of
Theorem~\ref{thm:rigidfloor}, with pan slack allowed by
Remark~\ref{rem:slack}. Therefore $\rho>T>\varphi$.

\emph{The greedy places $1$ at the root.} Now $P$ nests $1$ in
$A$, and $p,q$ are root children. In particular $A-w\ge1$.
In $F$, the ring $p$ must lie in the hole of $A$: if it were at the
root or inside $1$, the empty hole of $A$ would admit $q<1$.
Since the root pair $\{A,1\}$ is feasible, $R\ge A+1$. The exact
diametral pocket of this pair is
\[
 g(A)=\frac{A(A+1)}{A^2+A+1}.
\]
The sufficiency direction of Proposition~\ref{prop:S5}, rescaled,
supplies a packing of
$\{A,1,q\}$ at radius $A+1$ whenever $q\le g(A)$; such a packing
also works in the actual pan and is compatible with $p$ inside $A$.
Since $q$ is rejected, $q>g(A)$. Consequently
\[
 \rho\ge p+q>2g(A),\qquad
 \rho A\ge1+p+q>1+2g(A).
\]
Put $D=A^2+A+1>0$. The identity $\varphi^2=\varphi+1$ gives
\begin{align*}
 D\bigl(2g(A)-\varphi\bigr)
  &=(A-\varphi)\bigl((2-\varphi)A+1\bigr),\\
 D\bigl(1+2g(A)-\varphi A\bigr)
  &=(\varphi-A)\bigl(\varphi A^2+(2\varphi-2)A+\varphi-1\bigr).
\end{align*}
For $A\ge\varphi$, the first identity implies $2g(A)\ge\varphi$;
for $1<A\le\varphi$, the second implies $1+2g(A)\ge\varphi A$.
In both cases the strict inequalities above yield $\rho>\varphi$.
Finally Theorem~\ref{thm:golden} supplies four-ring failures with
$\rho\downarrow\varphi$, establishing the infimum and non-attainment.
\end{proof}

The two algebraic identities and the ordered-ring implication used in
the second case are kernel-checked in \texttt{Calamares/FourRing.lean}.
The forest reduction and the Euclidean pocket remain written proofs.
The preceding four-ring proof does not by itself cover arbitrary
inventories. The following uniform argument closes that gap.

\clearpage
\input{golden_global}

\section{The width program: profile thresholds and the \texorpdfstring{$13/7$}{13/7} corner}
\label{sec:width}

For quantitative bounds within restricted exchange templates, the proof of
Theorem~\ref{thm:oblivious} localizes all difficulty in a single
\emph{exchange step}: $m$ is the largest ring that the greedy $F$ and a
witness $P$ place in different containers (so $F$ and $P$ agree
on the placement of every ring larger than $m$) $u=c_F(m)$, $v=c_P(m)$, and the set
$S$ of rings smaller than $m$ that $P$ keeps in $u$ must be reinserted using
the resources freed by moving $m$ to $u$: the vacated disk $D_m$ (radius
$r_m$), the hole $H_m$ (capacity $r_m-w$, travelling with $m$), nesting, and
the geometry of $v$. Normalize $r_m=1$ and $\omega=w/r_m$.
\begin{figure}[ht]
\centering
\begin{tikzpicture}[scale=0.9, every node/.style={font=\small}]
\draw[thick] (0,0) circle (2.0);
\node at (0,2.35) {$v=c_P(m)$};
\draw[dashed] (-0.55,-0.25) circle (1.0);
\node at (-0.55,-0.25) {$D_m$};
\node[align=center] at (0,-2.45) {\footnotesize vacated disk, radius $r_m$};
\draw[thick] (6.5,0) circle (2.0);
\node at (6.5,2.35) {$u=c_F(m)$};
\draw[thick] (6.0,0.35) circle (1.0);
\draw[thick] (6.0,0.35) circle (0.62);
\node at (6.0,1.15) {\footnotesize $m$};
\node at (6.0,0.35) {\footnotesize $H_m$};
\draw[dashed] (7.35,-0.9) circle (0.42); \node at (7.35,-0.9) {\footnotesize $\sigma_1$};
\draw[dashed] (6.35,-1.35) circle (0.3); \node at (6.35,-1.35) {\footnotesize $\sigma_2$};
\draw[->, dashed] (6.6,-2.05) to[bend right=25] node[below, pos=0.6]{\footnotesize reinsertion} (0.3,-1.7);
\node at (6.5,-2.45) {\footnotesize $S$: displaced from $u$};
\draw[->, very thick] (1.6,0.9) to[bend left=18] node[above]{$m$ moves} (4.9,0.9);
\end{tikzpicture}
\caption{The exchange step of Theorem~\ref{thm:oblivious}: $m$ is the
largest ring placed differently by the greedy $F$ and a witness $P$.
Moving $m$ to $u$ frees the vacated disk $D_m$ (radius $r_m$) in $v$
and brings along the hole $H_m$ (capacity $r_m-w$); the set $S$ of
smaller rings that $P$ keeps in $u$ must be reinserted using $D_m$,
$H_m$, nesting, and the geometry of $v$.}
\label{fig:exchange}
\end{figure}
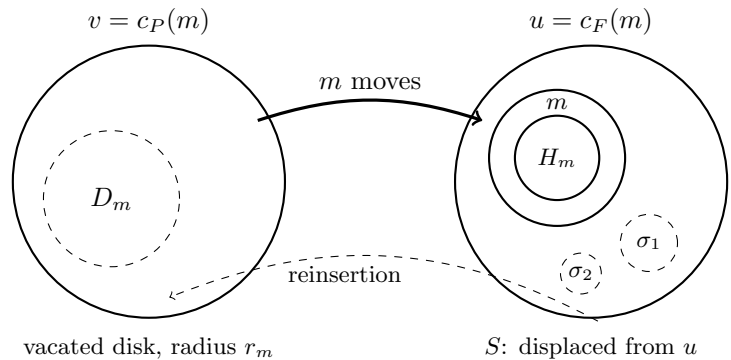

\emph{Width convention}: unless a result explicitly states otherwise,
throughout the exchange-step program (this section, the next, and
Appendices~\ref{app:widthproofs} and~\ref{app:genericproofs}) we assume
$0<\omega<1$, so that the pivot $m$ is a genuine ring with a nonempty
hole $H_m$. The model does admit $\omega\ge1$ (a solid-disk pivot); that
regime is outside this specialized program except in results that state ``every $\omega>0$''
explicitly: Theorem~\ref{thm:corner}, Corollary~\ref{cor:Cmin}'s
Tribonacci closure, Corollary~\ref{cor:DB2}, and cases (i)--(ii) of
Theorem~\ref{thm:DP}. Their proofs route the $\omega\ge1$ regime
through walls that do not use $H_m$ (the trio frontier and the witness
capacity) and so remain valid for a solid pivot. Everything $\rho$
imposes on $S$ is the pair of inequalities $\sum S\le\rho$ and
$\sum_{l>j}x_l\le\rho\,x_j$; a lower bound for $\rho$ over all
\emph{blocked} exchanges is therefore a lower bound for any failure of
placement obliviousness. We call a \emph{wall} any such necessary
condition for blocking: the contrapositive of an explicit unblocking
placement, which every blocked exchange must therefore satisfy. Two
regimes must be distinguished. When $u$ is
the \emph{hole} of a ring (the \emph{nested} exchange), the witness
placement supplies the capacity wall $\sum S\le\mathrm{cap}(u)$, and this
section and the next prove floors that all exceed $T$ for a hierarchy of
nested templates: the Tribonacci constant is the exact floor of the
rigid family and the conjectured universal floor of nested exchanges.
When $u$ is the pan, that wall is absent, the floor drops to the Descartes pocket, and
the golden family of Theorem~\ref{thm:golden} realizes it: the
\emph{pan} exchange governs the global threshold of
Theorem~\ref{thm:golden-global}. Complete proofs for every statement of this
section are given in Appendix~\ref{app:widthproofs}.

\paragraph{The combinatorial layer is closed.}
For two-ring profiles the blocked threshold (a \emph{floor} in the
terminology of Section~\ref{sec:threshold}: an infimum over blocked
exchanges) is exactly
$\rho^*_2(\omega)=\max\bigl(1,\,2(1-\omega)\bigr)$; in particular no
two-ring blocking is compatible with $\rho<T$ for
$\omega\le1-T/2=0.0804$. For three rings the threshold is closed,
\[
\rho^*_3(\omega)=\max\Bigl(1,\ \min\bigl(2(1-\omega),\
\max\bigl(\varphi,\ \tfrac{2}{1+2\omega}\bigr)\bigr)\Bigr),
\]
with a purely additive proof and a \emph{golden plateau}
$\rho^*_3\equiv\varphi$ on
$\omega\in\bigl[\tfrac{\sqrt5-2}{2},\,1-\tfrac{\varphi}{2}\bigr]$; and a
general additive tree argument shows $\rho^*_k=\rho^*_3$ for \emph{every}
$k\ge3$. Consequently the combinatorial threshold crosses $T$ exactly at
\[
\omega_T \;=\; \tfrac1T-\tfrac12 \;=\; T^2-T-\tfrac32 \;=\;0.043689\dots :
\]
below $\omega_T$ the exchange step never blocks with $\rho<T$, with no
geometry at all
(Propositions~\ref{prop:Cpair}--\ref{prop:Callk},
Theorem~\ref{thm:Ctrio}, Corollary~\ref{cor:Comegac}).

\paragraph{The canonical template with positive width.}
For $\omega>\omega_T$ geometry enters through the canonical template
($v$ a diametral pan with one big neighbor $\alpha$; $u$ the hole of
$\alpha$, capacity $\alpha-\omega$; $S=\{\sigma_1\ge\sigma_2\}$ placed there
by the witness). Blocking must defeat four resources, giving the walls
$\sigma_2>1-\omega$ ($H_m$), $\sigma_2>\alpha-\omega-1$ ($\sigma_2$ next to
$m$ inside $u$), $\sigma_1+\sigma_2\le\alpha-\omega$ (the witness), and
non-packability of the triple $\{\alpha,\sigma_1,\sigma_2\}$ in the pan
(Figure~\ref{fig:canonical}).
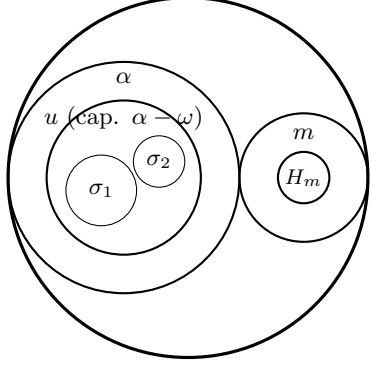
\begin{figure}[ht]
\centering
\begin{tikzpicture}[scale=0.85, every node/.style={font=\small}]
\draw[very thick] (0,0) circle (2.8);
\node at (0,3.15) {the pan ($v$): diametral pair $\{\alpha,m\}$};
\draw[thick] (-1,0) circle (1.8);
\node at (-1,1.55) {\footnotesize $\alpha$};
\draw[thick] (-1,0) circle (1.2);
\node at (-1,0.93) {\footnotesize $u$ (cap.$\,$ $\alpha-\omega$)};
\draw (-1.35,-0.2) circle (0.55); \node at (-1.35,-0.2) {\footnotesize $\sigma_1$};
\draw (-0.45,0.25) circle (0.4); \node at (-0.45,0.25) {\footnotesize $\sigma_2$};
\draw[thick] (1.8,0) circle (1.0);
\draw[thick] (1.8,0) circle (0.4);
\node at (1.8,0.68) {\footnotesize $m$};
\node at (1.8,0) {\scriptsize $H_m$};
\end{tikzpicture}
\caption{The canonical template: $v$ is a diametral pan with one big
neighbor $\alpha$; $u$ is the hole of $\alpha$ (capacity
$\alpha-\omega$), where the witness keeps
$S=\{\sigma_1\ge\sigma_2\}$. Blocking must defeat four resources,
giving the walls $\sigma_2>1-\omega$ (the hole $H_m$),
$\sigma_2>\alpha-\omega-1$ ($\sigma_2$ next to $m$ inside $u$),
$\sigma_1+\sigma_2\le\alpha-\omega$ (the witness), and
non-packability of the triple $\{\alpha,\sigma_1,\sigma_2\}$ in the
pan.}
\label{fig:canonical}
\end{figure}

Two structural facts organize the optimum. First, on the blocking frontier of the
triple one has the closed form $t(\sigma_1)+t(\sigma_2)=t(b(\alpha))$ with
$t(s)=\sqrt{(1-s)/s}$ (the Descartes pocket read in the $t$-coordinate), and
the frontier slope satisfies the $\alpha$-\emph{independent} identity
\[
\kappa \;=\; -\,\frac{\partial\sigma_2}{\partial\sigma_1}
\;=\;\sqrt{\frac{g(\sigma_2)}{g(\sigma_1)}},\qquad g(s)=s^3(1-s),
\qquad \kappa\ge1\ \text{for all }\alpha>1 ,
\]
so the cheapest blocked pair sits at $\sigma_1\to1$, $\sigma_2\to b(\alpha)$;
moreover $1+b(\alpha)\ge\alpha$ iff $\alpha\le T$, so no blocking exists at
all for $\alpha\le T$. Second, the witness capacity deforms the cubic: with
$T_c$ the positive root of $\alpha^3=c(\alpha^2+\alpha+1)$, the
witness branch has infimum
$\Phi(\omega)=T_{1+\omega}-\omega$, strictly increasing and concave with
$\Phi'(\omega)=(2\alpha+1)/\bigl(\alpha^2(\alpha^2+2\alpha+3)\bigr)$ at
$\alpha=T_{1+\omega}$. The full curve of the blocked infimum
$T_{\mathrm{can}}(\omega)$ is obtained in closed form \emph{as a proved
lower bound} --- the
matching upper bound is exact on the witness branch and at the corner below,
and inherits the angular-criterion caveat (the criterion is proved
sufficient, not necessary, so the matching upper bound is conditional
there) on the other two branches:
$2(1-\omega)$ on $(0,\omega_1]$, a sextic-algebraic middle branch on
$[\omega_1,1/7]$, and $\Phi$ on $[1/7,0.30]$, where
$\omega_1=0.0413570\dots$ is the root of $4\omega^3-20\omega^2+25\omega-1$
in $(1/25,1/14)$
(Lemma~\ref{lem:CF}, Theorem~\ref{thm:Ckappa}, Corollary~\ref{cor:Cmin},
Propositions~\ref{prop:Cwitness} and~\ref{prop:Ccurve}).

\begin{theorem}[The $13/7$ corner]\label{thm:corner}
For all $\omega>0$ (including the solid-pivot regime $\omega\ge1$: for
$\omega\ge\tfrac17$ the proof below uses only (B1) and (W), which do
not use $H_m$ and remain valid for a solid pivot),
\[
T_{\mathrm{can}}(\omega)\;\ge\;\tfrac{13}{7}\;=\;T+0.017856\dots,
\]
with equality exactly in the limit at the rational corner
$(\omega,\alpha,\sigma_2)=\bigl(\tfrac17,\,2,\,\tfrac67\bigr)$, where all
four walls saturate simultaneously (note $T_{8/7}=2$ because $2^3=8=
\tfrac87\cdot7$). The bound is unconditional -- in particular it does not
depend on the exactness of any angular feasibility criterion -- and the
approximating family is genuine. Thus positive width only \emph{raises} the
Tribonacci floor, uniformly.
\end{theorem}

The fine structure is curious: the curve is \emph{not} monotone on
$(0,1/7]$ -- $\omega_1$ is a local minimum, followed by a bump of height
$1.1\cdot10^{-4}$ at $\omega_{\mathrm{peak}}\approx0.0445$ (a root of an
explicit degree-8 polynomial) before descending to the corner.

\section{Generic containers: coronas, walls, and metallic floors}
\label{sec:generic}

The canonical template assumes $v$ contains only the big neighbor and $m$.
Generic occupied containers require additional walls within this
specialized template analysis. The following toolkit quantifies those
obstructions; it is not needed for the universal golden proof. Throughout, $v$ is a pan of radius $R$ with occupants
$\{\alpha\}\cup O\cup\{m\}$, $O=\{o_1\ge\dots\ge o_j\}$, $o_i\ge m=1$,
$u$ is the hole of $\alpha$, $S=\{\sigma_1\ge\sigma_2\}$, and holes may be
occupied arbitrarily, at any nesting depth. Complete proofs for every
statement of this section are given in Appendix~\ref{app:genericproofs}.

\paragraph{The universal trio frontier.}
For a trio $\{A,x,y\}$ wall-tangent in a disk of \emph{arbitrary} radius $R$
(interior domain, and under the necessary hypothesis $A\ge\min(x,y)$), the
angular blocking condition is \emph{linear} in the coordinates
$T_c(x)=\sqrt{(c-x)/x}$, $c=R-A$:
$\ \theta$-infeasibility $\iff T_c(x)+T_c(y)\le\tau_R=c/\sqrt{AR}$.
The pocket generalizes to $b_R(A)=ARc/(AR+c^2)$, increasing in $R$, and the
frontier slope identity $\kappa=\sqrt{g_c(y)/g_c(x)}$, $g_c(s)=s^3(c-s)$, is
uniform in $R$. Without the hypothesis $A\ge\min(x,y)$ the equivalence is
false (an explicit pointwise counterexample witnesses the failure; the
failure region is swept by script), a fact
found by adversarial verification (Lemma~\ref{lem:DU}).

\paragraph{An exact corona criterion.}
Call a \emph{corona} a packing of $k$ circles all wall-tangent. For a fixed
cyclic order the corona constraints form a linear feasibility system in the
angular gaps, so coronas are decidable exactly; every subset of a corona
yields a necessary \emph{certificate} (its induced cyclic sum of pairwise
angles is at most $2\pi$). For $k=4$ the certificates are exactly sufficient
-- for arbitrary symmetric separations in $(0,\pi]$, by a negative-cycle
argument: a violating cycle with $U$ ascents needs more than $2U$ edges, so
on four nodes only $U=1$ cycles can violate, and those are exactly the
subset certificates -- and the global criterion collapses to \emph{two}
inequalities: the top trio, plus the total of the \emph{zigzag} order
$(a_1,a_3,a_2,a_4)$, which minimizes the cyclic total by a
convexity/majorization argument ($\theta=g(\log f(a)+\log f(b))$ with $g$
convex). Two natural guesses are refuted by explicit counterexamples: the
consecutive-arc condition alone is \emph{not} sufficient for $k\ge4$, and the
sorted order is \emph{not} optimal (the zigzag is: large circles must not be
angular neighbors). For $k=5$ sufficiency requires exactly one additional
certificate, the \emph{pentagram} (the five-diagonal star $\{5/2\}$, with
combinatorially sharp example $\theta=\pi$ on the diagonals), conjecturally
redundant for separations of geometric origin
(Lemmas~\ref{lem:Dgaps}--\ref{lem:Dsubset},
Theorems~\ref{thm:Dk4}--\ref{thm:DU4}).

\paragraph{Occupant walls and metallic floors.}
Each extra occupant's hole is itself a reinsertion resource, and blocking it
is \emph{taxed}. The basic count: with free holes, blocking forces every
extra occupant to within one width of $m$ ($o_i<\sigma_2+\omega$), and the
largest occupant's tail gives
$\rho>(j+2)/(1+\omega)$ -- each occupant pays $1/(1+\omega)$ -- whence, with
a refined branch $\rho>4/(1+2\omega)$ for $\omega\ge\tfrac12$,
$\rho>T$ for $\omega<2/T-\tfrac12=0.5874$ when $j=1$ and for \emph{all}
$\omega$ when $j\ge2$; adding occupants is strictly suboptimal for the
blocker, since the canonical curve satisfies $\Phi(\omega)<2$ identically
(the polynomial identity
$(2+\omega)^3-(1+\omega)\bigl((2+\omega)^2+(2+\omega)+1\bigr)=1$).
For occupied holes the tax is quantified by the \emph{opposite-disk lemma}:
in a disk of capacity $c$ with largest piece $\sigma$, if the remaining mass
is at most $c-\sigma$ the whole family packs (the wall-tangent $\sigma$
leaves a free disk of radius exactly $c-\sigma$ opposite, filled in a row);
hence blocking a hole costs at least its slack: every \emph{node} $y$ --
an extra occupant $o_i$, or any ring of radius $\ge r_m$ nested at any
depth inside one -- satisfies $y<\sigma_2+\omega+X_y$, where $X_y$ is the
total radius of $y$'s hole content
($\alpha$ and $m$ themselves are excluded: $\alpha$'s hole is $u$, governed
by its own walls, and excluding $m$ is essential). Following the
\emph{minimal} such ring and balancing two tails yields, for arbitrary hole
occupancy at any depth,
\[
\rho\;>\;\Psi(\omega)\;=\;(1-\omega)+\sqrt{(1-\omega)^2+1} ,
\]
the metallic mean of $2(1-\omega)$ (we write $\Psi\equiv\Psi_1$
interchangeably: it is the $j=1$ member of the ladder $\Psi_j$ below):
$\Psi(0)=1+\sqrt2$ is the silver ratio,
$\Psi(1/4)=2$ exactly, and $\Psi>T$ iff $\omega<(T-1)^2/2=0.3522$.
Allowing $m$ itself to carry children splits by an evacuation dichotomy into
the same bound and a second metallic mean, the root of
$u^2-(2-\omega)u-1$, which dominates automatically; its own crossing is at
$(T-1)^2$ exactly -- twice the previous one, and the certifying identity
\emph{is} the Tribonacci polynomial:
$(T-1)^2\,T-(2T-T^2+1)=T^3-T^2-T-1$
(Lemmas~\ref{lem:DV1} and~\ref{lem:DR}, Proposition~\ref{prop:DV2},
Theorems~\ref{thm:DB} and~\ref{thm:DBpp}).

\paragraph{The double-pocket wall and the golden line.}
The geometric wall needs no angular criterion at all: blocking implies the
four rings $\{\alpha,o_1,\sigma_1,\sigma_2\}$ do not pack in the pan, hence
not in the disk of radius $\bar R=\alpha+o_1$ (containment), where the pair
$\{\alpha,o_1\}$ is diametrically rigid and leaves exactly \emph{two}
Descartes pockets
$b_2(\alpha,o_1)=\alpha o_1(\alpha+o_1)/(\alpha^2+\alpha o_1+o_1^2)$, one on
each side (their mirror centers are $4b_2$ apart: the exact identity
$y_0=2b_2$). Thus blocking forces $\sigma_1>b_2(\alpha,o_1)$. Feeding this
wall into the program of the previous walls produces a \emph{golden} optimum:
the binding corner is $\alpha=2$, $o_1=\sqrt5-1$ -- the pair with unit pocket,
$b_2(2,\sqrt5-1)=1$ exactly, self-dual in the sense
$A_{\max}(\sqrt5-1)=2$, $A_{\max}(2)=\sqrt5-1$ -- and, for one extra occupant
and \emph{any} hole occupancy, both branches of the analysis meet the same
line: $\rho>T$ unconditionally for all $\omega<\omega_A$, via
\[
\rho\;>\;\varphi^2-\tfrac{\varphi}{2}\,\omega,
\qquad
\omega_A\;=\;2-2(T-1)(\varphi-1)\;=\;0.962585\dots
\]
(the golden line itself is proved in all but one corner
configuration -- the corner with a nested node inside the occupant
and $\alpha<o_1$, called branch~B in the appendix, possible
only for $\omega<\tfrac12$ -- where the proved bound is $\ge2>T$, via
$\Psi_B(\tfrac12)=2$ exactly). The proof combines two exact facts -- $b_2$
is concave in $\alpha$
($\partial^2b_2/\partial\alpha^2=-6\alpha o^3(\alpha+o)/D^3$) and
$D^2-\alpha o^2(o+2\alpha)=(\alpha+o)(\alpha^3+\alpha^2o+o^3)>0$ -- with
one-variable boundary certificates anchored at the identity
$f_1\equiv G_\omega$ at $o_1=\sqrt5-1$, where
$G_\omega:=\varphi^2-(\varphi/2)\omega$ is the golden-line bound. For $j\ge2$ occupants the
ladder $\Psi_j=(1-\omega)+\sqrt{(1-\omega)^2+j}$ (crossing $T$ at
$1-(T^2-j)/(2T)$: $0.624$ for $j=2$, $0.896$ for $j=3$, all $\omega$ for
$j\ge4$) covers the corresponding ranges unconditionally, for arbitrary
nested occupancy of the holes: a \emph{leaf argument} closes the case where
occupants' holes contain further rings of radius $\ge r_m$. Every occupant
subtree contains a leaf (a node none of whose hole contents reach $r_m$);
taking $L$ the largest of the $j$ disjoint leaves, its tail collects the
other $j-1$ leaves together with $m$, $\sigma_1$, $\sigma_2$ and all small
mass $W$, while the blocking wall at a leaf reads $L<\sigma_2+\omega+W$;
optimizing the resulting two-branch program returns exactly $\Psi_j$, with
the crossing at $\sigma_2=1-\omega$ -- the same metallic-mean mechanism as
$\Psi$. Small rings anywhere compatible with each theorem's template
-- and with $S$ still a pair -- are free for all the combinatorial walls,
whose unblocking placements are local to $D_m$, $H_m$, the holes, and the
slot next to $m$ inside $u$
(Lemma~\ref{lem:DG}, Theorem~\ref{thm:DGp},
Proposition~\ref{prop:DPsij}, Corollary~\ref{cor:DS}).

\paragraph{Three-piece profiles: the zigzag pays.}
Larger reinsertion profiles are \emph{cheaper} for the adversary in pure
nest accounting ($\rho^*_3\le\rho^*_2$), but in the canonical template the
geometry reverses the sign: blocking with $S=\{\sigma_1\ge\sigma_2\ge
\sigma_3\}\subset(\omega,1)$ and arbitrary occupancy of $H_m$ forces
$\rho>\max\bigl(\Phi(\omega),\tfrac{13}7\bigr)$ for \emph{all} $\omega$,
with no exceptional branch. If $\sigma_3\le\sigma_1-\omega$ then $\sigma_3$ (the \emph{dust}) rides
inside $\sigma_1$ and the pair program is inherited wall by wall
(the $13/7$ corner bound applies); if nothing nests and the trio
$\{\alpha,\sigma_1,\sigma_2\}$ has no corona, the witness chain gives
$\sigma_1+\sigma_2\ge1+b(\alpha)\ge\Phi(\omega)$; and if the trio
\emph{does} pack, the $U_4$ criterion read backwards forces the \emph{zigzag}
certificate to fail, and the zigzag implies exactly $\sigma_2>b(\alpha)$
-- the Descartes pocket by a third route, via
$\sin^2\!A+\sin^2\!B>1$. Crossing that wall with the witness capacity
lands on the \emph{golden deformation} $\gamma_\omega$, the root of
$\alpha^3=\alpha^2+\alpha+\omega(\alpha^2+\alpha+1)$ with
$\gamma_0=\varphi$, and the tail bounds meet at another rational corner:
$\gamma_{2/7}=2$ exactly ($8=4+2+2$) with value $\tfrac{17}7$ -- the
sister of the $13/7$ corner
(Theorem~\ref{thm:DT3}, Proposition~\ref{prop:DT3j}).

\paragraph{Status.}
Together with Theorem~\ref{thm:corner}, every blocked \emph{nested}
exchange in the templates above satisfies $\rho>T$: for all $\omega$ in
the canonical template -- with $S$ a pair or a triple -- for
$\omega<0.9626$ with one extra occupant under arbitrary hole occupancy,
for $\omega<0.624$ and $\omega<0.896$ with two and three, and for all
$\omega\in(0,1)$ with four or more. All these floors exceed $T>\varphi$. The global golden guarantee is
already established by Theorem~\ref{thm:golden-global}; the following
questions concern the stronger nested-template threshold. Of the
remaining nested gaps, the width tips at $j=1$ ($\omega\ge0.9626$,
where the golden line still gives $\varphi^2-\tfrac\varphi2\omega
\ge\varphi^2-\tfrac\varphi2=1.809>\varphi$) and at $j=3$
($\omega\ge0.896$, where $\Psi_3(\omega)>\Psi_3(1)=\sqrt3>\varphi$)
matter only for pinning the nested threshold at exactly $T$. At the
golden level itself, the remaining nested configurations -- the
$j=2$ tip for $\omega\in[\varphi/2,1)$, the quantitative ``gap
lemma'' regime (small rings against the geometric walls; tiny
$\sigma_2$), and profiles of four or more pieces beyond the
reduction branch -- have written proofs
(Theorems~\ref{thm:nestedwritten} and~\ref{thm:gapwritten}: the
entire nested template, every $j$), with duality-certificate
backing and independent verification, in
Appendix~\ref{app:campaign}. Each statement in this section is
backed by a verification script and a verification report in the
repository (Appendix~\ref{app:verifmap}).

\section{The two objectives diverge: a phase diagram}\label{sec:divergence}

Contact area behaves like $2\pi w\sum r_i$ minus a per-ring penalty
$\pi w^2$; cardinality counts rings. \emph{Divergence} means the area
optimum has strictly smaller cardinality than the cardinality maximum.
Two rings never diverge: if both fit together the full set is optimal
for both objectives (area is strictly increasing under inclusion), and
otherwise every feasible set has at most one ring, a cardinality the
area optimum attains. Three rings can already diverge: at $R=10$, $w=9/2$, radii
$\{8,\,101/20,\,99/20\}$ (rational-exact; script
\texttt{divergencia3}, 5/5), the small pair is exactly diametral
($101/20+99/20=10$), nothing coexists with or nests in the $8$ (its
hole has radius $7/2<99/20$), and the areas compare as
$a(8)=\tfrac{207}{4}\pi>\tfrac{198}{4}\pi=a(101/20)+a(99/20)$: the
area optimum is the single $8$ while cardinality prefers the pair.
The mechanism is pure superadditivity at large width, where all holes are too small to permit nesting
and the problem degenerates to circle packing; sampling places the
onset of three-ring divergence near $w/R\approx0.26$ for this family
(swept, not a proved threshold). In the \emph{nesting-active} regime
the divergence survives only marginally: small rings that fit $k$
times in the pan but at most $k-2$ times in the large ring's hole
(with $k-1$ in the hole the two objectives tie), and the minimal
instance of this hole mechanism has four rings:
$R=10$, $w=1$, radii $\{9.0,\,4.2,\,4.2,\,4.2\}$ --- the area optimum
is the $9.0$ with one nested $4.2$ ($N=2$, $A\approx76.7$), the
cardinality optimum the three $4.2$'s ($N=3$, $A\approx69.7$).
Figure~\ref{fig:phase} maps the divergence region for the family
``one large ring $b$ plus unlimited equal small rings of radius
$s$'' \emph{at width $w=1$, $R=10$} (the regime of the hole
mechanism; script \texttt{franja}, which generates the figure): a
staircase governed by the proved optimal thresholds
for $n$ equal circles in a disk~\cite{GLNO1998,Pirl1969,Melissen1994,%
Fodor1999,EkanayakeLaFountain2024}; its upper edge is exactly the three-circle
threshold $0.4641R$.

\begin{figure}[ht]
\centering
\includegraphics[width=0.90\textwidth]{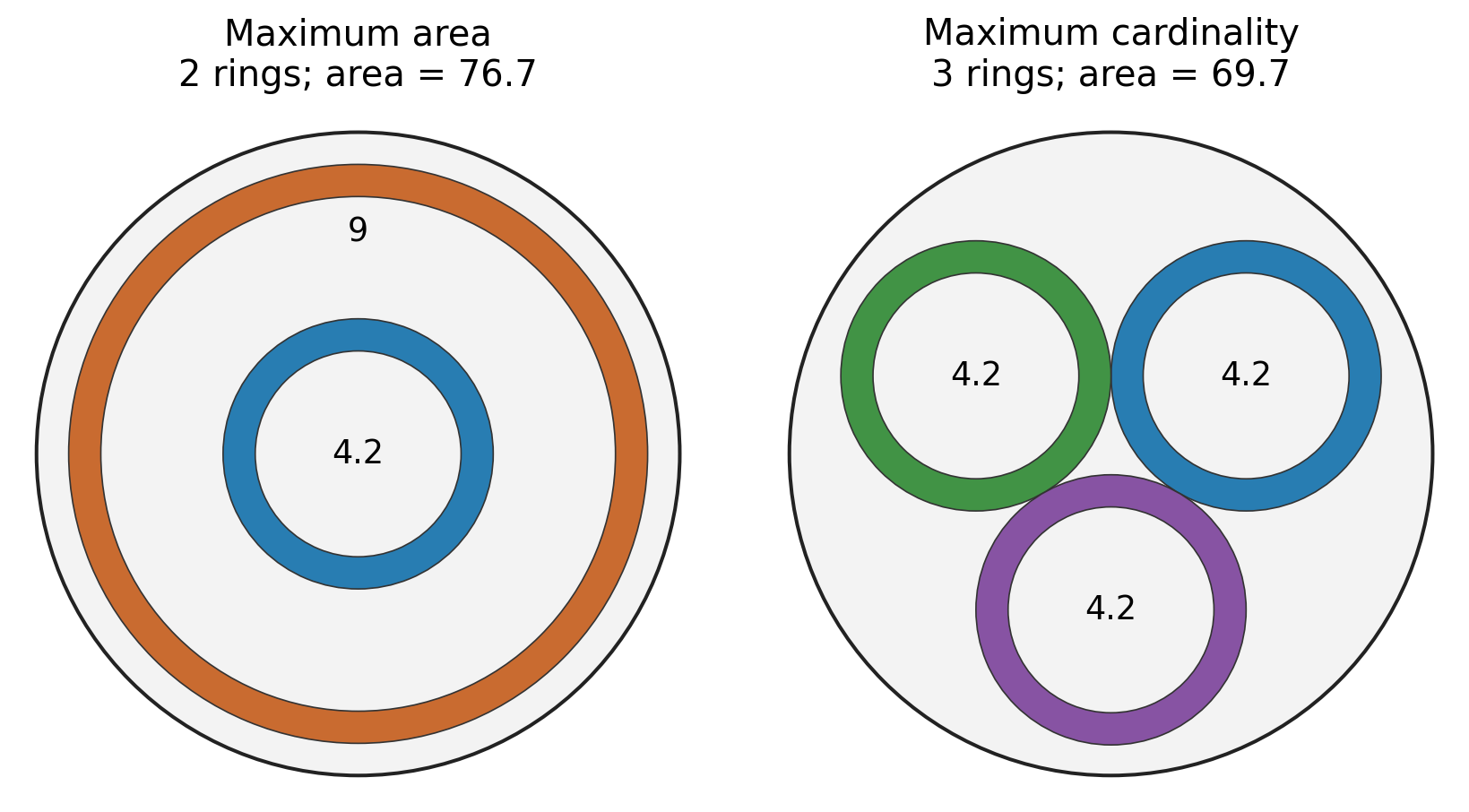}\\[1ex]
\includegraphics[width=0.85\textwidth]{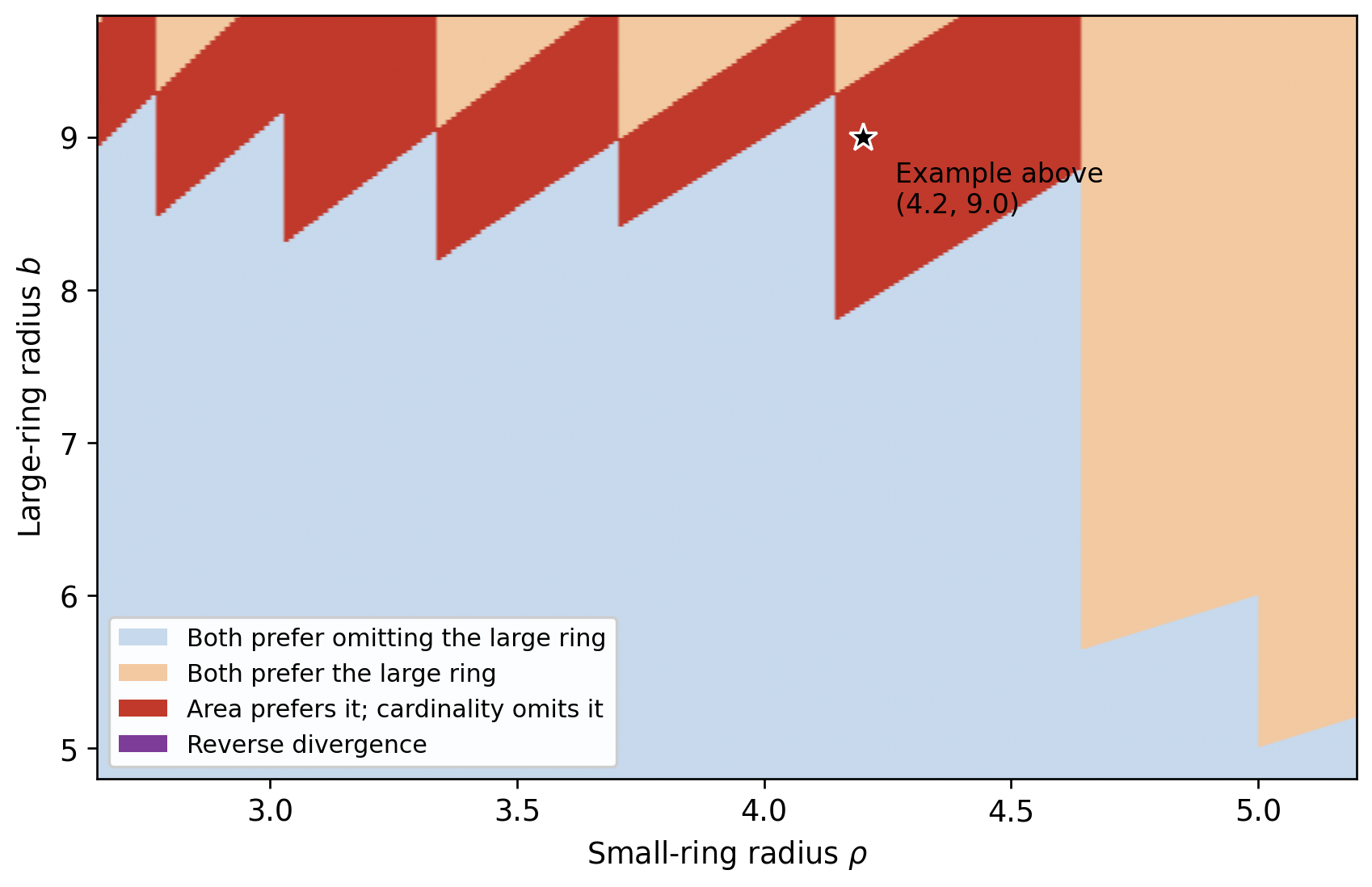}
\caption{Top: the minimal divergence instance; area optimum (left) versus
cardinality optimum (right). Bottom: phase diagram of the divergence band for
uniform width.}
\label{fig:phase}
\end{figure}

\FloatBarrier
\section{Hardness decouples into two layers}\label{sec:hardness}

If all radii lie within a band of width less than $w$, no nesting is
possible and the decision problem ``do all rings fit?'' is exactly circle
packing of the outer radii into the pan.

\paragraph{The geometric layer.}
Since Theorems~\ref{thm:selection} and~\ref{thm:oblivious} hold for
arbitrary containers, our model contains, as the special case of a
\emph{square} pan with a no-nesting band, the problem of deciding whether
given circles fit into a square, which is NP-hard by reduction from
3-\textsc{Partition}~\cite{DemaineFeketeLang2010}; hence the geometric layer
of the model is NP-hard. Membership in NP is itself delicate, because tight
packings may require coordinates without short descriptions -- a phenomenon
made precise by the $\exists\mathbb R$-completeness framework for
two-dimensional packing~\cite{AbrahamsenMiltzowSeiferth2020}. For the disk
pan specifically, NP-hardness of deciding whether given disks fit into a
circular container is established in the journal version of
Fekete--Keldenich--Scheffer~\cite[Theorem~5.1]{FeketeKeldenichScheffer2019}
(the conference version does not contain this result -- we thank the
referee for the precise pointer). Integer-programming
formulations for circle packing in rectangular containers are studied
in~\cite{Litvinchev2014}.

\paragraph{The combinatorial layer.}
Independently of the geometry, the \emph{additive surrogate}
(Theorem~\ref{thm:additive}) is weakly NP-hard, and here the reduction can
be made exact despite the per-ring penalty in the area.

\begin{proposition}[The additive layer is weakly NP-hard]
\label{prop:additivehard}
Deciding whether an instance of the additive surrogate admits a feasible
set of contact area at least a given threshold is weakly NP-hard, already
for instances with a no-nesting radius band.
\end{proposition}
\begin{proof}
Reduce from \textsc{Equal-Cardinality Partition}: given positive integers
$a_1,\dots,a_n$, $n$ even, with $\sum_i a_i=2B$, decide whether some $S$
with $|S|=n/2$ has $\sum_{S}a_i=B$. (This variant is NP-hard: from a
\textsc{Partition} instance $b_1,\dots,b_m$ with $\sum_i b_i=2B_0$, take
any integer $K\ge1$ and the $2m$ numbers $b_1+K,\dots,b_m+K,K,\dots,K$ ($m$
copies of $K$); an $m$-subset with $t$ items of the first kind has sum
$mK+\sum_{\mathrm{chosen}}b_i$, so $m$-subsets of sum $mK+B_0$ correspond
exactly to subsets of the original instance of sum $B_0$, of any size,
padded to cardinality $m$ with $K$-items.) Fix any width $w>0$, set $M=3B$, pick
$\delta\in\bigl(\tfrac{w}{3B},\tfrac{w}{2B}\bigr)$, and take radii
$r_i=\delta(a_i+M)$ with pan capacity $R^{+}=\delta\bigl(\tfrac n2M+B\bigr)$.
Dividing the area inequality by $\pi$ and choosing $\delta$ rational
in the interval makes every datum of the instance rational, so the
reduction is fully formal.
All radii lie in $[3B\delta,\,5B\delta]$, so $r_i>w$ (each is a genuine
ring, $a(r)=2\pi wr-\pi w^2$) and $r_{\max}-r_{\min}=2B\delta<w$ (no
nesting): feasibility is exactly $\sum_S r_i\le R^{+}$. Write $s=|S|$ and
$t=\sum_S a_i\le 2B$; the area is
$A(S)=2\pi w\delta(sM+t)-\pi w^2 s$, and we ask whether
$A(S)\ge\Theta:=2\pi w\delta\bigl(\tfrac n2M+B\bigr)-\pi w^2\tfrac n2$.
If a balanced partition exists, taking $S$ with $s=n/2$, $t=B$ attains
$\Theta$ exactly. Conversely: $s>n/2$ is infeasible, since it would force
$(s-\tfrac n2)M\le B-t\le B<M$. For $s=\tfrac n2-q$ with $q\ge1$, the
requirement $A(S)\ge\Theta$ rearranges to $wq\ge2\delta(qM+B-t)$, but
\[
2\delta(qM+B-t)\;\ge\;2\delta(qM-B)\;\ge\;2\delta q(M-B)\;=\;4\delta qB
\;>\;wq
\]
using $t\le2B$, $q\ge1$, $M=3B$ and $\delta>\tfrac{w}{4B}$ (implied by
$\delta>\tfrac{w}{3B}$): impossible. Finally $s=n/2$ forces $t\ge B$ from
$A\ge\Theta$ and $t\le B$ from feasibility, i.e.\ $t=B$: a balanced
partition.
\end{proof}

Superincreasing radii eliminate exactly this combinatorial layer --
precisely as in the superincreasing knapsack~\cite{Gupte2016} -- leaving
only the geometric one: by Theorem~\ref{thm:oblivious}, the area optimum
then reduces to at most $O(n^2)$ sibling-packing feasibility queries (per
ring, one per currently open container), or $n$ calls to the stronger
oracle ``does the current placement admit ring $i$''. Useful
black-box tools for the remaining geometry include the critical-density theorem
for disks in a disk (any family of total area at most half the container's
packs, and $1/2$ is tight)~\cite{FeketeKeldenichScheffer2019} and Split
Packing~\cite{FeketeMorrScheffer2019}.

\section{Generalizations}\label{sec:general}

The extensions of container shape and dimension below have complete
written proofs. Other changes of the model remain research directions.
\input{generalizations_v2}

\input{extras_v2}

\subsection{Further model variations}
\emph{Flexible rings.} Real squid rings bend: a ring partially resting on
another still touches the pan outside the overlap, up to a lifted ramp of
width $\delta$ around each overlap. The idealization $\delta=0$ (contact area
equals the annulus minus the overlapped region) defines a continuous
relaxation in which partial placements trade contact for cardinality; $\delta>0$
penalizes overlaps by a dead band proportional to the overlap perimeter.

\emph{Infinite inventories.} Unlimited supplies (already used in the
phase diagram) suggest degenerate limits: maximal nested chains
$r,\,r-w,\,r-2w,\dots$, density questions, and Apollonian-type limit
configurations.

\emph{Other objectives.} Theorem~\ref{thm:selection} applies to positive,
strictly increasing, superadditive values. Weighted counts and concave
values lie outside this class. Proposition~\ref{prop:count} shows that
the selection guarantee fails for the constant value $v\equiv1$,
which belongs to both categories.

\section{Conclusions}\label{sec:conclusions}

For disk containers, the exact placement threshold is the golden ratio:
every descending greedy execution reaches the lex-max feasible set at
$\rho\le\varphi$, whereas four-ring failures approach that value from
above. The proof is uniform in the finite inventory size. Golden tail
pressure reduces local disk feasibility to the three largest members,
and two auxiliary slots resolve every branching exchange while the
single-child construction resolves chains. These arguments allow
independent hole radii and do not depend on the historical numerical
coverage of exchange templates.

Placement independence and area optimality are different properties.
With common width, superadditivity supplies the earlier selection
guarantees. With independent holes, the sharp area factor is
$\min(1,\kappa^{-2}-1)$ under $\rho\le\kappa<1$, and its optimality
threshold is $1/\sqrt2$. The square bound $Y$, the rigid-family floor
$T$, and the higher-dimensional five-ring theorem retain their separate
scopes. The remaining questions below concern stronger or different
claims, rather than a missing case in the global disk proof.

\section{Open problems}\label{sec:open}

The universal disk threshold is settled by
Theorem~\ref{thm:golden-global}, including the endpoint and independent
holes. The following questions are not consequences of that theorem.

\begin{openproblem}\label{op:assembly}
Determine the sharp lower floors of the specialized nested-exchange
families beyond the rigid family, whose floor is exactly $T$. Replace
the remaining bounded computational certificates by analytic proofs
where useful, or exhibit limits of those particular constructions.
The historical coverage limits in Appendix~\ref{app:campaign} are not
unresolved cases of the global golden threshold.
\end{openproblem}

\begin{openproblem}\label{op:squareglobal}
Determine the exact square threshold, currently bounded above by $Y$,
and audit the extension of the new global disk theorem to spherical
containers in higher dimensions. The independently reviewed dimensional
claims in this version remain those stated in
Section~\ref{sec:dimension} and Theorem~\ref{thm:five-dim}.
\end{openproblem}

\begin{openproblem}\label{op:oracle}
Beyond the regimes settled here -- including $\rho\le\varphi$
for disk containers -- can a placement rule
with access to the full input attain the lex-max set with
polynomially many calls to a sibling-packing oracle, or is there a
complexity-theoretic obstruction?
\end{openproblem}

\begin{openproblem}\label{op:misc}
Characterize the divergence band of Section~\ref{sec:divergence} exactly;
determine the approximability of the cardinality objective; decide whether
the additive surrogate admits a pseudopolynomial algorithm for integer
radii; and settle membership in NP (versus
$\exists\mathbb R$-completeness) of disk-into-disk packability, whose
NP-hardness is~\cite[Theorem~5.1]{FeketeKeldenichScheffer2019}.
\end{openproblem}

\section*{Acknowledgements}
The results of Sections~\ref{sec:threshold}--\ref{sec:generic} were developed
with extensive computer assistance (Anthropic's Claude) under a
draft--verify workflow: every result was independently rederived from its
bare statement by an adversarial verification pass before being accepted,
and several statements were corrected or sharpened by that pass; the reports
are part of the repository. To be explicit about epistemic status: the
mathematical guarantee for every claim is the written proof and, where a
proof delegates an identity, an exact symbolic computation; the verification
workflow and the numerical sweeps are quality control and evidence, never a
substitute for proof.

The v2 extensions include Theorem~\ref{thm:golden-global}, the
five-ring dimensional guarantee and the area theorem for independent holes. They
were developed with
OpenAI Codex and reviewed independently using Claude Code/Fable.
The dated review payload manifests and verbatim reports are in
\texttt{docs/reviews/}. Those reviews are static; local execution of
Lean and the rational certificates is recorded separately. The new
golden proof and its exchange dependencies received two static Fable
reviews; their requested clarifications are incorporated. The integrated
English exposition and this complete PDF received a local editorial
and rendering review, not an additional full-manuscript Fable review.

\appendix
\section{Verification map}\label{app:verifmap}

Every theorem of Sections~\ref{sec:threshold}--\ref{sec:generic} is backed in
the accompanying repository by (i) a self-contained draft with complete
proofs, (ii) a verification script whose blocks separate exact symbolic
identities (\texttt{sympy}) from numerical sweeps and whose checks all pass (a few blocks
are declared explorations or statements of convention and cannot fail;
the scripts label them as such), and (iii) an
\emph{adversarial verification report}: an independent rederivation of each
result from its statement alone, followed by an audit and directed attacks.
The correspondence is tabulated below; per-round narratives (each
refutation, repair and re-verification) are recorded in the
repository reports (\texttt{docs/drafts/VEREDICTOS.md} and the
per-script reports); the Notes column keeps only the
headline. Unqualified fractions count script blocks; fractions
marked ``claims'' count verified claims.

\begin{center}
\small
\begin{longtable}{>{\raggedright\arraybackslash}p{0.30\textwidth}>{\raggedright\arraybackslash}p{0.28\textwidth}>{\raggedright\arraybackslash}p{0.33\textwidth}}
\hline
Result & Script(s) & Notes \\
\hline
\endfirsthead
\hline
Result & Script(s) & Notes \\
\hline
\endhead
Uniform golden theorem (Thm.~\ref{thm:golden-global}) &
\texttt{ThreeCore.lean}; \texttt{code/}\allowbreak
\texttt{tres\_}\allowbreak\texttt{mayores.py} &
Seven algebraic certificates; 10\,006 placement controls, not a universal
numerical proof; two static Fable reviews of the source proof \\
Five-ring spherical guarantee (Thm.~\ref{thm:five-dim}) &
\texttt{FiveRing.lean} &
Eleven certificates; explicit pockets and parent-alignment proof \\
Independent-hole area factor (Thm.~\ref{thm:variable-area}) &
\texttt{VariableWidth.lean}; \texttt{code/grosor\_}\allowbreak
\texttt{variable.py} &
Seven certificates; exact examples and negative controls; tail capture written \\
Arbitrary-container $n\le3$; dimension reduction and transfer &
\texttt{docs/drafts/} \texttt{generalizacion\_}\allowbreak
\texttt{dimensional.md} &
Written proofs; no packing heuristic used \\
Four-ring threshold (Thm.~\ref{thm:fourfloor}) &
\texttt{lean/Calamares/} \texttt{FourRing.lean} &
Two identities and ordered-ring bound; forest reduction written \\
Square confinement and D (Lem.~\ref{lem:squareQ}, Prop.~\ref{prop:squareD}) &
\texttt{code/cuadrado\_} \texttt{certificado.py} &
17 exact gates; universal Cartesian exclusion in Lean \\
Square twins and finite improvement &
\texttt{code/cuadrado\_} \texttt{gemelas.py},
\texttt{cuadrado\_optimizado.py} &
17+13 exact gates; Cartesian exclusions and witness in Lean \\
Square limiting bound (Thm.~\ref{thm:squarelimit}) &
\texttt{code/cuadrado\_} \texttt{limite.py} &
39 exact gates on three rational instances; continuity written;
three polynomial identities in Lean \\
Additive threshold (Thm.~\ref{thm:additive}) & \texttt{code/umbral.py} & reproduces the $\rho=1.03$/$1.02$ instances \\
Rigid floor (Thm.~\ref{thm:rigidfloor}) & \texttt{code/rigido.py} & 10/10 claims; draft \texttt{suelo\_rigido} \\
Slack pans; universal trio frontier & \texttt{code/universal.py} & 5/5 \\
$\kappa$ identity & \texttt{code/h1.py} & 6/6 claims \\
Width curve; Thm.~\ref{thm:corner} & \texttt{code/grosor.py}, \texttt{code/esquina.py} & 8/8 and 6/6 claims \\
Profile thresholds & \texttt{code/tresk.py}, \texttt{code/cuatrok.py} & drafts \texttt{perfil\_tres}, \texttt{cuatro} \\
Corona criterion & \texttt{code/corona.py} & 5/5 \\
Occupant walls; opposite disk; metallic floors & \texttt{code/ocupantes.py}, \texttt{code/bloqueadores.py} & 5/5 and 6/6 \\
Double pocket; golden line & \texttt{code/bolsillo.py} & 6/6 \\
Three-piece profiles; golden deformation & \texttt{code/striple.py} & 5/5 \\
Golden family (Thm.~\ref{thm:golden}) & \texttt{code/aureo.py} & 5/5 \\
Pan-exchange golden floor & \texttt{code/batalla2.py}, \texttt{code/microcelda.py} & 6/6 and 5/5 \\
Larger profiles & \texttt{code/perfilp.py}, \texttt{code/rstar.py} & 5/5 and 6/6 \\
Corona-versus-tails campaign & \texttt{code/coronacolas.py}, \texttt{code/coronanidada.py} & 5/5 each \\
Duality certificates & \texttt{code/zigzag.py} & 5/5 \\
Exchange assembly; container port; scaling & \texttt{code/ensamblaje.py}, \texttt{code/puertocii.py}, \texttt{code/escala.py} & 6/6, 7/7, 5/5 \\
Written-proof layer: compaction, insertion, gap lemma, hole coronas, geometric tail, optimization, repack & \texttt{code/compactacion.py}, \texttt{code/insercion.py}, \texttt{code/insercion}\discretionary{-}{}{}\texttt{anidada.py}, \texttt{code/gaplemma.py}, \texttt{code/coronaagujero.py}, \texttt{code/colageometrica.py}, \texttt{code/optimizacion.py}, \texttt{code/repack.py} & 5/5--7/7 each \\
Exact-arithmetic and LP layers & \texttt{code/r2bcert.py}, \texttt{code/arcolp.py}, \texttt{code/bolsillos.py} & 5/5 each \\
Conditional duality lemma; arc-LP to $k=6$ & \texttt{code/f3cierre.py} & 5/5; round refuted a stronger closure claim \\
Multipiece certificates & \texttt{code/r2bmulti.py} & 5/5; domain corrected to the sweep box \\
Golden-threshold reduction; heavy branches & \texttt{code/areduccion.py} & 5/5; algebra confirmed by hand \\
Diametral-pocket geometry ($X_Y>0$) & \texttt{code/espxy.py} & 5/5; an over-claimed legality was recorded as errata \\
Sliver vacuity; floor rigidity & \texttt{code/espvals.py} & 5/5; round supplied a missing wall \\
Global-tail audit & \texttt{code/auditcolas.py} & 5/5; first version refuted, repaired \\
Forced sub-pocket (near-equal tops) & \texttt{code/f3vacio.py} & 5/5 \\
Light mirrored $X_Y>0$ ($k$ pieces) & \texttt{code/espkp.py} & 5/5; dust convention made explicit \\
Heavy mirrored $X_Y>0$ & \texttt{code/esppesada.py}, \texttt{code/espfinal.py} & 5/5 and 6/6; one surgical gap repaired \\
Occupant channel (light; heavy) & \texttt{code/espcanal.py}, \texttt{code/espcanalp.py} & 5/5 each; a residual sheet dissolved, a tower cut repaired \\
Exact duality-gap converse & \texttt{code/f3converso.py} & 5/5; first round refuted the draft, all repairs re-verified \\
C3.3 corona wall (box certificate) & \texttt{code/rstarcert.py} & 5/5; 71 boxes \\
$2\to3$ divergence & \texttt{code/divergencia3.py} & 5/5; rational-exact counterexample \\
Phase-diagram figure & \texttt{code/franja.py} & generator \\
D1 shadow budget (box certificate) & \texttt{code/insercioncert.py} & 5/5; 149+173 boxes, all occupant counts \\
Branch-B golden-line patches & \texttt{code/goldencert.py} & 5/5; exact Sturm isolation \\
Gap-lemma cross-derivation & \texttt{code/gaplemmacert.py} & 5/5; contrast only \\
Width tails $\omega>1.6$; $Y>6.6$ tails & \texttt{code/espomegacola.py}, \texttt{code/r2bcolas.py} & 5/5 each \\
Deep-tower mid band (light; heavy) & \texttt{code/esptorre.py}, \texttt{code/esptorrep.py} & 5/5 each \\
Width tail of all six channel cells & \texttt{code/espomegacanal.py} & first round refuted two coverage gaps; repaired, re-confirmed \\
Slot-reduction lemma; placement engine & \texttt{code/lemaA.py} & 5/5; engine rebuilt after refutation, confirmed against an LP oracle \\
G-b$'$ over all $Y$, mass, piece counts & \texttt{code/lemaA2.py} & 6/6; two unsound caps repaired before refereeing \\
Heavy assembly murals G-e/G-g & \texttt{code/lemaA3.py} & 5/5; a $k\ge6$ block cap repaired \\
Light channel, $k\ge2$ extras & \texttt{code/lemaA4.py} & domain $\omega\le1.25$ (leaf extras, full nested-mass range), parents on $\omega\le1.05$, $W_v\le8$, and the $j_v\le1$ slice of $[1.6,2]$; nine rounds, four intermediate refutations repaired \\
Heavy profile, $k\ge2$ extras & \texttt{code/lemaA5.py} &
$\omega\le1.25$ after the certified-corner extension; leaf extras,
$W_z\le34$, full $W_v$ range \\
Quintet certificate of Theorem~\ref{thm:gapwritten} & \texttt{code/quintetocert.py} & 4/4; whole domain, 11{,}973 boxes, tolerance-free rational acceptance, no platform hypotheses; supersedes \texttt{code/bolsillos.py} and
\texttt{code/arcolp.py} for the $j=1$ quintet \\
\hline
\end{longtable}
\end{center}

Several statements were corrected or sharpened by the adversarial
rounds (a missing hypothesis in the trio frontier, a refuted
``leak'' family, the $\alpha\ge o_1$ case split, a refuted
induction in the duality layer, a repaired extension lemma and case
split in the assembly); the reports record each incident.

A fourth layer is machine-checked. The repository contains a Lean~4
development (\texttt{lean/}, core only, no \texttt{mathlib}) that formalizes
the \emph{exact-certificate layer} on which the proofs rest: arithmetic in
$\mathbb Q$ and $\mathbb Q[\sqrt5]$ with a decidable real order, and
polynomial identities over $\mathbb Z[A]$, $(\mathbb Q[e])[d]$ and
$(\mathbb Q[\sqrt5])[r]$. Its
original fifty-seven theorems cover the golden arithmetic of Theorem~\ref{thm:golden}
($\varphi^2=\varphi+1$, $b_2(\varphi,1)=\varphi/2$, the fixed point and its
factorization, the witness certificate at $s^{*}=4(\sqrt5-2)$, the
$\varepsilon_0$ margin and the width window), the metallic-mean certificates
of Theorem~\ref{thm:DP} (including $\Psi(1/2)=\Psi_B(1)=\Psi_2(\varphi/2)
=\varphi$ and the monotonicity certificate for $g$), the Tribonacci floor
($P(\varphi)=-1$, the bracket $1.8392<T<1.8393$ and a certificate that $P$
increases strictly on $[1,\infty)$), the $13/7$ corner, the additive
threshold family, the pincer and boundary-corner certificates of
Theorem~\ref{thm:DPr}, and the duality-certificate identities of
Appendix~\ref{app:campaign} (the vanishing Descartes discriminant and
pocket $\varphi/2$ of the golden pair, the critical margin identity,
the $(N)$-heritage golden line, the mirror corner
$b_2(2,\sqrt5-1)=1$, the strong-breathing chain
$\varphi(3-\varphi)=2\varphi-1=\sqrt5$,
$\sqrt5-(\varphi-1)=\varphi$, the golden $\pi$-trio identity
behind $\theta(\varphi,1/\varphi)+\theta(1/\varphi,1)
+\theta(1,\varphi)=\pi$ at $R=2\varphi$, the reduction-lemma
thresholds $t_0=(\varphi-1)/4=1/(4\varphi)$ with
$5t_0=\varphi-\beta^{*}$ and $4t_0=\varphi-1$ exact at
$\beta^{*}=(9-\sqrt5)/8$, the diametral-pocket identities
$\varphi^2+\varphi+1=2\varphi^2$,
$x^{*}(\varphi)=\varphi/2$ and
$\varphi(4-\varphi)=3\varphi-1$, and the forbidden-triple and
golden-cubic certificates of the near-equal-tops vacuity --
$2\cdot(\varphi/2)=\varphi$, the polynomial identity
$r(1+r+r^2)+r(1+r)-\varphi(1+r+r^2)
=r^3+(2-\varphi)r^2+(2-\varphi)r-\varphi$ in
$(\mathbb Q[\sqrt5])[r]$, $q(9/10)=171/271$ and the ceiling
inequality $\varphi-9/10-171/271<25/83$; and, from the
extra-ring campaign, the tail-slope identities $2/\varphi=\sqrt5-1>1$
and $(2/\varphi)(2/\varphi-1)=7-3\sqrt5$ of the $W_z$-tail's upper
piece with the numerator coefficient $(6+2\sqrt5)(7-3\sqrt5)
=12-4\sqrt5$, the heavy-mass bound $2\varphi/3<1.079$ with its
genuine extremum at $b=\varphi/3$, and the spillover-bound
identities of Theorem~\ref{thm:D1written} with
$\varphi-1=1/\varphi$; and the exact anchors of the quintet
certificate — the stacked-witness margin $2\varphi^4-(\varphi+2)^2
=\varphi-1$ at the golden point, $2\varphi-1=\sqrt5$ with
$\sqrt5\ge2$ for the mid-arc criterion, and the semiarc threshold
collapse $1+f(2/\varphi)+f(2)=1+\sqrt5$ at $R=1+\sqrt5$).
The v2 adds 24 square theorems, described in
Section~\ref{sec:squarelimit}, and three four-ring algebraic theorems,
bringing that stage to 84. Later modules add 18 certificates for five
rings and variable width, three for balanced partitions, three for
two-tail pressure, seven for reservoir bounds, and seven for the
three-largest-disks proof: 122 in the repository. The reservoir bounds
and the historical adjacent-swap classification are not premises of
Theorem~\ref{thm:golden-global}. There are no \texttt{sorry}s and no new axioms
(\texttt{\#print axioms} reports only \texttt{propext},
\texttt{Classical.choice} and \texttt{Quot.sound}); \texttt{native\_decide}
is not used, so every proof is checked by the Lean kernel. The square's
Cartesian inequalities are formalized. Corona rigidity, angular
criteria, placement trees, evacuations, dimension reduction and
continuity remain written proofs, with supporting scripts and reports.

\section{Complete proof of Theorem~\ref{thm:rigidfloor}}\label{app:rigidproof}

Throughout, work in the normalization of Theorem~\ref{thm:rigidfloor}:
$r_1=1$, $t=r_2$, $p=r_3$, $q=r_4$, width $\omega>0$, pan radius $R=1+t$,
and recall $b(t)=t(1+t)/(1+t+t^2)$, $T^3=T^2+T+1$, $t^{*}=1/T$
(equivalently $t^{*3}+t^{*2}+t^{*}=1$), $\varphi=(1+\sqrt5)/2$. All exact
algebraic identities used below are additionally verified in
\texttt{code/rigido.py} (block V5).

\begin{lemma}[Half-angle identity]\label{lem:S1}
Let two circles of radii $a,b$ be internally tangent to the wall of a disk
of radius $R$ (centers at distances $R-a$, $R-b$ from the center). If
$a+b\le R$, the minimal angular separation $\theta(a,b)$ guaranteeing
disjoint interiors is well defined, lies in $[0,\pi]$, satisfies
\[
\sin^2\tfrac{\theta(a,b)}2=f(a)f(b),\qquad f(x)=\frac{x}{R-x},
\]
and is increasing in each argument; two wall-tangent circles at angular
separation $\gamma$ have disjoint interiors iff $\gamma\ge\theta(a,b)$.
\end{lemma}
\begin{proof}
By the law of cosines the center distance
$D(\gamma)^2=(R-a)^2+(R-b)^2-2(R-a)(R-b)\cos\gamma$ is increasing on
$[0,\pi]$, and disjointness is $D\ge a+b$. At the threshold,
$1-\cos\theta=\bigl[(a+b)^2-(a-b)^2\bigr]/\bigl(2(R-a)(R-b)\bigr)
=2ab/\bigl((R-a)(R-b)\bigr)$, i.e.\
$\sin^2(\theta/2)=f(a)f(b)$; and $f(a)f(b)\le1$ is exactly $a+b\le R$.
Monotonicity is that of $f$.
\end{proof}

\begin{lemma}[Construction]\label{lem:S2}
Let $q\le p\le t<1$ and $R=1+t$. If
$\theta(1,p)+\theta(1,q)+\theta(p,q)\le2\pi$, then $\{1,p,q\}$ packs in the
disk of radius $R$.
\end{lemma}
\begin{proof}
All three angles are defined ($1+p\le R\iff p\le t$, and $p+q\le2t<R$).
Place the three circles wall-tangent with the unit circle in the middle:
centers at angles $0$, $+\theta(1,p)$, $-\theta(1,q)$. Each circle is
contained in the disk (internal wall tangency). The pairs $(1,p)$ and
$(1,q)$ are separated by exactly their thresholds. For $(p,q)$, the angular
difference is $\Delta=\theta(1,p)+\theta(1,q)$ and the separation is
$\gamma=\min(\Delta,2\pi-\Delta)$. If $\gamma=\Delta$: by monotonicity
(Lemma~\ref{lem:S1}, $1\ge p$), $\theta(1,q)\ge\theta(p,q)$, so
$\Delta\ge\theta(p,q)$. If $\gamma=2\pi-\Delta$: the hypothesis gives
$2\pi-\Delta\ge\theta(p,q)$.
\end{proof}

\begin{lemma}[Algebraic reduction]\label{lem:S3}
Define $\psi(x)=\sqrt{(t-x)/x}$ (decreasing, $\psi(t)=0$) and
$\tau_t=t/\sqrt{1+t}$; note the exact identities $\psi(b(t))=\tau_t$ and
$t/(1+\tau_t^2)=b(t)$. If $q\le p\le t<1$ and $\psi(p)+\psi(q)\ge\tau_t$, then
$\theta(1,p)+\theta(1,q)+\theta(p,q)\le2\pi$, and hence $\{1,p,q\}$ packs
in the disk $R=1+t$ by Lemma~\ref{lem:S2}.
\end{lemma}
\begin{proof}
Let $A=\sqrt{f(1)f(p)}$, $B=\sqrt{f(1)f(q)}$, $C=\sqrt{f(p)f(q)}$ (the
sines of the half-angles); the goal is
$\arcsin A+\arcsin B+\arcsin C\le\pi$. If $\arcsin A+\arcsin B\le\pi/2$
there is nothing to prove, since $\arcsin C\le\pi/2$. Otherwise the goal is
equivalent, both sides lying in $[0,\pi/2]$, to
$C\le\sin(\arcsin A+\arcsin B)=A\sqrt{1-B^2}+B\sqrt{1-A^2}$. Dividing by
$C>0$ and writing $G(x)=f(1)/f(x)-f(1)^2$,
\[
\frac{A\sqrt{1-B^2}}{C}=\sqrt{G(q)},\qquad
\frac{B\sqrt{1-A^2}}{C}=\sqrt{G(p)},\qquad
G(x)=\frac{(1+t)(t-x)}{t^2x},
\]
the last identity by direct computation with $f(1)=1/t$,
$f(x)=x/(1+t-x)$; the radicals are real because $f(1)f(x)\le1\iff x\le t$.
Thus $\sqrt{G(x)}=\bigl(\sqrt{1+t}/t\bigr)\psi(x)$ and, by hypothesis,
$\sqrt{G(p)}+\sqrt{G(q)}=\bigl(\sqrt{1+t}/t\bigr)(\psi(p)+\psi(q))
\ge\bigl(\sqrt{1+t}/t\bigr)\tau_t=1$, which is the required inequality.
\end{proof}

The contrapositive used below: hypothesis (F3) implies
$\psi(p)+\psi(q)<\tau_t$. This is the only direction ever used -- the
\emph{constructive} one -- so no exactness of any angular criterion is
assumed anywhere in this appendix.

\begin{lemma}[Corner minimization]\label{lem:S4}
Let $0<q\le p\le t<1$ with $\psi(p)+\psi(q)<\tau_t$. Then:
\textup{(1)} $p,q>b(t)$;
\textup{(2)} if $t\ge1/\varphi$ then $p+q>1$;
\textup{(3)} if $t<1/\varphi$ then $p+q>t+b(t)$.
\end{lemma}
\begin{proof}
(1) Each summand is $<\tau_t$, $\psi$ is decreasing and $\psi(b(t))=\tau_t$.
(2) $p+q>2b(t)$ and $2b(t)-1=(t^2+t-1)/(t^2+t+1)\ge0$ exactly when
$t\ge1/\varphi$.
(3) Substitute $\alpha=\psi(p)$, $\beta=\psi(q)$, so $p=U(\alpha)$,
$q=U(\beta)$ with $U(z)=t/(1+z^2)$ strictly decreasing; the region is
$\alpha,\beta\ge0$, $\alpha+\beta<\tau_t$. Increasing $\alpha$ to
$\tau_t-\beta$ decreases $p$, so $p+q>U(\tau_t-\beta)+U(\beta)=:W(\tau_t-\beta)$.
Now $U''(z)=2t(3z^2-1)/(1+z^2)^3\le0$ for $z\le1/\sqrt3$, and
$\tau_t\le1/\sqrt3\iff3t^2-t-1\le0\iff t\le(1+\sqrt{13})/6=0.7676\dots$,
which covers $t<1/\varphi=0.618\dots$ with room to spare. Hence $U$ is
concave on $[0,\tau_t]$, $W$ is concave, and its minimum over $[0,\tau_t]$ is
at the endpoints: $W\ge W(0)=U(0)+U(\tau_t)=t+t/(1+\tau_t^2)=t+b(t)$, by the
identity $t/(1+\tau_t^2)=b(t)$.
\end{proof}

\begin{proof}[Proof of the strict floor in Theorem~\ref{thm:rigidfloor}]
Let $I\in\mathcal F$, normalized. By (F3) and the contrapositive of
Lemmas~\ref{lem:S2}--\ref{lem:S3}, $\psi(p)+\psi(q)<\tau_t$. If
$t\ge1/\varphi$, Lemma~\ref{lem:S4}(2) gives $p+q>1$, contradicting (F2)
($p+q\le1-\omega<1$); so $t<1/\varphi$ and Lemma~\ref{lem:S4}(3) gives the
blocking pressure $p+q>t+b(t)$. Chaining with the witness pressure (F2),
\[
t+b(t)\;<\;p+q\;\le\;1-\omega\;<\;1,
\]
and $t+b(t)<1\iff t(1+t)<(1-t)(1+t+t^2)=1-t^3\iff t^3+t^2+t<1\iff t<t^{*}$
(which also proves that $\mathcal F$ forces $t<t^{*}$). Finally
\[
\rho\;\ge\;\rho_2=\frac{p+q}{t}\;>\;\frac{t+b(t)}{t}
\;=\;1+\frac{1+t}{1+t+t^2}\;=:\;L(t),
\]
and $L$ is strictly decreasing ($L'=-(t^2+2t)/(1+t+t^2)^2<0$) with
$L(t^{*})=1/t^{*}=T$ (an identity modulo $t^{*3}+t^{*2}+t^{*}=1$). Since
$t<t^{*}$: $\rho>L(t)>L(t^{*})=T$.
\end{proof}

It remains to prove exactness of the infimum. Two ingredients: the rigid
pocket is exact, and infeasibility survives small perturbations.

\begin{proposition}[Exact rigid pocket]\label{prop:S5}
For $0<q\le t<1$, the triple $\{1,t,q\}$ packs in the disk of radius
$R=1+t$ if and only if $q\le b(t)$.
\end{proposition}
\begin{proof}
\emph{Rigidity.} Centers satisfy $|c_1|\le t$, $|c_t|\le1$ and
$|c_1-c_t|\ge1+t$; the triangle inequality forces equality throughout, so
up to rotation $c_1=(-t,0)$, $c_t=(1,0)$.
\emph{Sufficiency.} The pocket circle tangent to the wall and to both has
radius $b(t)$; any $q\le b(t)$ placed concentrically inside it works.
\emph{Necessity.} Let $X$ be the center of $q$ and $d=|X|$; if $d=0$ then
$|X-c_1|=t<1+q$, infeasible, so $d>0$. Let $\gamma$ be the
angle between $X$ and $(1,0)$. Disjointness reads
$d^2+t^2+2dt\cos\gamma\ge(1+q)^2$ and $d^2+1-2d\cos\gamma\ge(t+q)^2$, which
bound $\cos\gamma$ from below and above; compatibility of the two bounds
requires $(1+t)d^2\ge(1+q)^2+t(t+q)^2-t(1+t)$, and with $d\le R-q=1+t-q$:
\[
(1+t)(1+t-q)^2-(1+q)^2-t(t+q)^2+t(1+t)\;\ge\;0.
\]
The left-hand side factors \emph{exactly} as
$4\,(t^2+t+1)\,\bigl(b(t)-q\bigr)$ (symbolic identity, V5), whence
$q\le b(t)$.
\end{proof}

\begin{lemma}[Closure and monotonicity of packability]\label{lem:S6a}
Let $D=\{(t,p,q):0<q\le p\le t<1\}$ and say $z=(t,p,q)$ \emph{packs} if
there are centers with $|c_1|\le t$, $|c_2|\le1+t-p$, $|c_3|\le1+t-q$,
$|c_1-c_2|\ge1+p$, $|c_1-c_3|\ge1+q$, $|c_2-c_3|\ge p+q$. Let
$E\subseteq D$ be the packable set. Then:
\textup{(1)} $E$ is downward closed in $(p,q)$;
\textup{(2)} $E$ is closed in $D$;
\textup{(3)} for fixed $t,q$, the set $\{p\in[q,t]:(t,p,q)\in E\}$ is empty
or a closed interval $[q,p_{\max}]$ with the maximum attained; in
particular, if $(t,t,q)\notin E$, then (when the interval is nonempty;
if it is empty every $p$ is infeasible and any $\delta$ serves)
$(t,p,q)\notin E$ for all
$p\in(p_{\max},t]$ -- the infeasibility of the rigid corner propagates to
$\delta:=t-p\in[0,\delta_0)$, $\delta_0=t-p_{\max}>0$.
\end{lemma}
\begin{proof}
(1) The same centers witness (all six constraints relax). (2) Witness
centers live in a fixed compact ball; take a convergent subsequence of
witnesses along $z_k\to z$ and pass to the limit in the six non-strict
continuous constraints. (3) By (1) the set is an initial segment, by (2) it
is closed.
\end{proof}

\begin{proposition}[The infimum is $T$ and is not attained]\label{prop:S6}
$\inf\{\rho(I):I\in\mathcal F\}=T$.
\end{proposition}
\begin{proof}
$\ge$ and non-attainment: the strict floor above. $\le$: fix $t<t^{*}$, set
$\varepsilon_t=(1-t-b(t))/2>0$ and
$q_t=b(t)+\min\bigl(\varepsilon_t,(t-b(t))/2\bigr)$. By
Proposition~\ref{prop:S5}, $(t,t,q_t)$ is infeasible; by
Lemma~\ref{lem:S6a}(3) there is $\delta\in(0,\min(\varepsilon_t,t-q_t))$
with $(t,t-\delta,q_t)$ infeasible. (Write
$\rho_i=(\sum_{j>i}r_j)/r_i$ for the $i$th tail ratio.) The instance
$I_t=(1,\,t,\,t-\delta,\,q_t)$ with $\omega=1-(t-\delta+q_t)>0$ satisfies
(F1)--(F3) and the strict orderings, so $I_t\in\mathcal F$, and
$\rho(I_t)=\rho_2\le\bigl(t+b(t)+\varepsilon_t\bigr)/t
=L(t)+\varepsilon_t/t\longrightarrow L(t^{*})=T$ as $t\to t^{*}$
(the other tails are smaller: $\rho_1=t+p+q<3t<T$ and $\rho_3=q/p<1$).
\end{proof}

\begin{proof}[Proof of Remark~\ref{rem:slack}]
A packing in a disk of radius $R'\le R$ is a packing in the disk of radius
$R$ (translate to concentric position; containment is preserved, pairwise
constraints are unchanged). Hence non-packability at radius $R$ implies
non-packability at radius $r_1+r_2\le R$, so an instance with a slack pan
satisfying (F2) and (F3) at its own radius satisfies (F2)--(F3) at radius
$r_1+r_2$, and the floor applies verbatim; the approximating family of
Proposition~\ref{prop:S6} has $R=r_1+r_2$, which is admissible, so the
infimum is unchanged.
\end{proof}

\section{Complete proofs for the width program}\label{app:widthproofs}

This appendix proves every statement of Section~\ref{sec:width}. Throughout,
normalize $r_m=1$ and $\omega=w/r_m\in(0,1)$, and write $\beta:=1-\omega$.
The reinsertion resources of the exchange step are the two root
containers $A$ (the vacated disk $D_m$, capacity $1$) and $B$ (the
hole $H_m$, capacity $\beta$), plus
recursive nesting inside $S$ itself: a ring $x$ fits in the hole of a ring
$y$ iff $x\le y-\omega$. For siblings in a disk of capacity $c$ we use:
one circle, $s\le c$; two circles, $s+s'\le c$, which is \emph{exact}
(sufficiency by tangent placement on a diameter; necessity from
$s+s'\le|c_1-c_2|\le|c_1|+|c_2|\le(c-s)+(c-s')$); and for three or more
only \emph{constructive} certificates -- the row of
Lemma~\ref{lem:row} and explicit placements -- so that every lower bound
below is unconditional. For a profile $S=\{s_1\ge\dots\ge s_k\}\subset(0,1)$,
\[
\rho_{\mathrm{needed}}(S)=\max\Bigl(\textstyle\sum S,\
\max_j\bigl(\textstyle\sum_{l>j}s_l\bigr)/s_j\Bigr),
\]
and $\rho^*_k(\omega)=\inf\{\rho_{\mathrm{needed}}(S):S$ not reinsertable
in $A\uplus B\}$.
Two elementary facts are used without mention: $\rho_{\mathrm{needed}}$ is
monotone under adding smaller rings, and if a prefix $\{s_1,\dots,s_j\}$ is
already non-reinsertable then so is $S$ (every placement restricts).

\subsection{Profile thresholds}\label{app:profiles}

\begin{proposition}[Two rings]\label{prop:Cpair}
$S=\{\sigma_1\ge\sigma_2\}$ is non-reinsertable iff $\sigma_1+\sigma_2>1$
and $\sigma_2>1-\omega$. Hence $\rho^*_2(\omega)=\max\bigl(1,2(1-\omega)\bigr)$,
approached by $\sigma_1=\sigma_2\to\max(\tfrac12,1-\omega)^+$.
\end{proposition}
\begin{proof}
The possible placements are: both in $A$ (iff $\sigma_1+\sigma_2\le1$,
two-circle exact); one in each nest (the one in $B$ measures
$\ge\sigma_2$, so it requires $\sigma_2\le\beta$); $\sigma_2$ nested in
$\sigma_1$ (requires $\sigma_2\le\sigma_1-\omega\le\beta$); both in $B$
(stronger than one in each). Both stated conditions kill all four; failure
of either condition enables the first or second placement. The infimum:
blocking forces $\rho\ge\sigma_1+\sigma_2>\max(1,2\beta)$, and the twin
family attains the bound in the limit (its tail ratios tend to $1$ and $0$).
\end{proof}

\begin{proposition}[Three rings: characterization]\label{prop:Ctrichar}
Let $s_1\ge s_2\ge s_3$ in $(0,1)$. $S$ is non-reinsertable iff exactly one
of the following holds (the cases are disjoint by the position of $s_1,s_2$
relative to $\beta$):
\begin{itemize}\itemsep2pt
\item[(i)] $s_2>\beta$ and $s_1+s_2>1$ (then $s_3$ is arbitrary);
\item[(ii)] $s_2>\beta$, $s_1+s_2\le1$, $s_3>\beta$, and the trio does not
pack in the unit disk (possible only for $\omega>\tfrac12$);
\item[(iii)] $s_2\le\beta<s_1$, $s_1+s_3>1$, $s_2+s_3>\beta$, and
$s_3>s_1-\omega$;
\item[(iv)] $s_1\le\beta$, $s_2+s_3>1$, and $s_3>s_1-\omega$.
\end{itemize}
\end{proposition}
\begin{proof}
The nesting forests on three rings are finite: no edge; one edge ($s_3$ in
$s_1$, $s_3$ in $s_2$, $s_2$ in $s_1$); the chain $s_3\subset s_2\subset
s_1$; the star $\{s_2,s_3\}$ in the hole of $s_1$. Top-level rings are
distributed among $A$ and $B$ under the sibling rules above; $S$ is
reinsertable iff some combination is feasible.

\emph{(i).} Neither $s_1$ nor $s_2$ fits in $B$ ($>\beta$) and $s_2$ cannot
nest in $s_1$ ($s_2\le s_1-\omega<\beta<s_2$ is absurd), so both must go to
$A$, and $s_1+s_2>1$ forbids it: blocked for every $s_3$. Conversely, with
$s_2>\beta$ and $s_1+s_2\le1$: if $s_3\le\beta$, place $\{s_1,s_2\}$ in a
row in $A$ and $s_3$ in $B$; if $s_3>\beta$, all three live in the band
$(\beta,1)$, $B$ is useless, nobody nests ($s_i-\omega<\beta<s_3$), so all
three must go to $A$ and reinsertability is exactly packability of the trio
in the unit disk -- case (ii), whose condition is also sufficient by the
same argument. Habitability of (ii): $s_1+s_2>2\beta$ and $s_1+s_2\le1$
force $\omega>\tfrac12$.

\emph{(iii).} ($\Leftarrow$) $s_1>\beta$ has no parent and does not fit in
$B$: $s_1\in A$. No edge survives: $s_3>s_1-\omega\ge s_2-\omega$ kills all
three. A companion of $s_1$ in $A$ needs $s_1+s_i\ge s_1+s_3>1$: impossible.
So $s_2,s_3$ must both go to $B$, and $s_2+s_3>\beta$ forbids it.
($\Rightarrow$) If $s_1+s_3\le1$: $\{s_1,s_3\}$ in $A$, $s_2$ in $B$. If
$s_2+s_3\le\beta$: $\{s_2,s_3\}$ in $B$, $s_1$ in $A$. If $s_3\le
s_1-\omega$: nest $s_3$ in $s_1$, $s_1\to A$, $s_2\to B$.

\emph{(iv).} ($\Leftarrow$) $s_2+s_3>1$ makes \emph{every} pair infeasible
in any root ($s_i+s_j\ge s_2+s_3>1>\beta$), so each root hosts at most one
top-level ring; two roots cannot host three rings unless someone nests, and
$s_3>s_1-\omega$ kills all edges as in (iii). ($\Rightarrow$) If
$s_2+s_3\le1$: $\{s_2,s_3\}$ in $A$, $s_1$ in $B$ ($s_1\le\beta$). If
$s_3\le s_1-\omega$: nest $s_3$ in $s_1$, $s_1\to A$, $s_2\to B$.
\end{proof}

\begin{theorem}[Three rings: closed formula]\label{thm:Ctrio}
For all $\omega\in(0,1)$,
\[
\rho^*_3(\omega)=\max\Bigl(1,\ \min\bigl(2(1-\omega),\
\max\bigl(\varphi,\ \tfrac{2}{1+2\omega}\bigr)\bigr)\Bigr),
\]
with extremal families (as $\varepsilon\to0^+$):
$\{\tfrac12+\omega,\tfrac12+\varepsilon,\tfrac12+\varepsilon\}$ on the
hyperbola, $\{1/\varphi,\tfrac12+\varepsilon,\tfrac12+\varepsilon\}$ on the
golden plateau
$\omega\in\bigl[\tfrac{\sqrt5-2}{2},1-\tfrac{\varphi}{2}\bigr]$,
$\{\beta+\varepsilon,\beta+\varepsilon,\delta\}$ with $\beta=1-\omega$
on $\omega\in\bigl[1-\tfrac{\varphi}{2},\tfrac12\bigr]$ (there
$2(\beta+\varepsilon)>1$: the pair cannot row), and
$\{\tfrac12+\varepsilon,\tfrac12+\varepsilon,\delta\}$ for
$\omega\ge\tfrac12$ (where $\rho^*_3=1$; the family
$\{\beta+\varepsilon,\beta+\varepsilon,\delta\}$ rows into the unit
disk there and no longer blocks). The proof is
purely additive: neither case (ii) nor any three-circle geometry enters.
\end{theorem}
\begin{proof}
We bound the infimum of $\rho_{\mathrm{needed}}$ over each case of
Proposition~\ref{prop:Ctrichar} and take the minimum.

\emph{Case (i).} $\rho\ge s_1+s_2>\max(1,2\beta)$, and the family
$s_1=s_2=\max(\beta,\tfrac12)+\varepsilon$, $s_3=\delta\to0$ is blocked
with $\rho_{\mathrm{needed}}\to\max(1,2\beta)=\rho^*_2$. (This ``pair plus
dust'' argument also shows $\rho^*_{k+1}\le\rho^*_k$ for every $k$.)

\emph{Case (ii).} Any blocked profile has $\sum S>1$ (otherwise the row of
Lemma~\ref{lem:row} places the trio in $A$), so its infimum is $\ge1$,
the formula's value on $\omega\ge\tfrac12$.

\emph{Case (iii).} From $s_3>\max(1-s_1,\,s_1-\omega)$ and $s_2\ge s_3$:
$\sum S\ge g(s_1):=s_1+2\max(1-s_1,s_1-\omega)$, which decreases until
$s_1=(1+\omega)/2$ and increases after, with minimum $(3-\omega)/2$; and
$s_1>\beta$. If $\omega\le\tfrac13$ then $\beta\ge(1+\omega)/2$ and
$\sum S>g(\beta)=3-5\omega\ge2-2\omega$; if $\omega\in(\tfrac13,\tfrac12)$,
$\sum S>(3-\omega)/2\ge2-2\omega$; if $\omega\ge\tfrac12$,
$\sum S\ge s_1+s_3>1$. In all ranges case (iii) is dominated by the
formula's value: it never furnishes the minimum.

\emph{Case (iv).} Set $q=s_2+s_3$. Blocking gives $q>1$ and
$q\ge2s_3>2(s_1-\omega)$, with $s_1\le\beta$, while
$\rho\ge\max(s_1+q,\ q/s_1)$. For fixed $s_1$ both bounds grow with $q$, so
$q\downarrow Q(s_1):=\max(1,2(s_1-\omega))$. On the branch
$s_1\ge\tfrac12+\omega$ ($Q=2(s_1-\omega)$) both
$3s_1-2\omega$ and $2-2\omega/s_1$ increase in $s_1$: the branch minimum is
$\max\bigl(\tfrac32+\omega,\tfrac{2}{1+2\omega}\bigr)$ at the left endpoint.
On the branch $s_1\le\tfrac12+\omega$ ($Q=1$),
$h(s_1)=\max(s_1+1,1/s_1)$ decreases until the crossing $s_1^2+s_1=1$,
i.e.\ $s_1=1/\varphi$ where $h=\varphi$, and increases after; the branch
minimum sits at $s_1^*=\min(\beta,\tfrac12+\omega,1/\varphi)$. Comparing
(exact crossings): for $\omega\le\omega_\varphi:=\tfrac{\sqrt5-2}{2}$ the infimum
is $\tfrac{2}{1+2\omega}$ (the crossing
$\tfrac{2}{1+2\omega}=\tfrac32+\omega$ is $2\omega^2+4\omega-\tfrac12=0$,
positive root $\omega_\varphi$; this local symbol is unrelated to the root
$\omega_1$ of Proposition~\ref{prop:Ccurve}); for
$\omega_\varphi\le\omega\le1-1/\varphi$ it is $\varphi$; for
$1-1/\varphi\le\omega<\tfrac12$ it is
$\max(2-\omega,1/\beta)=1/\beta\ge2\beta$: dominated by case (i). For
$\omega\ge\tfrac12$ case (iv) is empty ($s_2+s_3\le2\beta<1$). The lower
bounds are approached by the stated families ($s_2=s_3=q/2$ with $s_1$ at
each branch minimum satisfies all blocking conditions with margin
$\varepsilon$).

\emph{Global minimum.} $\tfrac{2}{1+2\omega}<2(1-\omega)$ on
$(0,\tfrac12)$ (equivalent to $\omega(1-2\omega)>0$) and
$\varphi\le2(1-\omega)\iff\omega\le1-\varphi/2$: the case-wise minimum is
the stated formula.
\end{proof}

\begin{proposition}[All profile sizes]\label{prop:Callk}
$\rho^*_k(\omega)=\rho^*_3(\omega)$ for every $k\ge3$ and every
$\omega\in(0,1)$.
\end{proposition}
\begin{proof}
$(\le)$ Dust: append $k-3$ superdecreasing rings
$\delta,\delta^2,\dots,\delta^{k-3}$ with $\delta\to0$ to any blocked
three-ring witness (superdecreasing, so that every dust tail also tends to
$0$ and $\rho_{\mathrm{needed}}$ tends to that of the witness).
$(\ge)$ Let $S$ be non-reinsertable with $k$ rings and
$\rho:=\rho_{\mathrm{needed}}(S)$.

\emph{Step 0.} The row of all of $S$ in $A$ fails: $\sum S>1$, so $\rho>1$;
this closes $\omega\ge\tfrac12$. Assume $\omega<\tfrac12$ and
$\rho<2\beta$ (otherwise done).

\emph{Step 1.} Two rings $>\beta$ would give $\sum S>2\beta\ge\rho\ge\sum S$.
If only $s_1>\beta$: the placement ``$s_1\to A$ alone, the rest in a row in
$B$'' fails, and since each $s_i\le\beta$ fits alone the failure is the
sum: $\sum S-s_1>\beta$, whence $\sum S>2\beta$: impossible. So
$s_1\le\beta$.

\emph{Step 2.} ``$s_1\to B$, the rest in a row in $A$'' fails:
$Q:=\sum S-s_1>1$. From $\rho\ge s_1+Q$ and $\rho\ge Q/s_1$:
$\rho\ge\min_{s\le\beta}\max(1+s,1/s)$, which equals $\varphi$ if
$\beta\ge1/\varphi$ and $1/\beta$ otherwise; in either case
$\rho\ge\rho^*_3$ for all $\omega\ge\tfrac{\sqrt5-2}{2}$ (on the plateau
$\rho^*_3\le\varphi$; beyond it $\rho^*_3=2\beta\le\varphi$ resp.\
$2\beta\le1/\beta$ for $\beta\le1/\sqrt2$). Assume then
$\omega<\tfrac{\sqrt5-2}{2}$, where $\rho^*_3=\tfrac{2}{1+2\omega}$, and
towards a contradiction $\rho<\tfrac{2}{1+2\omega}$. From
$\rho\ge Q/s_1>1/s_1$ we get $s_1>\tfrac12+\omega$, in particular
$s_1-\omega>\tfrac12$.

\emph{Step 3.} Let $p:=\#\{i:s_i>s_1-\omega\}$; all such rings exceed
$\tfrac12$ and $\sum S\le\rho<2$ allows at most three, so $p\in\{1,2,3\}$.

$p\ge3$: $s_1,s_2,s_3>s_1-\omega>\tfrac12$. No pair fits together in a root
(sums $>1>\beta$, two-circle exact) and nobody nests
($s_j\le s_i-\omega\le s_1-\omega<s_j$ is absurd): the prefix
$\{s_1,s_2,s_3\}$ is already non-reinsertable, so
$\rho\ge\rho_{\mathrm{needed}}(\{s_1,s_2,s_3\})\ge\rho^*_3$.

$p=2$: $s_3,\dots,s_k\le s_1-\omega$. The placement ``$\{s_3,\dots,s_k\}$
in a row inside the hole of $s_1$ (capacity $s_1-\omega$), $s_1\to A$,
$s_2\to B$'' fails, and the only piece that can fail is the row:
$R:=\sum_{i\ge3}s_i>s_1-\omega$. With $s_2>s_1-\omega$, the tail
$Q=s_2+R>2(s_1-\omega)$ and $Q>1$ (Step 2): exactly the case-(iv) program
of Theorem~\ref{thm:Ctrio}, whose value is $\ge\rho^*_3$. (This is the only
biting branch: the dust of $(\le)$ lives here.)

$p=1$: $s_2\le s_1-\omega$. The placement ``$s_2$ in the hole of $s_1$,
$s_1\to B$, $\{s_3,\dots,s_k\}\to A$ in a row'' fails:
$R':=\sum_{i\ge3}s_i>1$. Then
$\sum S=s_1+s_2+R'>(s_2+\omega)+s_2+1=1+2s_2+\omega$ and the tail of $s_2$
gives $\rho\ge R'/s_2>1/s_2$, so
\[
\rho\ \ge\ \min_s\max\bigl(1+2s+\omega,\ 1/s\bigr)
=\tfrac{(1+\omega)+\sqrt{(1+\omega)^2+8}}{2}\ \ge\ 2,
\]
contradicting $\rho<2\beta<2$.
\end{proof}

\begin{corollary}[Exact combinatorial threshold]\label{cor:Comegac}
For $\omega<\omega_T=\tfrac1T-\tfrac12$ no profile of any size blocks the
exchange step with $\rho<T$; for $\omega>\omega_T$ the witness
$\{\tfrac12+\omega,\tfrac12+\varepsilon,\tfrac12+\varepsilon\}$ blocks with
$\rho$ arbitrarily close to $\tfrac{2}{1+2\omega}<T$. \qed
\end{corollary}

\subsection{The blocking frontier of the canonical trio}\label{app:frontier}

Work in the canonical template: the trio $\{\alpha,\sigma_1,\sigma_2\}$
with $0<\sigma_2\le\sigma_1\le1<\alpha$ must pack in the pan of radius
$R=\alpha+1$. With the half-angle machinery of Appendix~\ref{app:rigidproof}
(Lemma~\ref{lem:S1}), set $f(x)=x/(R-x)$,
$\theta(a,b)=2\arcsin\sqrt{f(a)f(b)}$ and
$F(\sigma_1,\sigma_2)=\theta(\alpha,\sigma_1)+\theta(\alpha,\sigma_2)
+\theta(\sigma_1,\sigma_2)$; by the constructive direction
(Lemma~\ref{lem:S2}), $F\le2\pi$ implies the trio packs, so
non-packability implies $F\ge2\pi$ on the closure. Each $\theta$ is
strictly increasing in both arguments, so $F=2\pi$ defines the
\emph{blocking frontier} $\sigma_2=h(\alpha,\sigma_1)$ where it meets the
band. Two evaluations used repeatedly: $f(\alpha)=\alpha$ (since
$R-\alpha=1$) and $1-\alpha f(\sigma)=R(1-\sigma)/(R-\sigma)$.

\begin{lemma}[Closed frontier]\label{lem:CF}
Let $\alpha>1$ and $\sigma:=1/\sqrt{\alpha(\alpha+1)}$. In the band
$0<\sigma_2\le\sigma_1\le1$,
\[
F(\sigma_1,\sigma_2)=2\pi\iff t(\sigma_1)+t(\sigma_2)=\sigma,
\qquad t(s)=\sqrt{\tfrac{1-s}{s}},
\]
and $\sigma=t(b(\alpha))$. Consequently
$h(\alpha,\sigma_1)=t^{-1}(\sigma-t(\sigma_1))$ with $t^{-1}(u)=1/(1+u^2)$,
defined and strictly decreasing exactly for $\sigma_1\in[s^*,1]$ where
$s^*=4\alpha(\alpha+1)/(2\alpha+1)^2$ is the diagonal point
($t(s^*)=\sigma/2$); the range of $h$ is $[b(\alpha),s^*]$, with
$h(\alpha,1)=b(\alpha)$.
\end{lemma}
\begin{proof}
On a frontier point all three angles are interior: for the pairs
$(\alpha,\sigma_i)$, $P=\alpha f(\sigma_i)<1$ since $\sigma_i<1$
(and $\alpha f(1)=1$); for $(\sigma_1,\sigma_2)$,
$\sigma_1+\sigma_2\le2<R$. Write $A=\theta(\alpha,\sigma_1)/2$,
$B=\theta(\alpha,\sigma_2)/2\in(0,\pi/2)$, so that
$\sin^2\!A=\alpha f(\sigma_1)$, $\cos^2\!A=R(1-\sigma_1)/(R-\sigma_1)$
(same for $B$ with $\sigma_2$), and on the frontier
$\theta(\sigma_1,\sigma_2)/2=\pi-A-B\in(0,\pi)$.

($\Rightarrow$) Taking sines,
$\sin\bigl(\theta(\sigma_1,\sigma_2)/2\bigr)=\sin(A+B)
=\sin A\cos B+\cos A\sin B$. The cross terms factor:
\[
(\sin A\cos B)^2=\alpha f(\sigma_1)\cdot\frac{R(1-\sigma_2)}{R-\sigma_2}
=\alpha R\,t(\sigma_2)^2\,f(\sigma_1)f(\sigma_2),
\]
because $(1-\sigma_2)/(R-\sigma_2)=t(\sigma_2)^2 f(\sigma_2)$; symmetrically
with $t(\sigma_1)$. Since
$\sin(\theta(\sigma_1,\sigma_2)/2)=\sqrt{f(\sigma_1)f(\sigma_2)}\neq0$,
dividing gives $1=\sqrt{\alpha R}\,\bigl(t(\sigma_1)+t(\sigma_2)\bigr)$,
the stated equation with $\sigma=1/\sqrt{\alpha R}$. The identity
$(1-b)/b=1/(\alpha(\alpha+1))$ for $b=b(\alpha)$ gives
$t(b(\alpha))=\sigma$.

($\Leftarrow$) For fixed $\sigma_1$, both $F(\sigma_1,\cdot)$ (increasing)
and $t(\sigma_1)+t(\cdot)$ (decreasing) are strictly monotone in
$\sigma_2$; by ($\Rightarrow$) their level sets coincide where the frontier
meets the band, and by monotonicity each has a unique solution.
Existence: on the diagonal, $F(s,s)$ increases from $0^+$ to
$F(1,1)>2\pi$, so there is a unique $s^\circ$ with $F=2\pi$, and
($\Rightarrow$) forces $t(s^\circ)=\sigma/2$, i.e.\ $s^\circ=s^*$. For
$\sigma_1\in[s^*,1]$, $F(\sigma_1,0^+)\le\pi<2\pi\le F(\sigma_1,\sigma_1)$:
the frontier crosses, and $t(\sigma_1)\le\sigma/2$ keeps
$h(\sigma_1)\le\sigma_1$. For $\sigma_1<s^*$ it would need
$\sigma_2>\sigma_1$. Endpoints: $h(s^*)=s^*$ and
$h(1)=t^{-1}(\sigma)=b(\alpha)$.
\end{proof}

\begin{theorem}[The $\kappa$ identity]\label{thm:Ckappa}
On the frontier, for every $\alpha>1$ and $\sigma_1<1$,
\[
\kappa:=-\frac{\partial h}{\partial\sigma_1}
=\sqrt{\frac{g(\sigma_2)}{g(\sigma_1)}},\qquad
g(s):=s^3(1-s),
\]
where $g$ is local to this subsection, and $\kappa\ge1$,
with equality iff $\sigma_1=\sigma_2=s^*$.
\end{theorem}
\begin{proof}
The identity: $t'(s)=-1/(2\sqrt{g(s)})$, and differentiating
$t(\sigma_1)+t(h(\sigma_1))=\sigma$ gives
$\kappa=t'(\sigma_1)/t'(\sigma_2)=\sqrt{g(\sigma_2)/g(\sigma_1)}$.

The inequality, in the $t$-coordinate: since $t$ reverses order,
$\sigma_2\le\sigma_1\iff t_2\ge t_1$ on the segment $t_1+t_2=\sigma$, and
the claim is $G(t_2)\ge G(t_1)$ for
\[
G(t):=g\bigl(1/(1+t^2)\bigr)=\frac{t^2}{(1+t^2)^4},\qquad
G'(t)=\frac{2t(1-3t^2)}{(1+t^2)^5}:
\]
$G$ increases on $[0,1/\sqrt3]$ and decreases after. Note
$\alpha>1\iff\alpha(\alpha+1)>2\iff\sigma<1/\sqrt2$. If
$\sigma\le1/\sqrt3$, both $t_i\le\sigma$ lie in the increasing range and
$t_2\ge t_1$ gives the claim, with equality only at $t_1=t_2$. If
$\sigma>1/\sqrt3$ and $t_2>1/\sqrt3$ (otherwise the previous case
applies), then $t_1=\sigma-t_2<1/\sqrt2-1/\sqrt3$, so
$G(t_1)\le G(1/\sqrt2-1/\sqrt3)=0.01574\dots$, while
$t_2\in(1/\sqrt3,\sigma)$ lies in the decreasing range, so
$G(t_2)>G(\sigma)>G(1/\sqrt2)=\tfrac{8}{81}$: here
$\kappa^2>\tfrac{8/81}{0.0158}>6$, far from tight.
\end{proof}

\begin{corollary}[Minimum of $\sigma_1+\sigma_2$; Tribonacci closure]
\label{cor:Cmin}
For every $\alpha>1$, the minimum of $\sigma_1+\sigma_2$ over the
region $\{F\ge2\pi,\ 0<\sigma_2\le\sigma_1\le1\}$ -- the closure of
the non-packable locus $F>2\pi$; on the boundary $F=2\pi$ the corona
closes with exact tangencies -- is
$1+b(\alpha)$, attained at $(\sigma_1,\sigma_2)=(1,b(\alpha))$
(equivalently, $1+b(\alpha)$ is the unattained infimum over
$F>2\pi$).
Consequently a blocking satisfying the witness wall
$\sigma_1+\sigma_2\le\alpha-\omega$ forces $1+b(\alpha)\le\alpha-\omega$,
i.e.\ $\alpha\ge T_{1+\omega}$ (the root of
$\alpha^3=(1+\omega)(\alpha^2+\alpha+1)$); in particular no blocking exists
for $\alpha\le T$, at any width.
\end{corollary}
\begin{proof}
The region is non-empty ($F(1,1)>2\pi$) and a minimum cannot have
$F>2\pi$ (decrease $\sigma_2$: the sum drops, $F$ stays $\ge2\pi$ by
continuity until the frontier is crossed, and
$F(\sigma_1,0^+)<2\pi$ guarantees it is crossed). On the frontier,
$d(\sigma_1+h)/d\sigma_1=1-\kappa\le0$ (Theorem~\ref{thm:Ckappa}): the
minimum sits at $\sigma_1=1$ and equals $1+b(\alpha)$. The sign of
$1+b(\alpha)-(\alpha-\omega)$ is governed by the polynomial identity
$1+b(\alpha)-\alpha=-(\alpha^3-\alpha^2-\alpha-1)/(\alpha^2+\alpha+1)$
and its deformation
$b(\alpha)-(\alpha-1-\omega)
=-\bigl[\alpha^3-(1+\omega)(\alpha^2+\alpha+1)\bigr]/(\alpha^2+\alpha+1)$,
whose numerator has the single positive root $T_{1+\omega}$ (Descartes'
rule); $T_1=T$.
\end{proof}

\subsection{Positive width: the lower-bound curve and the \texorpdfstring{$13/7$}{13/7} corner}
\label{app:corner}

The full blocking program of the canonical template adds the remaining
resources to the trio wall. With capacities $c_2:=1-\omega$ (the hole
$H_m$) and $c_4:=\alpha-\omega-1$ ($\sigma_2$ back in $u$ next to $m$,
two-circle exact), a realizable blocking satisfies (closures for infima)
\[
\text{(B1) trio non-packable},\quad
\text{(B2) }\sigma_2\ge c_2,\quad
\text{(B4) }\sigma_2\ge c_4,\quad
\text{(W) }\sigma_1+\sigma_2\le\alpha-\omega,
\]
plus the band $\sigma_2\le\sigma_1\le1$ and $\alpha-\omega\ge1$, and
$T_{\mathrm{can}}(\omega)$ is the infimum of
$\rho=\max\bigl(\sigma_1+\sigma_2,(1+\sigma_1+\sigma_2)/\alpha\bigr)$ over
blockings. (At $\omega=0$ the program is empty -- (B2) would force
$\sigma_2\ge1$ -- which is why the limit value $T$ uses only (B1)+(W).)

\begin{proposition}[Witness branch]\label{prop:Cwitness}
The infimum of $\rho$ over (B1)+(W) alone is
$\Phi(\omega)=T_{1+\omega}-\omega$, attained in the limit
$\alpha=T_{1+\omega}$, $\sigma_1\to1$, $\sigma_2\to b(\alpha)=
\alpha-1-\omega$. Moreover $\Phi$ is strictly increasing and concave, with
\[
\Phi'(\omega)=\frac{2\alpha+1}{\alpha^2(\alpha^2+2\alpha+3)}
\Big|_{\alpha=T_{1+\omega}}>0,\qquad
\Phi'(0)=\frac{2T+1}{7T^2+4T+3}=0.1374\dots,
\]
and on $[0,\tfrac17]$ the chord--tangent bounds
$T+(13-7T)\,\omega\le\Phi(\omega)\le T+\Phi'(0)\,\omega$ hold, the chord
endpoint being exact: $T_{8/7}=2$ and $\Phi(\tfrac17)=\tfrac{13}7$.
\end{proposition}
\begin{proof}
By Corollary~\ref{cor:Cmin}, blocking at a given $\alpha$ forces
$\rho\ge\sigma_1+\sigma_2\ge1+b(\alpha)$ and (W) forces
$\alpha\ge T_{1+\omega}$; $b$ is strictly increasing
($b'=(2\alpha+1)/(\alpha^2+\alpha+1)^2$), so the infimum over admissible
$\alpha$ is $1+b(T_{1+\omega})=T_{1+\omega}-\omega$ (the cubic identity
with equality), approached by $\sigma_1=1-\varepsilon$,
$\sigma_2=b(\alpha)+\varepsilon'$ -- a genuine blocking: the rigidity of
Proposition~\ref{prop:S5} at $\sigma_1=1$ and the propagation of
Lemma~\ref{lem:S6a}(3), written for the corner $\omega=\tfrac17$ in the
proof of Theorem~\ref{thm:corner}, apply verbatim at $\alpha=T_{1+\omega}$
for every $\omega\in[\tfrac17,0.30]$ ($\alpha\ge2$ keeps the disk
diametrally full). The derivative: implicit
differentiation of $\alpha^3=(1+\omega)(\alpha^2+\alpha+1)$ and
substitution of $1+\omega=\alpha^3/(\alpha^2+\alpha+1)$ reduce the
numerator of $d\alpha/d\omega-1$ to $2\alpha+1$ (symbolic identity).
Concavity: $\frac{d}{d\alpha}\log\Phi'=\frac{2}{2\alpha+1}-\frac2\alpha
-\frac{2\alpha+2}{\alpha^2+2\alpha+3}<0$ and $\alpha$ increases with
$\omega$. $T_{8/7}=2$ because $2^3=8=\tfrac87\cdot7$.
\end{proof}

Adding (B2) gives the second branch: every full blocking has
$\rho\ge\sigma_1+\sigma_2\ge2(1-\omega)$. The maximum of the two branches
is minimized at their crossing $T_{1+\omega}=2-\omega$, i.e.\ at the root
$\omega_\times=0.0754315\dots$ of $2\omega^3-10\omega^2+14\omega-1$, giving
the uniform floor
$T_{\mathrm{can}}(\omega)\ge2(\alpha_\times-1)=T+0.00985\dots$ for
$\omega\in(0,0.30]$ -- already enough for ``width only raises the floor''.
The full curve is finer. For fixed $\alpha$, pass to the coordinate
$t_i=t(\sigma_i)$, where (B1) becomes \emph{linear}
(Lemma~\ref{lem:CF}): $t_1+t_2\le\Theta(\alpha):=t(b(\alpha))$, and the
walls read $t_2\le\tau_c:=t(c)$, $c:=\max(c_2,c_4)$. If $\alpha>2+\omega$
there is no blocking at all ((B4) and the band give
$\sigma_1+\sigma_2\ge2c_4>\alpha-\omega$, violating (W)); so
$\alpha-\omega\in[1,2]$ effectively.

\begin{lemma}[Conditional minimum]\label{lem:CE1}
Let $\alpha\in(1,2+\omega)$ and $c=\max(1-\omega,\alpha-\omega-1)<1$. The
infimum of $S_0=\sigma_1+\sigma_2$ over non-packable trios with
$\sigma_2\ge c$ and $\sigma_2\le\sigma_1\le1$ is
\[
\text{(a) } 1+b(\alpha)\ \ \text{if } c\le b(\alpha);\quad
\text{(b) } c+t^{-1}\bigl(\Theta(\alpha)-t(c)\bigr)\ \ \text{if }
b(\alpha)<c\le s^*(\alpha);\quad
\text{(c) } 2c\ \ \text{if } c>s^*(\alpha),
\]
where $t^{-1}(u)=1/(1+u^2)$ as in Lemma~\ref{lem:CF}. Moreover each case
is non-decreasing in $\alpha$ on its domain.
\end{lemma}
\begin{proof}
$t^{-1}(\cdot)$ is decreasing, so $S_0$ drops when either $t_i$ rises. In
the truncated box $\{t_1\le t_2\le\tau_c,\ t_1+t_2\le\Theta\}$: if
$2\tau_c\le\Theta$ (case (c)) the box corner $(\tau_c,\tau_c)$ is admissible and
optimal, $S_0=2\,t^{-1}(\tau_c)=2c$. Otherwise the minimum lives on the segment
$t_1+t_2=\Theta$, where
\[
\frac{dS_0}{dt_2}=-(t^{-1})'(\Theta-t_2)+(t^{-1})'(t_2)
=2\bigl[\sqrt{G(t_1)}-\sqrt{G(t_2)}\bigr]\le0
\]
by the $G$-lemma inside Theorem~\ref{thm:Ckappa}
($(t^{-1})'(x)=-2x/(1+x^2)^2$, $\bigl((t^{-1})'/2\bigr)^2=G$): push $t_2$
to its cap.
If $\tau_c\ge\Theta$ (case (a)): $t_2=\Theta$, $t_1=0$, i.e.\
$(\sigma_1,\sigma_2)=(1,b(\alpha))$. If $\tau_c<\Theta$ (case (b)):
$t_2=\tau_c$, $t_1=\Theta-\tau_c$. Monotonicity in $\alpha$: (a) is
$1+b(\alpha)$ with $b$ increasing; (b) with $c=c_2$ fixed has
$dS_0/d\alpha=(t^{-1})'(\Theta-\tau_c)\,\Theta'(\alpha)\ge0$ ($(t^{-1})'<0$,
$\Theta'<0$); (b)/(c) with the moving cap $c=c_4$ ($\alpha\ge2$,
$\tau_4=t(c_4)$) has
$dS_0/d\alpha=\tau_4'\,(-2)\bigl[\sqrt{G(\tau_4)}-\sqrt{G(t_1)}\bigr]
+(t^{-1})'(t_1)\,\Theta'\ge0$, since $\tau_4'<0$, $G(\tau_4)\ge G(t_1)$
($\tau_4\ge t_1$ on the segment, same $G$-lemma) and the last two factors
are both negative.
\end{proof}

\begin{proposition}[The lower-bound curve; exact on the witness
branch]\label{prop:Ccurve}
For $\omega\in(0,0.30]$ (the cap $0.30$ is a convenient working
restriction of the statement, not a structural obstruction:
$T_{1+\omega}<2+\omega$ holds identically), the proved lower bound for
$T_{\mathrm{can}}$ is
\[
T_{\mathrm{can}}(\omega)\ \ge\
\begin{cases}
2(1-\omega) & \omega\in(0,\omega_1],\\[2pt]
\alpha_m(\omega)-\omega & \omega\in[\omega_1,\tfrac17],\\[2pt]
\Phi(\omega) & \omega\in[\tfrac17,0.30],
\end{cases}
\]
where $\omega_1=0.0413570\dots$ is the root in
$(\tfrac1{25},\tfrac1{14})$ of $4\omega^3-20\omega^2+25\omega-1$ and
$\alpha_m(\omega)\in(1,2]$ is the unique root of
$t(\alpha-1)+t(1-\omega)=\Theta(\alpha)$, an algebraic function of degree
$6$ (its polynomial $P(\alpha,\omega)$, of bidegree $(6,2)$, is
irreducible over $\mathbb Q[\alpha,\omega]$). Equality holds on the witness
branch (the family is genuine as in Proposition~\ref{prop:Cwitness}) and at
the corner; on the other two branches the matching upper
bound inherits the angular-criterion caveat.
\end{proposition}
\begin{proof}
Fix $\omega$ and minimize the value of Lemma~\ref{lem:CE1} over $\alpha$
subject to (W).

\emph{$H_m$ branch ($\omega\le\omega_1$).} Case (c) gives $S=2(1-\omega)$,
constant in $\alpha$, admissible iff (W): $\alpha\ge2-\omega$, and iff
$c_2>s^*(\alpha)$, i.e.\ $\alpha\le\bar\alpha(\omega)$. The window is
non-empty iff $s^*(2-\omega)\le1-\omega$, and
$(1-\omega)-s^*(2-\omega)$ has numerator
$-(4\omega^3-20\omega^2+25\omega-1)$; the cubic is strictly increasing on
$(0,\tfrac56)$ (derivative $(5-2\omega)(5-6\omega)$), so the window exists
exactly for $\omega\le\omega_1$. The other cases do not undercut
$2(1-\omega)$ (monotonicity in Lemma~\ref{lem:CE1} plus continuity across
case junctions).

\emph{Mixed branch ($\omega_1<\omega\le\tfrac17$).} For $\alpha\le2$, case
(c) is empty (it would need $1-\omega\ge s^*(\alpha)\ge s^*(2-\omega)$,
against the cubic) and case (a) needs $b(\alpha)\ge1-\omega$, impossible
($b(2)=\tfrac67\le1-\omega$). Case (b) with $c=c_2$ has value
$S_0^b(\alpha)=1-\omega+t^{-1}(\Theta-t(1-\omega))$, increasing in $\alpha$,
and (W) is equivalent to
$\Xi(\alpha):=\Theta(\alpha)-t(\alpha-1)\ge t(1-\omega)$ with $\Xi$
strictly increasing on $(1,2]$ (the claim reduces to the polynomial
inequality $(\alpha(\alpha+1))^3>(2\alpha+1)^2(\alpha-1)^3(2-\alpha)$,
which has no root in $(1,2)$), $\Xi(1^+)=-\infty$ and
$\Xi(2)=t(\tfrac67)\ge t(1-\omega)\iff\omega\le\tfrac17$: the admissible
$\alpha$ form $[\alpha_m,2]$ and the minimum is
$S_0^b(\alpha_m)=\alpha_m-\omega$ ((W) with equality). For
$\alpha\in(2,2+\omega)$: if $b(\alpha)\le1-\omega$, relax (B4) to (B2) and
use monotonicity, $S_0\ge S_0^b(\alpha)>S_0^b(\alpha_m)$; if $b(\alpha)>1-\omega$,
$S_0\ge1+b(\alpha)>2-\omega>\alpha_m-\omega$. Squaring the defining equation
twice yields the polynomial $P$ (necessary condition; the right branch is
pinned by $\alpha_m\in(1,2]$), with the junction identities
$P(2-\omega,\omega)=(\omega-1)^3(4\omega^3-20\omega^2+25\omega-1)$ (so
$\alpha_m(\omega_1)=2-\omega_1$) and $\alpha_m(\tfrac17)=2$.

\emph{Witness branch ($\omega\ge\tfrac17$).} Now
$b(T_{1+\omega})=T_{1+\omega}-1-\omega\ge1-\omega\iff T_{1+\omega}\ge2
\iff\omega\ge\tfrac17$: the optimum of case (a) satisfies (B2), (W) forces
$\alpha\ge T_{1+\omega}$ (Corollary~\ref{cor:Cmin}) and the minimum is
$\Phi(\omega)$; larger $\alpha$ and the cases (b)/(c$_4$) are increasing
and start at $\Phi(\omega)$; smaller $\alpha$ admit no blocking (for
$\alpha\le2$, $\Xi(\alpha)\le\Xi(2)=t(\tfrac67)\le t(1-\omega)$ with
equality only at the corner; case (c) would need $\omega\le\omega_1$).
\end{proof}

\begin{proof}[Proof of Theorem~\ref{thm:corner}]
\emph{Lower bound.} Take any blocking at width $\omega$. If $\alpha\ge2$:
by Corollary~\ref{cor:Cmin} alone (no (B2), (B4), (W) needed),
$\rho\ge1+b(\alpha)\ge1+b(2)=\tfrac{13}7$. If $\alpha<2$: for
$\omega\ge\tfrac17$, by Corollary~\ref{cor:Cmin} and (W) blocking
forces $\alpha\ge T_{1+\omega}\ge T_{8/7}=2$, contradicting $\alpha<2$:
empty.
For $\omega\le\omega_1$: $\rho\ge2(1-\omega)\ge2(1-\omega_1)>\tfrac{13}7$
(since $\omega_1<\tfrac1{14}$). For $\omega_1<\omega<\tfrac17$:
$\rho\ge V(\omega):=\alpha_m(\omega)-\omega$, and $V>\tfrac{13}7$ there:
if $V(\omega')=\tfrac{13}7$ then $(\alpha,\omega)=(\tfrac{13}7+\omega',
\omega')$ satisfies $P=0$, but
\[
P\bigl(\tfrac{13}7+\omega,\ \omega\bigr)
=\frac{4\,(7\omega-1)\,Q_5(\omega)}{7^6},
\]
\[
Q_5=33614\,\omega^5+235298\,\omega^4+620830\,\omega^3+766066\,\omega^2
+397831\,\omega-289,
\]
and $Q_5$ has no root in $[\tfrac1{25},\tfrac17]$ (exact root isolation;
its only positive root is $\approx0.0007$), so the only root of the
product there is $\omega=\tfrac17$; since
$V(\omega_1)-\tfrac{13}7=\tfrac17-2\omega_1>0$ and $V$ is continuous,
$V>\tfrac{13}7$ on all of $[\omega_1,\tfrac17)$. Uniqueness of the
infimum configuration: equality in the branch $\alpha\ge2$ forces
$\alpha=2$, $\sigma_1=1$, $\sigma_2=\tfrac67$; then (B2) gives
$\omega\ge\tfrac17$ and (W) gives $\omega\le\tfrac17$.

\emph{The approximating family is genuine.} At exactly $\alpha=2$,
$\omega=\tfrac17$ there is no family: $\kappa>1$ gives
$\sigma_1+h(2,\sigma_1)>\tfrac{13}7=\alpha-\omega$ for $\sigma_1<1$,
violating (W); the corner is approached by \emph{opening} (B4). Fix
$\omega=\tfrac17$ and let $\delta\downarrow0$:
\[
\alpha=2+\delta,\quad\varepsilon=\delta^2/4,\quad
\sigma_1=1-\varepsilon,\quad\sigma_2=(\alpha-\omega-1)+\varepsilon/2.
\]
All walls hold ((B4) with slack $\varepsilon/2$; (B2) since
$\alpha-\omega-1\ge1-\omega$ for $\alpha\ge2$; (W) since
$\sigma_1+\sigma_2=\alpha-\omega-\varepsilon/2$). Genuineness: since
$\alpha>T_{8/7}=2$, Corollary~\ref{cor:Cmin} gives
$\sigma_2>b(\alpha)$ strictly, so the trio $\{\alpha,1,\sigma_2\}$ is
non-packable by the \emph{rigidity} of Proposition~\ref{prop:S5}
(the disk $R=\alpha+1$ is diametrically full; rescale by $t=1/\alpha$) --
not merely by the angular criterion -- and the infeasibility propagates
from $\sigma_1=1$ to $\sigma_1=1-\varepsilon$ by
Lemma~\ref{lem:S6a}(3) for all small $\varepsilon$ (shrinking
$\varepsilon$ below the propagation radius $\delta_0$ of that lemma if
needed; any smaller $\varepsilon$ serves). Then
$\rho=\alpha-\omega-\varepsilon/2\to\tfrac{13}7$.
\end{proof}

\begin{remark}[Fine structure]\label{rem:Cbump}
The curve is not monotone on $(0,\tfrac17]$: differentiating the defining
equation of $\alpha_m$ at the junction (where $\alpha=2-\omega_1$, so the
leading terms $(\alpha-1)^3(2-\alpha)$ and $(1-\omega_1)^3\omega_1$ agree
by direct substitution) gives
$V'(\omega_1^+)>0$ (numerically $+0.0721$, by implicit differentiation of
$P$): $\omega_1$ is a
\emph{local minimum} of $T_{\mathrm{can}}$, followed by a local maximum at
$\omega_{\mathrm{peak}}=0.0444700$ (the unique root in
$(\omega_1,\tfrac17)$ -- by Sturm -- of an explicit degree-8 factor $R_8$
of the resultant of $P$ and $P_\alpha+P_\omega$), a bump of height
$1.1\cdot10^{-4}$, before descending to the corner. $T_{\mathrm{can}}$ has
two local minima, $\omega_1$ and $\tfrac17$, and the global one is the
corner. (The computational claims of this appendix -- the exact root
isolation of $Q_5$, the degree-8 factor $R_8$ and the bump height -- are
reproduced by \texttt{code/esquina.py} and \texttt{code/grosor.py}.)
\end{remark}

\section{Complete proofs for generic containers}\label{app:genericproofs}

This appendix proves every statement of Section~\ref{sec:generic}; its
final subsection (\ref{app:pan-app}) proves the golden floor of the pan
exchange for pair profiles, in full.
Normalize $r_m=1$, $\omega=w/r_m$. The template: $v$ is the pan, of radius
$R$, containing $\{\alpha\}\cup O\cup\{m\}$ with $O=\{o_1\ge\dots\ge o_j\}$,
$j\ge0$, $o_i\ge m=1$ (rings larger than $m$ agree in $F$ and $P$); $u$ is
the hole of $\alpha$, capacity $\alpha-\omega\ge1$; the witness placed
$S=\{\sigma_1\ge\sigma_2\}$ in $u$, $\sigma_2\le\sigma_1\le1$. Holes may be
occupied arbitrarily: contents of a hole at any depth are parametrized by
$X_y:=\sum\mathrm{children}(y)$. The exchange builds a placement $P'$ that
preserves the container assignments of all rings $\ge m$; \emph{positions}
inside each container are existential (a placement is a nesting forest
whose sibling groups pack). A \emph{node} is any ring of radius $\ge1$
that is an occupant of $v$ other than $\alpha$ \emph{and} $m$, or is
nested (at any depth) inside such an occupant; nodes and their descendants live outside
$m$. Unless stated otherwise, blocked-exchange walls are closures of
strict inequalities, as in Appendix~\ref{app:widthproofs}. Throughout
this appendix and Appendix~\ref{app:campaign} we use the
\emph{first-copy convention}: ties among equal radii are broken by index,
so the tail of a ring collects all later copies.

\subsection{The universal frontier and the corona criterion}
\label{app:corona-app}

\begin{lemma}[Universal frontier]\label{lem:DU}
Let $\{A,x,y\}$ be wall-tangent in a disk of arbitrary radius $R$, in the
interior domain $A+x,A+y,x+y<R$. Set $c:=R-A$,
$T_c(x):=\sqrt{(c-x)/x}$ (a function symbol local to the corona
criterion, unrelated to the deformed Tribonacci root $T_{1+\omega}$ of
Appendix~\ref{app:widthproofs}) and $\tau_R:=c/\sqrt{AR}$. If $A\ge\min(x,y)$,
then
\[
F=\theta(A,x)+\theta(A,y)+\theta(x,y)\ \ge\ 2\pi
\iff T_c(x)+T_c(y)\ \le\ \tau_R .
\]
The direction $(\Rightarrow)$ is unconditional; the converse fails without
the hypothesis ($R=1$, $A=0.01$, $x=y=0.45$ is a counterexample).
\end{lemma}
\begin{proof}
The cross term factors for arbitrary $R$: with
$\sin^2(\theta_{Az}/2)=f(A)f(z)$, $f(z)=z/(R-z)$, one has the identity
\[
f(A)f(x)\bigl(1-f(A)f(y)\bigr)
=\frac{AR}{(R-A)^2}\cdot\frac{c-y}{y}\cdot f(x)f(y),
\]
so
$\sin(\theta_{Ax}/2)\cos(\theta_{Ay}/2)
=\tfrac{\sqrt{AR}}{c}\,T_c(y)\sqrt{f(x)f(y)}$ and, adding the symmetric
term, $\sin s=\tfrac{\sqrt{AR}}{c}\bigl(T_c(x)+T_c(y)\bigr)\sin w$ with
$s=(\theta_{Ax}+\theta_{Ay})/2$, $w=\theta_{xy}/2$. The linear form is
thus $\sin s\le\sin w$, while $F\ge2\pi$ is $s+w\ge\pi$. Since
$w\le\pi/2$: $s+w\ge\pi\Rightarrow\sin s\le\sin w$ always; the converse
needs $s>w$, which the hypothesis provides: if $A\ge y$, monotonicity
gives $\theta(A,x)\ge\theta(y,x)=2w$, so $s>w$.
\end{proof}

Uniform companions (same computations as in
Appendix~\ref{app:widthproofs}, now for general $R$): the general pocket
$b_R(A)=ARc/(AR+c^2)$, equal to Descartes' $b_2$ when $R=A+B$, increasing
in $R$ and always $\le A$; and the slope identity
$\kappa=\sqrt{g_c(y)/g_c(x)}$ with $g_c(s)=s^3(c-s)$, since $T_c$ is a
primitive of $-c/(2\sqrt{g_c})$.

\begin{lemma}[Gap system]\label{lem:Dgaps}
A \emph{corona} of $\{a_1,\dots,a_k\}$ in the disk $R$ (all $k$ circles
wall-tangent; all pairs admissible, $a_i+a_j\le R$, so that every
$\theta_{ij}$ is defined) with cyclic order $\pi$ exists iff the linear system
\[
\theta_{ij}\ \le\ S_j-S_i\ \le\ 2\pi-\theta_{ij}
\qquad(1\le i<j\le k,\ S_1=0)
\]
is feasible, where $\theta_{ij}$ is the angle of the pair in positions
$i,j$ of $\pi$.
\end{lemma}
\begin{proof}
The angular separation of positions $i<j$ is
$\min(S_j-S_i,\,2\pi-(S_j-S_i))$, and
$\min(\Delta,2\pi-\Delta)\ge\theta\iff\theta\le\Delta\le2\pi-\theta$
because $\theta\le\pi$; wall-tangent circles are contained in the disk and
pairwise disjointness is exactly the angular condition.
\end{proof}

\begin{lemma}[Subset certificates]\label{lem:Dsubset}
If $\{a_1,\dots,a_k\}$ has a corona with cyclic order $\pi$, then every
subset $T$, $|T|\ge3$, satisfies
$\sum_{\text{consecutive pairs of }T}\theta\le2\pi$ in the induced cyclic
order. Consequently, if some $T$ exceeds $2\pi$ in \emph{all} its cyclic
orders, no corona exists in any order -- the unconditional direction used
by all lower bounds. \qed
\end{lemma}

The naive criterion (``some order with consecutive sum $\le2\pi$ implies a
corona'') is \emph{false} for $k\ge4$: the conditional triangle inequality
$\theta(a,c)\le\theta(a,b)+\theta(b,c)$ holds when $b\ge\min(a,c)$ but
fails badly for small $b$ ($R=1$, $a=c=0.45$, $b\to0$), and
$\{0.47,0.47,0.47,0.02\}$ has a small cyclic sum while its top trio is
already infeasible. The correct statement:

\begin{theorem}[Exact criterion per order, $k=4$]\label{thm:Dk4}
For arbitrary symmetric $\theta_{ij}\in(0,\pi]$, the system of
Lemma~\ref{lem:Dgaps} for the order $(1,2,3,4)$ is feasible iff the four
trio certificates and the total
$\theta_{12}+\theta_{23}+\theta_{34}+\theta_{14}\le2\pi$ hold.
\end{theorem}
\begin{proof}
Necessity: Lemma~\ref{lem:Dsubset}. Sufficiency, by difference-constraint
duality: each upper bound is an up-edge $i\to j$ of weight
$2\pi-\theta_{ij}$, each lower bound a down-edge $j\to i$ of weight
$-\theta_{ij}$; the system is feasible iff the digraph has no negative
directed cycle, and it suffices to exclude \emph{simple} cycles. A simple
cycle $\Gamma$ with $U$ up-edges has weight
$w(\Gamma)=2\pi U-\sum_{e\in\Gamma}\theta_e$, so $w<0$ requires
$\#\mathrm{edges}(\Gamma)>2U$ (as $\theta\le\pi$). With $k=4$, a simple
cycle has $\le4$ edges, so only $U=1$ can go negative; a $U=1$ cycle is an
up-edge $i\to j$ plus a decreasing chain from $j$ to $i$, and its
negativity condition $\theta_{ij}+\sum_{\text{chain}}\theta>2\pi$ is
exactly the subset certificate of $\{i,\text{chain},j\}$: enumerating,
sizes $3$ give the four trios and size $4$ gives the total.
\end{proof}

\begin{proposition}[The zigzag minimizes the total]\label{prop:Dzigzag}
For $a_1\ge a_2\ge a_3\ge a_4$, the three cyclic orders have totals
$\sum_{i<j}\theta_{ij}$ minus one perfect matching, and the maximum-sum
matching is $\theta_{12}+\theta_{34}$; hence
$\min_\pi\mathrm{TOT}(\pi)=\sum_{i<j}\theta_{ij}-\theta_{12}-\theta_{34}$,
achieved by the zigzag order $(a_1,a_3,a_2,a_4)$.
\end{proposition}
\begin{proof}
With $\sigma_i:=\log f(a_i)$ (decreasing) one has
$\theta(a_i,a_j)=g(\sigma_i+\sigma_j)$ for $g(s)=2\arcsin(e^{s/2})$ on
$s\le0$ (a notation local to this proof), with
$g''(s)=e^{-s}/\bigl(2(e^{-s}-1)^{3/2}\bigr)>0$: $g$ is convex and
increasing (continuity at $s=0$ by passage to the limit). The three
matchings have pair-sums $\{\sigma_1+\sigma_2,\sigma_3+\sigma_4\}$,
$\{\sigma_1+\sigma_3,\sigma_2+\sigma_4\}$,
$\{\sigma_1+\sigma_4,\sigma_2+\sigma_3\}$ with equal totals, and the first
majorizes the other two; convexity turns majorization into the claimed
maximality.
\end{proof}

\begin{theorem}[The $U_4$ criterion: global, two inequalities]
\label{thm:DU4}
$\{a_1\ge a_2\ge a_3\ge a_4\}$ admits a corona in the disk $R$ (all pairs
admissible) iff
\[
\theta_{12}+\theta_{13}+\theta_{23}\le2\pi
\qquad\text{and}\qquad
\sum_{i<j}\theta_{ij}-\theta_{12}-\theta_{34}\le2\pi .
\]
\end{theorem}
\begin{proof}
Trio certificates do not depend on the order, so a corona exists iff all
trios pass and the minimal total does
(Theorem~\ref{thm:Dk4}~+~Proposition~\ref{prop:Dzigzag}). By monotonicity
of $\theta$, every trio is dominated pairwise by the top trio
$\{a_1,a_2,a_3\}$, so one trio inequality suffices. Neither inequality is
redundant (both are active on the frontier). The top-trio condition
linearizes by Lemma~\ref{lem:DU} with automatic hypothesis
($A=a_1$ is the global maximum): it reads
$T_c(a_2)+T_c(a_3)\ge\tau_R$ with $c=R-a_1$.
\end{proof}

For $k=5$ the subset-only criterion fails for exactly one pattern, the
\emph{pentagram}: of the $84$ simple cycles, all with $U\ge2$ are
count-dominated or LP-dominated except the cycle through the five
diagonals, and adding its certificate $\sum_D\theta\le4\pi$ restores an
exact criterion for arbitrary $\theta$; for \emph{geometric} $\theta$ --
separation matrices with $\sin^2(\theta_{ij}/2)=f_i f_j$ -- the identity
$\prod_{(i,j)\in D}f_if_j=\prod_{(i,j)\in S}f_if_j=(\prod_i f_i)^2$
(over the diagonal set $D$ and the side set $S$ of the cycle) obstructs the
pattern and the best value found is $7.556$ against the $4\pi=12.57$
needed (the $84$-cycle enumeration and the dominance checks are
reproduced by script \texttt{code/corona.py}, a computational check); its redundancy for geometric $\theta$ is stated as a conjecture in
the repository and is \emph{not used}: lower bounds only ever use the
unconditional direction of Lemma~\ref{lem:Dsubset}. For $k\ge5$ the
subset quantifier does not commute with the order quantifier (explicit
five-ring counterexamples in the repository), so the per-order criterion
or the LP is the correct tool.

\subsection{Occupant walls and metallic floors}\label{app:walls-app}

\begin{lemma}[Walls, free-hole template]\label{lem:DV1}
Suppose the holes of $m$, $\sigma_1$ and all $o_i$ are free. If the
exchange is blocked, then
\[
\text{(B2) }\sigma_2>1-\omega;\quad
\text{(B3) }\sigma_2>\sigma_1-\omega;\quad
\text{(B4) }\sigma_2>\alpha-\omega-1;\quad
\text{(Bo) }\sigma_2>o_k-\omega\ \forall k;
\]
\[
\text{(D) }\sigma_1+\sigma_2>1;\qquad
\text{(W) }\sigma_1+\sigma_2\le\alpha-\omega.
\]
In particular $o_k<\sigma_2+\omega\le1+\omega$ for all $k$: every extra
occupant stays within one width of $m$.
\end{lemma}
\begin{proof}
Each line is the contrapositive of an explicit placement whose two
resources are disjoint and whose feasibility criteria (pair, nesting, row)
are exact: $\sigma_2\to H_m$ with $\sigma_1\to D_m$; $\sigma_2\subset
\sigma_1\to D_m$; $\sigma_2$ next to $m$ in $u$ ($1+\sigma_2\le
\alpha-\omega$, pair-exact) with $\sigma_1\to D_m$; $\sigma_2\subset o_k$
with $\sigma_1\to D_m$; the whole pair into $D_m$ (a free unit disk). (W)
is the legality of the witness's placement.
\end{proof}

\begin{proposition}[Occupant price]\label{prop:DV2}
Blocking in the free-hole template with $j\ge1$ extra occupants implies
\[
\rho\ >\ 2+\frac{j-2\omega}{o_1}\ \ge\ \frac{j+2}{1+\omega}
\quad(\omega\le j/2),
\qquad\text{and for }j=1,\ \omega\ge\tfrac12:\quad
\rho>\frac{4}{1+2\omega}.
\]
\end{proposition}
\begin{proof}
The tail of $o_1$ contains $\{o_2,\dots,o_j,m,\sigma_1,\sigma_2\}$ (the
first-copy convention of the preamble). By (Bo), $\sigma_1\ge\sigma_2>o_1-\omega$,
and each $o_i\ge1$:
$\rho\ge\bigl((j-1)+1+\sigma_1+\sigma_2\bigr)/o_1
>\bigl(j+2(o_1-\omega)\bigr)/o_1=2+(j-2\omega)/o_1$. If $\omega\le j/2$
the expression decreases in $o_1$ and $o_1<1+\omega$
(Lemma~\ref{lem:DV1}). For $j=1<2\omega$, add (D):
$\sigma_1+\sigma_2>\max(1,2(o_1-\omega))$, so
$\rho>\bigl(1+\max(1,2(o_1-\omega))\bigr)/o_1$, whose infimum over $o_1$
is $4/(1+2\omega)$ at $o_1=\tfrac12+\omega$.
\end{proof}

\begin{corollary}[Crossings; canonical optimality]\label{cor:DV34}
For $j=1$: $\rho>3/(1+\omega)>T$ for
$\omega\le\tfrac12$ (the range where Proposition~\ref{prop:DV2}
provides that branch; on it $3/(1+\omega)>T$ holds throughout since
$\tfrac12<3/T-1=0.6310\dots$), and combining with the fine branch
$\rho>4/(1+2\omega)$ for $\omega\ge\tfrac12$,
$\rho>T$ for all $\omega<2/T-\tfrac12=2T^2-2T-\tfrac52=0.5873\dots$, with
$\rho\ge\tfrac{13}7$ for $\omega\le\tfrac{15}{26}$. For $j\ge2$:
$\rho>4/(1+\omega)>2>T$ for all $\omega<1$. Moreover every
extra-occupant blocking has $\rho>2$ for $\omega<\tfrac12$, while the
entire canonical curve lies strictly below $2$ -- the witness branch by
the exact identity
$(2+\omega)^3-(1+\omega)\bigl((2+\omega)^2+(2+\omega)+1\bigr)=1>0$, i.e.\
$T_{1+\omega}<2+\omega$, hence $\Phi(\omega)<2$ -- so on
$\omega<\tfrac12$ adding occupants is
strictly suboptimal for the adversary and the canonical template is
the optimal $v$ in this class (for $j=1$ and
$\omega\ge\tfrac12$ the proven bound $4/(1+2\omega)$ may dip below
the canonical curve, so there the comparison is not settled by
these inequalities; for $j\ge2$ the bound $4/(1+\omega)>2$ covers
all widths). \qed
\end{corollary}

\subsection{Occupied holes: the opposite-disk lemma and the silver floor}
\label{app:blockers-app}

\begin{lemma}[Opposite disk]\label{lem:DR}
Let a disk have capacity $c$, let $\sigma\le c$ be the largest piece and
$C$ a multiset of pieces $\le\sigma$. If $\sum C\le c-\sigma$, then
$\{\sigma\}\cup C$ packs in the disk.
\end{lemma}
\begin{proof}
Place $\sigma$ wall-tangent with center $(-(c-\sigma),0)$. The disk of
radius $r=c-\sigma$ centered at $(c-r,0)$ is wall-tangent and exactly
tangent to $\sigma$ (center distance $(c-\sigma)+(c-r)=c=\sigma+r$): a
free disk of radius $c-\sigma$, in which $C$ fits in a row
(Lemma~\ref{lem:row}).
\end{proof}

\begin{lemma}[General blocking wall Bo$''$]\label{lem:DBo}
Blocking implies, for every node $y$: $y<\sigma_2+\omega+X_y$.
\end{lemma}
\begin{proof}
Otherwise $\sum\mathrm{children}(y)\le(y-\omega)-\sigma_2$ and
Lemma~\ref{lem:DR} packs $\{\sigma_2\}\cup\mathrm{children}(y)$ in the
hole of $y$ (if some child exceeds $\sigma_2$, apply the lemma with that
child as the largest piece; the inequality is the same), with
$\sigma_1\to D_m$: unblocked. Positions are re-eligible, and inserting
$\sigma_2$ (a ring $<m$) only changes the container of $\sigma_2$: legal.
\end{proof}

\begin{theorem}[Silver floor; $m$ without children]\label{thm:DB}
Blocking with $j\ge1$, holes occupied arbitrarily at any depth, $m$
childless, implies
$\rho>\Psi(\omega)=(1-\omega)+\sqrt{(1-\omega)^2+1}$, the positive root of
$u^2-2(1-\omega)u-1$. In particular $\Psi(0)=1+\sqrt2$ (the silver ratio),
$\Psi(\tfrac14)=2$, and $\Psi>T\iff\omega<(T-1)^2/2=0.35220\dots$
\end{theorem}
\begin{proof}
Let $y^*$ be a node of minimal radius; its children are all $<1$ (a child
$\ge1$ would be a smaller node), so $X:=X_{y^*}$ is a sum of rings smaller
than $m$, living in the tail of $m$. With $\sigma:=\sigma_2>1-\omega$
(wall (B2), valid since $H_m$ is free): the tail of $m$ gives
$\rho\ge\sigma_1+\sigma_2+X\ge2\sigma+X$, and the tail of $y^*$ with
Lemma~\ref{lem:DBo} gives
$\rho\ge(1+\sigma_1+\sigma_2+X)/y^*>(1+2\sigma+X)/(\sigma+\omega+X)$.
Minimize $F(\sigma,X)=\max\bigl(2\sigma+X,\
(1+2\sigma+X)/(\sigma+\omega+X)\bigr)$ over $\sigma\ge1-\omega$, $X\ge0$:
for fixed $\sigma$ the first term grows in $X$ and the second decreases,
so $F(\sigma,\cdot)$ is bounded below by the crossing value
$u=2\sigma+X$ with $u(u+\omega-\sigma)=1+u$, i.e.\ the positive root of
$u^2-(1+\sigma-\omega)u-1=0$ (when the crossing falls at $X<0$, the same
algebra shows $F(\sigma,0)=2\sigma$ already exceeds the root). The root
increases in $\sigma$, so the program minimum is at $\sigma=1-\omega$,
giving $\Psi(\omega)$.
\end{proof}

\begin{theorem}[Evacuation dichotomy; $m$ with children]\label{thm:DBpp}
The conclusion of Theorem~\ref{thm:DB} holds without the hypothesis that
$m$ is childless.
\end{theorem}
\begin{proof}
Evacuation to $D_m$: place $\sigma_1$ and all children of $m$ in a row
inside $D_m$ (possible if $\sigma_1+M\le1$, $M:=\sum\mathrm{children}(m)$;
children of $\sigma_1$ travel inside $\sigma_1$) and $\sigma_2$ in the
emptied $H_m$ (possible if $\sigma_2\le1-\omega$). Contrapositive:
blocking implies $\sigma_2>1-\omega$ \emph{or} $\sigma_1+M>1$. In branch A
the proof of Theorem~\ref{thm:DB} applies verbatim (it used only
$\sigma_2>1-\omega$, Bo$''$ and the two tails; the mass $M\ge0$ only
fattens them). In branch B, let $y^*$ be the minimal node,
$s:=\sigma_2+X>1-\omega$ (Bo$''$ and $y^*\ge1$) and
$A:=\sigma_1+\sigma_2+M+X>1+s$; the two tails give $\rho\ge A>1+s$ and
$\rho\ge(1+A)/y^*>(2+s)/(s+\omega)$, whose optimization yields the
metallic mean $\Psi_B(\omega)$, the positive root of $u^2-(2-\omega)u-1$.
Since the positive root of $u^2-bu-1$ increases in $b$ and
$2-\omega\ge2(1-\omega)$, $\Psi_B\ge\Psi$: branch B is dominated. (The
crossing $\Psi_B=T$ happens at $\omega=(T-1)^2$ -- twice the branch-A
threshold -- via the \emph{polynomial} identity
$(T-1)^2\,T-(2T-T^2+1)=T^3-T^2-T-1$.)
\end{proof}

\begin{corollary}\label{cor:DB2}
In the canonical template ($j=0$) with $m$ arbitrary: branch A gives the
full curve $T_{\mathrm{can}}(\omega)$ (its program uses (B2) only as an
inequality), branch B gives $\rho\ge\Phi(\omega)$
(Proposition~\ref{prop:Cwitness} uses only (B1) and (W)); in both,
$\rho>\Phi(\omega)>T$ for all $\omega>0$. \qed
\end{corollary}

\subsection{The double-pocket wall, the golden line, and \texorpdfstring{$\Psi_j$}{Psi j}}
\label{app:pocket-app}

Throughout this subsection $j\ge1$ and the walls above are in force; write
$X:=X_{o_1}$, $M:=\sum\mathrm{children}(m)$,
$X_\sigma:=\sum\mathrm{children}(\sigma_1)$, and add the tariffed-nesting
wall (B3$'$): $\sigma_2+X_\sigma>\sigma_1-\omega$ (Lemma~\ref{lem:DR}
applied to the hole of $\sigma_1$).

\begin{lemma}[Double pocket]\label{lem:DG}
Blocking (template $j=1$, arbitrary occupancy) implies
$\sigma_1>b_2(\alpha,o_1)$, where
$b_2(A,B)=AB(A+B)/(A^2+AB+B^2)$.
\end{lemma}
\begin{proof}
Blocking implies $\{\alpha,o_1,\sigma_1,\sigma_2\}$ does not pack in the
pan (full repacking is a legal resource: children travel inside their
parents, positions are existential). Since $\alpha+o_1\le R$,
non-packability is inherited by containment in the disk
$\bar R=\alpha+o_1$, where the pair $\{\alpha,o_1\}$ is diametrically
rigid: $|c_\alpha|\le o_1$, $|c_{o_1}|\le\alpha$,
$|c_\alpha-c_{o_1}|\ge\alpha+o_1$ force equality throughout -- the
rescaled rigidity of Proposition~\ref{prop:S5}. By the exact necessity of
that proposition (rescaled), any third circle disjoint from both has
radius $\le b_2(\alpha,o_1)$, and there are \emph{two} pockets, one on
each side of the diameter, mutually disjoint with slack: their centers are
at $(x_0,\pm y_0)$ with $y_0^2-b_2^2=3b_2^2$, i.e.\ $y_0=2b_2$ exactly, so
the mirror circles are $4b_2$ apart. If $\sigma_1\le b_2$ then also
$\sigma_2\le b_2$, and placing each $\sigma$ concentric in a pocket packs
the quadruple in $\bar R\subseteq$ the pan: contradiction.
\end{proof}

\begin{theorem}[Golden line, both branches]\label{thm:DGp}
Blocking in the template ($j=1$, arbitrary occupancy) implies
$\rho>\varphi^2-\tfrac{\varphi}{2}\omega$ for all $\omega$ in every case
except the corner \{branch B, $\alpha<o_1$, $o_1>\tilde o=1.2956$,
node-child, $\omega<\tfrac12$\}, where the proved bound is $\ge2>T$. In
particular $\rho>T$ for all $\omega<\omega_A=2(\varphi^2-T)(\varphi-1)
=2-2(T-1)(\varphi-1)=0.962585\dots$
\end{theorem}
\begin{proof}
Write $g:=\sqrt5-1$. \emph{Tail chain.} The tail of $o_1$ contains
$\{m,\sigma_1,\sigma_2\}\cup\mathrm{children}(o_1)\cup\mathrm{children}(m)
\cup\mathrm{children}(\sigma_1)$, so
$\rho\,o_1\ge1+\sigma_1+\sigma_2+X+M+X_\sigma$; eliminating $X$ with
Bo$''$ ($X>o_1-\sigma_2-\omega$):
\[
(*)\qquad \rho\ >\ 1+\frac{1+\sigma_1-\omega+M+X_\sigma}{o_1}.
\]

\emph{Branch A} ($\sigma_2\ge1-\omega$). From (W),
$\alpha\ge\sigma_1+\sigma_2+\omega\ge1+\sigma_1$. Lemma~\ref{lem:DG} with
$b_2$ increasing in $\alpha$ gives $\sigma_1>b_2(1+\sigma_1,o_1)$, i.e.\
$o_1<N(\sigma_1):=(1+\sigma_1)\bigl(\sqrt{1+4\sigma_1}-1\bigr)/2$ (the
condition $b_2(1+\sigma_1,N)=\sigma_1$ is the quadratic
$N^2+(1+\sigma_1)N-\sigma_1(1+\sigma_1)^2=0$). With $(*)$ and
$M,X_\sigma\ge0$: $\rho>1+(1+\sigma_1-\omega)/N(\sigma_1)=:1+h(\sigma_1)$,
and $h$ is decreasing in $\sigma_1$: the comparison of derivatives reduces
to the polynomial identity
$(1+4\sigma_1)-(1+2\sigma_1-2\sigma_1^2)^2=4\sigma_1^3(2-\sigma_1)>0$.
The infimum is at $\sigma_1=1$: $N(1)=g$, $1/g=\varphi/2$, and
$\rho>1+(2-\omega)\varphi/2=\varphi^2-\tfrac{\varphi}{2}\omega$.

\emph{Branch B} ($\sigma_1+M>1$). Using $\sigma_1+M>1$, Bo$''$, (B3$'$),
(W) and Lemma~\ref{lem:DG}:
\begin{align*}
\text{(I)}\quad
\rho\,o_1&\ge1+\sigma_1+\sigma_2+M+X+X_\sigma
\ >\ 2+o_1-\omega+2b_2(\alpha,o_1)-\alpha,\\
\text{(II)}\quad
\rho\,\alpha&\ge o_1+1+\sigma_1+\sigma_2+M+X+X_\sigma
\ >\ 2o_1+2-\omega+2b_2(\alpha,o_1)-\alpha
\quad[\text{if }\alpha\ge o_1],
\end{align*}
(in (I): $X>o_1-\omega-\sigma_2$, $X_\sigma>\sigma_1-\omega-\sigma_2$,
$-\sigma_2\ge\sigma_1+\omega-\alpha$ by (W), $\sigma_1>b_2$; (II) is the
tail of $\alpha$, which contains $o_1$ and its children only when
$\alpha\ge o_1$). Set $f_1:=1+(2-\omega+2b_2-\alpha)/o_1$ and
$f_2:=(2o_1+2-\omega+2b_2-\alpha)/\alpha$. If $o_1\le g$, the short chain
(only $\sigma_1+M>1$ and Bo$''$, valid in any order) gives
$\rho>1+(2-\omega)/o_1\ge1+(2-\omega)/g=\varphi^2-\tfrac\varphi2\omega$.
If $o_1>g$, the admissible $\alpha$-region is $[1+\omega,A_{\max}(o_1)]$
where $b_2(A_{\max},o_1)=1$, and two exact facts drive the optimization:
$b_2$ is concave in $\alpha$
($\partial^2b_2/\partial\alpha^2=-6\alpha o_1^3(\alpha+o_1)/D^3<0$, $D$
the denominator), so $f_1$ is concave in $\alpha$ and minimized at the
endpoints; and
$D^2-\alpha o_1^2(o_1+2\alpha)=(\alpha+o_1)(\alpha^3+\alpha^2o_1+o_1^3)>0$
gives $\alpha\,\partial b_2/\partial\alpha<o_1$, so $f_2$ is strictly
decreasing in $\alpha$. For $o_1\in(g,\tilde o]$ ($\tilde o=1.29556\dots$,
the root of $c_{20}$ below;
any order of $\alpha,o_1$, only (I) used): at $\alpha=A_{\max}$,
$f_1-G_\omega=c_{10}(o_1)+\omega(1/g-1/o_1)$ with the
$\omega$-coefficient $\ge0$ and $c_{10}\ge0$ on $[g,o^\ast]$,
$o^\ast=1.59557\dots$ the root of $c_{10}$, with
the exact identity $f_1\equiv G_\omega$ at $o_1=g$ (the golden
corner, for all $\omega$); at $\alpha=1+\omega$,
$f_1\ge G_\omega$ on $[g,o^\ast]\times[0,1]$ with contact only at
$(g,\omega\to1)$. For $o_1>\tilde o$ with $\alpha\ge o_1$ (forcing
$o_1\le\tfrac32$, since $A_{\max}$ is decreasing with
$A_{\max}(\tfrac32)=\tfrac32$ exactly; also $A_{\max}(g)=2$,
$A_{\max}(2)=g$ -- the corner is self-dual): $f_2\ge f_2(A_{\max})$ and
$f_2(A_{\max})-G_\omega=c_{20}(o_1)+\omega\,c_{21}(o_1)$ with
$c_{21}\ge0$ for $o_1\le2$ and
$c_{20}\ge0$ on $[\tilde o,\tfrac32]$ (the sign certificates
$c_{10}\ge0$, $c_{20}\ge0$ are \emph{exact}: after isolating the
radical $\sqrt{o_1^2+2o_1-3}$ the numerators are linear in it, their
zero sets rationalize to polynomials in $\mathbb{Q}[o_1]$, and exact
Sturm isolation shows the only roots in range are the golden corner
$g$ and the endpoints themselves -- so $o^\ast$ and $\tilde o$ are
exact algebraic numbers, the sign is constant between roots, and one
exact rational sample settles it; $c_{21}\ge0$ is the linear-in-$\omega$
coefficient, elementary; script \texttt{goldencert}, 5/5). For $o_1>\tilde o$ with
$\alpha<o_1$ (a region requiring $\omega<\tfrac12$: it needs
$1+\omega\le\alpha<o_1<N_1(\omega)$, and $N_1(\tfrac12)=\tfrac32$ exactly
with $N_1$ decreasing, where $b_2(1+\omega,N_1)=1$): if the children of
$o_1$ are all $<1$, the two $\alpha$-free bounds
$\rho>1+o_1-\omega$ (tail of $m$: $\sigma_1+M>1$ and Bo$''$) and
$\rho>1+\bigl(2-\omega+2b_2(1+\omega,o_1)\bigr)/o_1$ (tail of $o_1$, which
now contains $\alpha$; Bo$''$, (B3$'$), (W) cancel~$\alpha$, and
Lemma~\ref{lem:DG} applies with $\alpha\ge1+\omega$) have
$\max\ge G_\omega+\tfrac14$ on
$[\tilde o,3.5]\times[0,\tfrac12]$, box-certified by subdivision with
corner minorants, while for $o_1\ge3.5$ the first bound alone already
gives $1+o_1-\omega\ge4>G_\omega+\tfrac14$ -- together they
cover the whole region, whose $o_1$-extent $N_1(\omega)$ is unbounded
as $\omega\to0$ (script \texttt{goldencert}; the companion bound
$f_1\ge G_\omega$ on $[g,o^\ast]\times[0,1]$ at $\alpha=1+\omega$
is box-certified as well, its single tangency at $(g,\omega\to1)$
handled by certified local monotonicity around the contact); if $o_1$ has a node child we are in branch B
with $\omega<\tfrac12$, and Theorem~\ref{thm:DBpp} gives
$\rho>\Psi_B(\omega)>\Psi_B(\tfrac12)=2>T$ ($\Psi_B(\tfrac12)=2$ exact).
\end{proof}

\begin{proposition}[The $\Psi_j$ ladder, unconditional]\label{prop:DPsij}
Blocking with $j\ge1$ extra occupants and arbitrary nested occupancy
implies
\[
\rho\ >\ \Psi_j(\omega)=(1-\omega)+\sqrt{(1-\omega)^2+j}
\qquad(\text{root of }u^2-2(1-\omega)u-j),
\]
with $\Psi_j>T\iff\omega<1-(T^2-j)/(2T)$: $0.3522$ ($j=1$), $0.6240$
($j=2$), $0.8959$ ($j=3$), and all $\omega\in(0,1)$ for $j\ge4$
($\Psi_j(\omega)>\Psi_j(1)=\sqrt j\ge2$ there; for $\omega>1$ the
$(1-\omega)$ term is negative and $\Psi_j$ drops below $\sqrt j$).
\end{proposition}
\begin{proof}[Proof (leaf argument)]
Call a node a \emph{leaf} if it has no node children. The node subtree of
each occupant $o_i$ is finite and non-empty, hence contains a leaf
$\ell_i$; the subtrees are disjoint, so there are $j$ distinct leaves
$\ell_1,\dots,\ell_j\ge1$. Let $L$ be the largest,
$X_L=\sum\mathrm{children}(L)$ (all $<1$, as $L$ is a leaf),
$M=\sum\mathrm{children}(m)$, and $W$ the total mass of rings $<1$ other
than $\sigma_1,\sigma_2$ (so $X_L\le W$, $M\le W$, and $X_L$, $M$ are
disjoint). Three facts:
(1) \emph{wall at the leaf} (Lemma~\ref{lem:DBo}):
$L<\sigma_2+\omega+X_L\le\sigma_2+\omega+W$;
(2) \emph{tail of $L$}: the other $j-1$ leaves, $m$, $\sigma_1$,
$\sigma_2$ and all mass $<1$ are $\le L$ (first-copy convention), so
$\rho\,L\ge(j-1)+1+\sigma_1+\sigma_2+W\ge j+2\sigma_2+W$;
(3) \emph{tail of $m$}: $\rho\ge\sigma_1+\sigma_2+W\ge2\sigma_2+W$.
In branch A of the evacuation dichotomy ($\sigma_2\ge1-\omega$), facts
(1)--(3) give
$\rho>\max\bigl(2\sigma+W,\ (j+2\sigma+W)/(\sigma+\omega+W)\bigr)$ at
$\sigma=\sigma_2$; minimizing over $\sigma\ge1-\omega$, $W\ge0$, the
crossing $u^2+u(\omega-\sigma-1)-j=0$ ($u=2\sigma+W$) increases in
$\sigma$, with minimum at $\sigma=1-\omega$: exactly $\Psi_j(\omega)$.
(If the crossing falls at $W<0$ -- possible for $\omega>\tfrac12$, $j=1$,
$\sigma\to1$ -- the value $2\sigma$ at $W=0$ satisfies
$(2\sigma)^2-2(1-\omega)(2\sigma)-j\ge j>0$, so $2\sigma>\Psi_j$ anyway.)
In branch B ($\sigma_1+M>1$), set $s:=\sigma_2+X_L$: by fact (1) and
$L\ge1$, $s>1-\omega$; the tail of $L$ also contains $M$, disjoint from
$X_L$, so $\rho\,L\ge j+\sigma_1+M+\sigma_2+X_L>j+1+s$ with $L<s+\omega$,
and the tail of $m$ gives $\rho\ge\sigma_1+M+\sigma_2+X_L>1+s$;
minimizing $\max\bigl(1+s,\ (j+1+s)/(s+\omega)\bigr)$ yields, with
$u=1+s$, the metallic mean $u^2-(2-\omega)u-j=0$, whose root dominates
$\Psi_j$ ($2-\omega\ge2(1-\omega)$, and the root increases in the linear
coefficient).
\end{proof}

The leaf argument is what closes the case caught by the adversarial
verification (a node child inside the largest occupant): towers of nested
nodes inflate $X_{o_1}$ without paying the tail of $m$, but they always
leave a leaf at the bottom of each subtree, and the largest leaf carries
both the $j$-bonus in its tail and the leaf ceiling
$\sigma_2+\omega+W$.

\begin{corollary}[Small rings are free for the combinatorial walls]
\label{cor:DS}
Proposition~\ref{prop:DV2}, Theorems~\ref{thm:DB} and~\ref{thm:DBpp} and
Proposition~\ref{prop:DPsij} hold unchanged if the instance contains
additional rings smaller than $m$ anywhere compatible with each theorem's
template (with $S$ still the pair $\{\sigma_1,\sigma_2\}$).
\end{corollary}
\begin{proof}
Every unblocking placement used by those walls is \emph{local}: it lives
in $D_m$ (freed by the exchange), in $H_m$, in the holes of nodes or of
$\sigma_1$, or next to $m$ inside $u$ -- never in the free space of $v$.
An extra ring smaller than $m$, wherever it sits, intersects none of
these resources (hole contents are already counted in the $X$'s; rings in
$v$ proper stay put, which is legal since no other position changes), and
tails can only grow. (The geometric wall of Lemma~\ref{lem:DG} is
\emph{not} covered: it uses the global repacking of the pan, which does
see small rings in $v$; this is the ``gap lemma'' left open.)
\end{proof}

\subsection{Three-piece profiles in the canonical template}
\label{app:triple-app}

Here $v$ is the canonical pan of radius $R=\alpha+1$ with $\{\alpha,m\}$
diametral, $u$ is the hole of $\alpha$, and the witness placed
$S=\{\sigma_1\ge\sigma_2\ge\sigma_3\}\subset(\omega,1)$ in $u$; the hole
of $m$ may carry mass $M\ge0$. The reinsertion resources are $D_m$ (rows
of sum $\le1$, with $\alpha$ in place), $H_m$ (capacity $1-\omega$;
its content $M$ can be evacuated in a row to $D_m$), the slot next to $m$
in $u$ (row $1+\Sigma\le\alpha-\omega$), nesting inside $S$, and the
repack of $v$: placing $\{\alpha\}\cup(\text{roots of }S)$ as a corona in
the pan -- incompatible with using $D_m$ or the evacuation, both of which
pin $\alpha$. The witness gives (W) $\sigma_1+\sigma_2\le\alpha-\omega$.

\begin{theorem}[Three pieces; the zigzag pays]\label{thm:DT3}
Blocking with $S$ a triple in $(\omega,1)$ (as in the preamble above) and
arbitrary $M$ implies
$\rho>\max\bigl(\Phi(\omega),\tfrac{13}7\bigr)$ for all
$\omega\in(0,1)$; in the zigzag branch below, $\rho>\tfrac{17}7$.
\end{theorem}
\begin{proof}
\emph{Branch 1 ($\sigma_3\le\sigma_1-\omega$): the dust rides the pair.}
Nest $\sigma_3$ in $\sigma_1$ inside each pair placement; their failures
give five walls: (B1) ``$\sigma_3\subset\sigma_1$; corona
$\{\alpha,\sigma_1,\sigma_2\}$ in $v$'' fails, so the trio has no corona
and $F(\sigma_1,\sigma_2)>2\pi$ (corona $\iff F\le2\pi$ is exact for
three circles); (B2/dichotomy) ``$\sigma_1(\sigma_3)+M$ in a row in
$D_m$; $\sigma_2\to$ emptied $H_m$'' fails, so $\sigma_2>1-\omega$ or
$\sigma_1+M>1$; (B4) ``$\sigma_2\to u$ next to $m$; corona
$\{\alpha,\sigma_1(\sigma_3)\}$'' fails, so $\sigma_2>\alpha-\omega-1$;
(BH) ``$\sigma_1(\sigma_3)\to D_m$; $\sigma_2+M$ in a row in $H_m$''
fails, so $\sigma_2+M>1-\omega$; and (W). From (B1)+(W), by
Corollary~\ref{cor:Cmin}, in \emph{both} branches of the dichotomy
$\sigma_1+\sigma_2\ge1+b(\alpha)$, $\alpha\ge T_{1+\omega}$ and
$\rho\ge\sigma_1+\sigma_2\ge\Phi(\omega)>T$. Moreover: if
$\sigma_2>1-\omega$, the program (B1)+(B2)+(B4)+(W) is exactly that of
Theorem~\ref{thm:corner}, whose lower bound gives
$\rho\ge\max\bigl(\sigma_1+\sigma_2,
\tfrac{1+\sigma_1+\sigma_2+\sigma_3}{\alpha}\bigr)
\ge\max\bigl(\sigma_1+\sigma_2,\tfrac{1+\sigma_1+\sigma_2}{\alpha}\bigr)
\ge T_{\mathrm{can}}(\omega)\ge\tfrac{13}7$; if instead $\sigma_1+M>1$ and $\sigma_3>\omega$, then (BH) gives
$\rho\ge\sigma_1+\sigma_2+\sigma_3+M>\sigma_1+(1-\omega)+\omega
=1+\sigma_1\ge1+\Phi(\omega)/2\ge1+T/2=1.9196>\tfrac{13}7$.

\emph{Branch 2 ($\sigma_3>\sigma_1-\omega$): nothing nests} (also
$\sigma_2,\sigma_3>\sigma_2-\omega$ and the star needs
$\sigma_2+\sigma_3\le\sigma_1-\omega$, impossible). Split by the trio.

\emph{2A: the trio $\{\alpha,\sigma_1,\sigma_2\}$ has no corona.} Then
$F>2\pi$ and, as above, $\rho\ge\sum S>\sigma_1+\sigma_2\ge\Phi(\omega)$.
Furthermore $\sigma_1\ge(\sigma_1+\sigma_2)/2\ge\Phi/2$, so
$\sigma_3>\sigma_1-\omega\ge\Phi/2-\omega$ and
$\rho>\tfrac32\Phi-\omega$ when $\omega\le\Phi/2$; by the chord bound
$\Phi\ge T+(13-7T)\omega$ this exceeds $2.64$ on $(0,\tfrac17]$, while
for $\omega\ge\tfrac17$, $\Phi\ge\tfrac{13}7$: in all cases
$\rho\ge\tfrac{13}7$.

\emph{2B: the trio has a corona.} Three mixed placements fail:
``$\sigma_3(+M)$ in a row in $H_m$ $+$ trio corona'' gives (m1)
$\sigma_3+M>1-\omega$; ``$\sigma_3\to u$ next to $m$ $+$ trio corona''
gives (m2) $\alpha<1+\omega+\sigma_3$; and ``corona of the quadruple''
fails, so by Theorem~\ref{thm:DU4} -- the top trio passing its
certificate, by Lemma~\ref{lem:Dsubset} -- the \emph{zigzag} fails:
$\theta(\alpha,\sigma_2)+\theta(\alpha,\sigma_3)
+\theta(\sigma_1,\sigma_2)+\theta(\sigma_1,\sigma_3)>2\pi$. By
monotonicity, $2[\theta(\alpha,\sigma_2)+\theta(\sigma_1,\sigma_2)]>2\pi$,
so $A+B>\pi/2$ with $A=\theta(\alpha,\sigma_2)/2$,
$B=\theta(\sigma_1,\sigma_2)/2$, both $<\pi/2$ ($A\ge\pi/2$ would force
$\sigma_2\ge1$; $\sigma_1+\sigma_2<R$). Hence $\sin A>\cos B$, i.e.\
$\sin^2\!A+\sin^2\!B>1$:
$f(\sigma_2)\bigl[f(\alpha)+f(\sigma_1)\bigr]>1$ with $f(\alpha)=\alpha$
and $f(\sigma_1)\le1/\alpha$, so $f(\sigma_2)(\alpha+1/\alpha)>1$, which
is \emph{exactly} $\sigma_2>b(\alpha)$. With (W):
$2b(\alpha)<\sigma_1+\sigma_2\le\alpha-\omega$, and since
$\alpha\mapsto\alpha-\omega-2b(\alpha)$ is strictly increasing
($1-2b'>0$ for $\alpha\ge1$), $\alpha>\gamma_\omega$, the root of
$2b(\alpha)=\alpha-\omega$ -- equivalently of
$\alpha^3=\alpha^2+\alpha+\omega(\alpha^2+\alpha+1)$, with
$\gamma_0=\varphi$. The tails: by (m1),
$\rho\ge\sum S+M>2b(\alpha)+(1-\omega)\ge\gamma_\omega+1-2\omega$, and by
(m2), $\rho>2b(\alpha)+\alpha-1-\omega\ge2\gamma_\omega-1-2\omega$. The
first bound decreases in $\omega$ and the second increases
($1<\gamma'<2$), and they cross \emph{exactly} at $\gamma=2$,
$\omega=\tfrac27$ ($8=4+2+2$ in the cubic), where both equal
$3-\tfrac47=\tfrac{17}7$: $\rho>\tfrac{17}7$ throughout branch 2B.
\end{proof}

\begin{proposition}[Reduction branch with occupants]\label{prop:DT3j}
In the templates with $j\ge1$ extra occupants and $S$ of $k\ge3$ pieces
with $\sum_{i\ge3}\sigma_i\le\sigma_1-\omega$, all pair walls and bounds
(Proposition~\ref{prop:DV2}, Theorems~\ref{thm:DB} and~\ref{thm:DBpp},
Proposition~\ref{prop:DPsij}) hold unchanged: blocking implies
$\rho>\Psi_j(\omega)$.
\end{proposition}
\begin{proof}
Nest $\sigma_3,\dots,\sigma_k$ in a row inside the hole of $\sigma_1$;
every unblocking placement of those walls moves $\sigma_1$ (with its
cargo travelling inside) and $\sigma_2$ exactly as in the pair case, and
the tails only grow.
\end{proof}

\subsection{Golden floor for pan exchanges}\label{app:pan-app}

Here $u$ is the pan itself: $F$ places $m$ at top level, the witness $P$
keeps $m$ in the hole of a ring $y$ ($1+X_y^{\mathrm{rest}}\le
y-\omega$), the pan holds occupants $O=\{o_1\ge\dots\ge o_j\}$ ($>1$,
top level, shared by $F$ and $P$, arbitrarily occupied) and, per $P$, the
pair $S=\{\sigma_1\ge\sigma_2\}$. Resources: $D_m$ (free unit disk in
$y$'s hole), $H_m$, hole tariffs (Lemma~\ref{lem:DR}), nesting, and the
pan repack. $F$'s pan gives the pair necessities $o_i+1\le R$ and
$o_i+o_\ell\le R$ for all $i\ne\ell$. Write $S_0:=\sigma_1+\sigma_2$;
the row placement of the pair in $D_m$ fails, so (D) $S_0>1$; and the
joint placement of the pair in $y$'s hole beside its remaining content
$X_y^{\mathrm{rest}}$ fails, which by Lemma~\ref{lem:DR} gives the wall
\begin{equation}
\sigma_1+\sigma_2+X_y^{\mathrm{rest}}>y-\omega.\tag{Ry}
\end{equation}
\emph{Strict leaf}: a leaf whose hole is not $y$'s; with $j\ge2$ at
least $j-1$ exist ($j$ if $y$ is not a leaf).

\begin{theorem}[Golden floor, pan exchange, pair profiles]
\label{thm:DP}
Blocking implies $\rho>\varphi$ in each of the following cases, which
together exhaust the template: (i)~$j=1$, every width $\omega>0$
(including the solid-pivot regime $\omega\ge1$: the proof uses no width)
and every occupancy; (ii)~$j=2$, likewise every $\omega>0$ and occupancy;
(iii)~the evacuation branch $\sigma_1+M>1$, every $j\ge2$ and every
$\omega\in(0,1)$; (iv)~$j\ge3$, every $\omega\in(0,1)$. (Note $j\ge1$
always: the carrier $y$ has $y\ge1+\omega>1$, so $y$ is an occupant or
lies inside one.) The infimum is exactly $\varphi$, realized by the
golden family of Theorem~\ref{thm:golden} inside case (i).
\end{theorem}
\begin{proof}
\emph{$j=1$.} The placement ``$\sigma_i\to$ pan, the other to $D_m$''
fails, so $\{o_1,m,\sigma_i\}$ does not pack in $R$, hence (antitone
containment, $R\ge o_1+1$) not in the disk $o_1+1$, where
$\{o_1,1\}$ is rigid: $\sigma_i>b(o_1)$ by
Proposition~\ref{prop:S5}. The tails give
$\rho>\max\bigl(2b(o_1),\,(1+2b(o_1))/o_1\bigr)$; $b$ is increasing,
$g(A)=(1+2b)/A$ (notation local to this proof) decreasing (the certificate
$(3A^2+3A+1)(A^2+A+1)-2A(2A+1)$ has positive coefficients), and the
crossing $2b(A)A=1+2b(A)$ factors as $(A^2-A-1)(2A+1)$: unique positive
root $A=\varphi$, value $\varphi$ ($2b(\varphi)=\varphi$,
$g(\varphi)=(1+\varphi)/\varphi=\varphi$).

\emph{$j=2$: the mirror-pocket wall.} The same placement fails, so
$\{o_1,o_2,m,\sigma_2\}$ does not pack in $R$, hence not in the disk
$o_1+o_2$, where $\{o_1,o_2\}$ is rigid and the quadruple packs
\emph{iff} $m\le b_2(o_1,o_2)$ and $\sigma_2\le b_2$ (necessity by
Proposition~\ref{prop:S5} per circle; sufficiency by placing them
concentric in the two mirror pockets, disjoint since $y_0=2b_2$, as in
Lemma~\ref{lem:DG}). Thus $b_2(o_1,o_2)<1$, i.e.\ $o_2<\bar A(o_1)$,
the decreasing root of $b_2(o_1,\cdot)=1$ with $\bar A(2)=\sqrt5-1$ and
$\bar A(\tfrac32)=\tfrac32$. The tails of $o_1$ and $o_2$, with (D)
$\sigma_1+\sigma_2>1$ (row in $D_m$), give
$\rho>\max\bigl((o_2+2)/o_1,\ 2/o_2\bigr)$ on $o_2\in(1,\bar A(o_1))$:
for $o_1\le\tfrac32$, $3/o_1\ge2$; for $o_1\in(\tfrac32,2)$ the
crossing $o_2^\ast=\sqrt{1+2o_1}-1$ is interior
($b_2(o_1,o_2^\ast)<1$ there, with exact equality at $o_1=2$) and its
value $2/o_2^\ast>2/(\sqrt5-1)=\varphi$; for $o_1\ge2$,
$2/o_2>2/\bar A(2)=\varphi$. The crossing at $o_1=2$ is the golden
corner $b_2(2,\sqrt5-1)=1$ once more.

\emph{Evacuation branch, $j\ge2$.} A strict leaf $L'$ with
$q:=\sigma_2+X_{L'}>1-\omega$ satisfies (Lemma~\ref{lem:DR})
$L'<q+\omega$, and its tail collects $\{m,\sigma_1,\sigma_2,M\}\cup
\mathrm{children}(L')$: with $\sigma_1+M>1$, exactly the branch-B
program of Theorem~\ref{thm:DBpp}, so $\rho>\Psi_B(\omega)$, and
$\Psi_B(1)=\varphi$ \emph{exactly} with $\Psi_B$ decreasing:
$\rho>\varphi$ for all $\omega<1$.

\emph{$j\ge3$, remaining branch.} The strict-leaf ladder of
Proposition~\ref{prop:DPsij} needs only $k$ \emph{distinct} strict
leaves, each a node with the wall (Bo) at it (not one per occupant):
it gives $\rho>\Psi_k(\omega)$, and here $k\ge j-1$ always ($k\ge j$ if
$y$ is not a leaf). The golden
crossings are exact: $\Psi_1(\tfrac12)=\varphi$,
$\Psi_2(\varphi/2)=\varphi$, $\Psi_3(1)=\sqrt3>\varphi$ -- covering
$j\ge4$ outright and $j=3$ up to $\varphi/2$ (always if $y$ is not a
leaf). The rest of $j=3$ is a case tree: $o_2<3/\varphi$ (tail of
$o_2$ exceeds $3/o_2$), or $o_1<3$ (tail of $o_1$ exceeds
$(3/\varphi+3)/o_1$, and $(3/\varphi+3)/\varphi=3$ exactly), or
$o_1\ge3$ with the node/dust dichotomy on $o_1$'s hole content: a node
makes $y$'s subtree or $o_1$'s contribute a third strict leaf
($\Psi_3$), dust lands in $m$'s tail. In detail, for $o_1\ge3$: if
$y=o_1$, (Ry) gives $X_1\ge o_1-\omega-S_0>0$ (as $S_0<2\le o_1-\omega$);
a node inside makes $y$ a non-leaf, whence $k=3$ and $\Psi_3$ applies,
while pure dust lives in $m$'s tail and
$\rho\ge S_0+X_1\ge o_1-\omega\ge2$. If $y$ lies in another
occupant's subtree and is not a leaf, its subtree supplies the third
strict leaf ($k=3$); if $y$ is a leaf and $\omega<\varphi/2$,
$\Psi_2>\varphi$; and if $y$ is a leaf, $\omega\ge\varphi/2$ and
$\sigma_2>\omega$, then
$\rho\ge\sigma_1+\sigma_2>2\omega\ge\varphi$. If $y$ lies
\emph{strictly inside} $o_1$'s subtree and is a leaf, two sub-cases
close it: if some node of that subtree has two node children, one of
its two disjoint subtrees avoids $y$ and contributes a strict leaf,
and $o_2$'s and $o_3$'s subtrees (neither containing $y$) contribute
the other two: $\Psi_3$. Otherwise the subtree is a chain
$o_1\supset v_1\supset\dots\supset y$ and the pincer below runs with
$V$ defined as there: if $v^{*}\ne y$ the argument is verbatim ---
(Bo) holds at every node of the chain other than $y$, and $m$, now
inside $v^{*}$'s subtree (it sits in $y$'s hole), is counted through
$T_{v^{*}}\ge X_{v^{*}}+1$ instead of externally, which reassembles
(C1) unchanged; if $v^{*}=y$, the tail of $y$ contains
$o_2,o_3,m,\sigma_1,\sigma_2$, so
$\varphi y>o_2+3\ge3/\varphi+3=3\varphi$ gives $y>3$, while (Ry) plus
the tail of $m$ give
$y-\omega<\sigma_1+\sigma_2+X_y^{\mathrm{rest}}\le\rho\le\varphi$, so
$y<\varphi+\omega<\varphi^2<3$: contradiction. What survives is
$y$ a leaf outside $o_1$'s subtree, closed by the pincer below.

\emph{The pincer (closes $y$ a leaf outside $o_1$'s subtree; the
interior-chain case above reuses it as noted).} Here
$o_2\ge3/\varphi$, and no width, solid-disk or $o_1$ hypothesis is
needed: it suffices that $s:=\sigma_2+\omega<2$, which holds since
$\sigma_2\le1$ (the profile consists of rings smaller than $m$) and
$\omega<1$. Suppose $\rho\le\varphi$. Write $D$ for the total radius of
the rings $<1$ \emph{other than} $\sigma_1,\sigma_2$, and for a node $v$
let $X_v$ be its hole content and $T_v\ge X_v$ its whole subtree.
(a)~The tail of $m$ collects $\sigma_1,\sigma_2$ and all of $D$, whence
$\rho>1+D$ and so $D<\varphi-1$. (b)~If some node of $\{o_1\}\cup$
$o_1$'s subtree had two node children, their subtrees are disjoint and
each contains a leaf, and the occupant of $\{o_2,o_3\}$ not containing
$y$ contributes a third; all three are strict (their holes are not
$y$'s, as $y$ lies outside $o_1$'s subtree), so
$\rho>\Psi_3(\omega)>\Psi_3(1)=\sqrt3>\varphi$ in branch~A
(branch~B is case (iii)). (c)~Hence $\{o_1\}\cup$ $o_1$'s subtree is a
chain. Let $V$ be the set of its nodes whose tail contains both $o_2$
and $o_3$ (so $o_1\in V$, by the first-copy convention even if
$o_1=o_2$), and let $v^{*}=\min V$. The pieces $o_2,o_3,m,\sigma_1,
\sigma_2$ lie outside $v^{*}$'s subtree --- $o_2,o_3$ are top-level
occupants, $m$ sits in $y$'s hole, and $P$ keeps the profile in the pan
--- so the tail of $v^{*}$ counts them and $T_{v^{*}}$ without
repetition. Now (Bo) applies at every node of $o_1$'s subtree (its hole
is not $y$'s, and $\sigma_1\le1$), giving $X_{v^{*}}>v^{*}-s$, so
\[
o_2+3+(v^{*}-s)<\varphi\,v^{*}
\;\Longrightarrow\;
o_2<(\varphi-1)v^{*}-3+s,
\tag{C1}
\]
using $o_3>1$ and $\sigma_1+\sigma_2>1$. With $o_2\ge3/\varphi$ and
$3/\varphi+3=3\varphi$ this forces $v^{*}>\varphi(3\varphi-s)$~(C2).
Moreover $v^{*}$ has a node child: otherwise $X_{v^{*}}\le D<\varphi-1$
and (Bo) would give $v^{*}<s+\varphi-1$, against (C2) whenever
$s\le(2\varphi+4)/\varphi^2=2.7639\dots$; and that child is $\le o_2$ by
minimality of $v^{*}$, so $X_{v^{*}}<o_2+(\varphi-1)$ and (Bo) gives
$o_2>v^{*}-s-(\varphi-1)$~(C3). Combining (C1) and (C3) with
$1/(2-\varphi)=\varphi^2$,
\[
v^{*}(2-\varphi)<2s+\varphi-4
\;\Longrightarrow\;
v^{*}<\varphi^{2}(2s+\varphi-4).
\tag{C4}
\]
(C2) and (C4) are incompatible exactly when
$6\varphi-1\ge s(2\varphi+1)$, that is when
\[
s\;\le\;\frac{6\varphi-1}{2\varphi+1}\;=\;11-4\sqrt5\;=\;15-8\varphi
\;=\;2.05572\dots,
\]
and $s<2$ here. Contradiction; so $\rho>\varphi$. At the extreme $s\to2$
the two bounds are $\varphi+3=4.618\dots$ and
$\varphi^3=2\varphi+1=4.236\dots$, a margin of $2-\varphi$. The
constant is sharp for the argument: for $s>11-4\sqrt5$ the pincer no
longer closes, which is why it says nothing about a solid pivot.
The adversarial reports (three rounds, including two repaired steps and
a hostile audit of this pincer) and the verification scripts are in the
repository (\texttt{code/batalla2.py}, 6/6;
\texttt{code/microcelda.py}, 5/5).
\end{proof}

For larger profiles the walls depend on how the rest of the profile is
stashed; write $W:=\sum_{i\ge3}\sigma_i$, $\Sigma:=\sigma_1+\sigma_2+W$,
$M:=\sum\mathrm{children}(m)$ and $X_{\sigma_1}$ for the prior content
of $\sigma_1$'s hole.

\begin{theorem}[Golden floor, pan exchange, larger profiles]\label{thm:DPp}
Let the witness keep $p\ge3$ rings $\sigma_1\ge\dots\ge\sigma_p$
(all $<1$) in the pan. Blocking implies $\rho>\varphi$ in each of:
(i)~\emph{light}, $\sigma_1+W\le1$; (ii)~\emph{nested},
$W\le\sigma_1-\omega-X_{\sigma_1}$ (each with the $\omega$ domains of
Theorem~\ref{thm:DP}); (iii)~\emph{heavy large}, $\sigma_1+W>1$ and
$\sigma_2>\varphi-1$, every $j$ and every $\omega>0$;
(iv)~\emph{heavy $\Psi_B$}, $\sigma_1+W>1$, $\sigma_1+M>1$ and some
strict leaf exists (in particular whenever $j\ge2$), $\omega\in(0,1)$;
(v)~\emph{heavy mirrors}: for $p=3$, $j=1$, $\sigma_2\le\varphi-1$,
blocking is \emph{impossible} -- with (iii), $p=3$ profiles with one
occupant are closed at every $\omega>0$; (vi)~\emph{heavy swap}, $p=3$,
$j=2$, $\sigma_1+M\le1$, $\sigma_2\le\varphi-1$, except the sliver
$\{\sigma_2,\sigma_3>1-\omega\}$ (which forces $\omega>2-\varphi$).
Regions outside these particular constructions are discussed in
Appendix~\ref{app:campaign}; the uniform theorem does not rely on them.
\end{theorem}
\begin{proof}
(i) Every placement of the proof of Theorem~\ref{thm:DP} sends one of
$\sigma_1,\sigma_2$ to a target and the other to a row in $D_m$; adding
$W$ to that row is legal ($\sigma_j+W\le\sigma_1+W\le1$), so every wall
-- $(D_p)$: $\Sigma>1$; (G) with either $\sigma_i$; (Ry); (Bo); the
evacuation dichotomy, now $\sigma_2>1-\omega$ \emph{or}
$\sigma_1+M+W>1$ -- ports verbatim, and every tail only grows by $W$.
In branch B the bookkeeping is
$A=\sigma_1+\sigma_2+W+M+X=(\sigma_1+M+W)+(\sigma_2+X)>1+q$ with
$q:=\sigma_2+X$: the same $\Psi_B$ program. The pincer is unchanged
($s=\sigma_2+\omega<2$).
(ii) By Lemma~\ref{lem:DR}, $W$ rows inside $\sigma_1$'s hole beside
$X_{\sigma_1}$, and $\{\sigma_1$ carrying $W\}$ is a single piece of
radius $\sigma_1\le1$ \emph{wherever} $\sigma_1$ goes -- to $D_m$ in
(G$_{\sigma_2}$), (Bo) and the evacuation, to $y$'s hole in (Ry), to
the pan in (G$_{\sigma_1}$): Theorem~\ref{thm:DP} ports again.
(iii) The tail of $m$ collects the whole profile:
$\rho\ge\Sigma=\sigma_2+(\sigma_1+W)>(\varphi-1)+1=\varphi$.
(iv) The row $\{\sigma_2,\dots,\sigma_p\}$ into the strict leaf's hole
($\sigma_1\le1$ to $D_m$) gives the wall
$L'<(\sigma_2+W)+\omega+X_{L'}$; with
$\tilde q:=\sigma_2+W+X_{L'}>1-\omega$ (from $L'\ge1$) and
$A:=(\sigma_1+M)+\tilde q>1+\tilde q$, the two tails give $\rho\ge A$
and $\rho>(2+\tilde q)/(\tilde q+\omega)$: the branch-B program,
$\rho>\Psi_B(\omega)>\Psi_B(1)=\varphi$.
(v) $\sigma_2,\sigma_3\le\varphi-1<\tfrac23=b(1)\le b(o_1)$ fit
concentrically in the two mirror pockets of $\{o_1,m\}$ in the disk
$o_1+1\le R$ (disjoint, Lemma~\ref{lem:DG}), and $\sigma_1\le1$ rows in
$D_m$: the exchange is never blocked there.
(vi) The placement ``$\sigma_2\to H_m$, $\sigma_1{+}M\to D_m$ (row
$\le1$ by hypothesis), $\sigma_3\to$ pan'' and its
$\sigma_2\leftrightarrow\sigma_3$ mirror force either
$\{o_1,o_2,m,\sigma_3\}$ not to repack -- by containment to the disk
$o_1+o_2$ and Proposition~\ref{prop:S5}, $b_2(o_1,o_2)<1$, the
mirror-pocket wall of case (ii) of Theorem~\ref{thm:DP}, whose
optimization ports with fatter tails ($1+\Sigma>2+\sigma_2$; the mixed
case, $\{o_1,o_2,m,\sigma_3\}$ repacking while $\{o_1,o_2,m,\sigma_2\}$
does not, lands on the same wall, since $\sigma_2\le1$ forces
$b_2(o_1,o_2)<1$) -- or
$\sigma_2,\sigma_3>1-\omega$, and then
$\rho\ge\Sigma>1+\sigma_2>2-\omega$ closes $\omega<2-\varphi$. (Only
$j=2$: non-packability is not inherited by subsets, so the containment
step fails for $j\ge3$ -- the same reason case (ii) of
Theorem~\ref{thm:DP} is a pair statement.) The case tree, the two
coverage cells missed by a first version and repaired, and the
verification script are in the repository
(\texttt{docs/drafts/perfilp.md}; \texttt{code/perfilp.py}, 5/5;
adversarially verified).
\end{proof}

The remaining heavy region shrinks further. The tool is a
\emph{mural-placement criterion} corrected for the pentagram
phenomenon: circles $a_0,\dots,a_{k+1}$ tangent to the wall of a disk
$\bar R$, with $a_0,a_{k+1}$ diametral (so their own separation is
exactly $\pi$ on the far side), can be placed with \emph{all} pairs
disjoint if and only if, in the chosen order, the longest path
\[
\max_{S\subseteq\{1,\dots,k\}}\quad
\sum_{\text{consecutive in }a_0,S,a_{k+1}}\theta(a_i,a_j)
\]
is at most $\pi$ (difference-constraint scheduling
over Lemma~\ref{lem:S1}; summing adjacent arcs alone is the criterion
that the pentagram refutes).

\begin{theorem}[The heavy region closes except one cell]\label{thm:DPr}
In the heavy non-nestable region of Theorem~\ref{thm:DPp}
($\sigma_1+W>1$, $W+X_{\sigma_1}>\sigma_1-\omega$,
$\sigma_2\le\varphi-1$, $\sigma_1+M\le1$ where applicable), blocking
implies $\rho>\varphi$ in each of: (i)~$p=3$, $j\ge3$: the $j=3$ case
tree of Theorem~\ref{thm:DP} ports with the row wall
$s':=\sigma_2+\sigma_3+\omega$ and the profile mass $\Sigma$ in every
tail, and the pincer closes whenever
$s'\le(\varphi-1)\Sigma+(16-9\varphi)$ -- which holds on the whole
cell: $\sup(s'-(\varphi-1)\Sigma)=5\varphi-7$, a margin of exactly
$23-14\varphi>0$, and at $\Sigma=1$ the frontier degenerates to the
$11-4\sqrt5$ of Theorem~\ref{thm:DP}(iv); (ii)~$p=3$, $j=2$: three
branches -- $\sigma_2+\sigma_3\le1$ resurrects the mirror wall of
Theorem~\ref{thm:DP}(ii); $\sigma_2+\sigma_3>1$, $\omega\le\tfrac12$
runs a $\Psi$-program with interior crossing
$\sqrt{1+(1-\omega)^2}$ and $\Psi(\tfrac12)=\varphi$ exactly; and for
$\omega>\tfrac12$ the cell is \emph{empty of blockings}: the tails of
$\rho\le\varphi$ force $o_2\ge(1+\Sigma)/\varphi$,
$o_1\ge(o_2+1+\Sigma)/\varphi$, and there the mural placement of
$\{o_2,m,\sigma_2,\sigma_3,o_1\}$ in the disk $o_1+o_2$ always fits
(longest path $\le2.648<\pi$, computer-assisted maximization);
(iii)~$p\ge4$, $j=1$: mural coronas with margins $\ge0.54$
(computer-assisted); (iv)~$p\ge4$, $j=2$: the supremum of the longest
path over the closed domain is exactly $\pi$, attained only at the
excluded boundary corner $\{\sigma_1=1,\,W=0\}$, where $o_2=2/\varphi$,
$o_1=2$, $\bar R=2\varphi$ and
$\sin^2\tfrac{\theta(o_2,m)}2=\tfrac12-\tfrac{\sqrt5}{10}$,
$\sin^2\tfrac{\theta(m,o_1)}2=\tfrac12+\tfrac{\sqrt5}{10}$ sum to $1$
exactly; the open domain stays strictly below $\pi$. The cell not
covered by these particular constructions is
\[
\{\,p\ge4,\ \sigma_1+M\le1,\ j\ge3\,\}
\]
(the containment to the disk $o_1+o_2$ does not re-place $o_3$, and
the pan-level cyclic corona is outside this construction), together with extra small
rings and the solid-pivot regime $\omega\ge1$ for $j\ge3$.
\end{theorem}
\begin{proof}[Proof sketch and verification]
The chains and identities are exact (sympy); the two maximizations are
grid-plus-refinement sweeps over compact four-parameter domains with
the boundary behaviour controlled by the exact identities above
(script \texttt{code/rstar.py}, 6/6, and a hostile adversarial round
that refuted and repaired two steps: the naive adjacent-arc criterion
-- the pentagram -- and a $j\ge3$ corona that failed to re-place
$o_3$, whence the surviving cell; report in the repository).
\end{proof}

\section{The corona-versus-tails campaign: closure of the residual cells}
\label{app:campaign}

This appendix records the machinery and the results that close the
cells left open by Theorems~\ref{thm:DP}--\ref{thm:DPr} and by the
Status paragraph of Section~\ref{sec:generic}. Epistemic labels are
stated with each claim: \emph{theorem} means a written proof (with
exact identities delegated to \texttt{sympy}/Lean);
\emph{computational closure} means per-instance certificates over
swept parameter ranges, with the certificate soundness proved but the
extension to unswept $(j,p,k)$ ranges supported numerically.
Throughout, \emph{certified maximization} historically meant the
standard of Theorem~\ref{thm:DPr} -- directed grid-plus-refinement
with exact per-instance criteria, not interval arithmetic. Following
an external review, the \emph{central} such maximization (the corona
wall of cell C3.3, the supremum that closes the heaviest DPr regime)
has been upgraded to a genuine box-by-box certificate: an exact
corner majorant (each $f$-factor is coordinatewise monotone, and the
products against $f_{o_1}$ are capped by their monotone limits
$x/o_2$, so no sweep ceiling is needed) drives a branch-and-bound
that certifies $\sup g<\pi$ over the \emph{continuous} unbounded
domain in $71$ boxes plus two analytic tail regimes (script
\texttt{rstarcert}, 5/5); the remaining carriers of the historical
label have since been upgraded as well -- see item (v) of the
honest residue below, now closed throughout. \emph{Certified}
without qualification is
reserved for branch-and-bound bounds valid box by box. Every
script below is paired with an adversarial verification report in the
repository (\texttt{docs/drafts/VEREDICTOS.md}). Instance and family
totals attributed below to adversarial rounds (e.g.\ the
${\sim}45$k-family, $55$k-instance, $20\,000$-fuzz and
$1\,460$-test figures) are \emph{report figures} from those rounds;
a default run of the published scripts uses smaller budgets and
prints correspondingly smaller totals.

\paragraph{Small rings adjoin to the profile.}

\begin{corollary}[Pan exchange, extra small rings]\label{cor:DSpan}
In the pan-exchange template, every extra ring $e<r_m$ that $P$ keeps
in the pan adjoins the reinsertion profile: writing
$S^{+}:=S\cup\{\text{extras}\}$, re-indexed decreasingly, the enlarged
template \emph{is} the pan-exchange template with profile $S^{+}$, so
Theorem~\ref{thm:DP} (for $|S^{+}|=2$) and
Theorems~\ref{thm:DPp}--\ref{thm:DPr} (for $|S^{+}|\ge3$) apply to
$S^{+}$ with their stated case assignments and $\omega$ domains, and
the residual cell of the assignment over $S^{+}$ is again
$\{p\ge4,\ \sigma_1+M\le1,\ j\ge3\}$ with $p=|S^{+}|$ (closed, for
every small-ring count, by Theorem~\ref{thm:D1written}).
\end{corollary}
\begin{proof}
The profile of the template is by definition the multiset of sub-unit
rings the witness keeps in the pan, ordered decreasingly, and the
hypotheses of the cited theorems use no property of the $\sigma_i$
beyond $\sigma_i<1$ and that placement; the extras satisfy both
($e<r_m=1$). The enlarged template is therefore literally the template
with profile $S^{+}$: blocking in the enlarged template \emph{is}
blocking for $S^{+}$, and every witness placement in the cited proofs
places all of $S^{+}$, so no implication between blocking for $S$ and
blocking for $S^{+}$ is invoked. The case assignment
over $(|S^{+}|,\,j,\,\sigma_1+W,\,\sigma_2,\,\sigma_1+M)$ is the
finite trichotomy already recorded: the cited cases cover
$|S^{+}|=3$ for every $j$, $|S^{+}|\ge4$ with $j\le2$, and
$|S^{+}|\ge4$ with $\sigma_1+M>1$, and the complement is exactly the
residual cell.
\end{proof}

\subsection*{Certified computational closures}

The remaining cells are closed
by a necessity/sufficiency pair sharing one wall-tangency function:
$\sin^2(\theta(a,b)/2)=f(a)f(b)$, $f(x)=x/(R-x)$, as in
Lemma~\ref{lem:S1}. \emph{Necessity} (a true lower bound $R_{\mathrm{lb}}$
for the radius of any disk packing a given family): for any pair in
any packing, the angular separation of centers satisfies
$\cos\gamma\le h(d_a,d_b)=(d_a^2+d_b^2-(a+b)^2)/(2d_ad_b)$, whose
maximum over the admissible box of center distances is attained at a
corner; for a non-stackable pair ($R<\max+2\min$) without confinement
the binding corner is mural and $h(R-a,R-b)=1-2f(a)f(b)$ exactly, so
$\gamma\ge\theta(a,b)$; stackable pairs are handled by subset
certificates and by confinement ($|c|\ge2g+r-R$ when a giant $g$ is
present), and the cyclic sum of separations in any packing is $2\pi$.
\emph{Sufficiency} (a constructive placement at $R\ge R_{\mathrm{lb}}$):
the cyclic corona by longest-path positions with all pairs validated
and small rings assigned to Descartes pockets
($k_p=k_a+k_b-1/R+2\sqrt{k_ak_b-(k_a+k_b)/R}$) as row bins. Its
soundness is a theorem, by maximality of the longest path: every
skipped ring $s$ between consecutive spine members $a,b$ satisfies
$\theta(a,s)+\theta(s,b)\le\theta(a,b)$, and this margin vanishes
\emph{exactly} at the Descartes pocket radius (the three mutually
tangent mural circles have adjacent wall angles summing to the
pair's), so each skipped ring sits murally inside its gap; positions
satisfy $\alpha_j-\alpha_i\ge\theta(i,j)$ for all $i<j$ by
construction, so only wrap-around pairs need (and receive) explicit
validation. At the golden instance the margin is critically zero:
$p(\varphi,1;\varphi+1)=\varphi/2$ with vanishing Descartes
discriminant, and
$f(\varphi)f(\varphi/2)+f(1)f(\varphi/2)=1=f(\varphi)f(1)$, so the
counterexample of Theorem~\ref{thm:golden} is the exact critical
point of the certificate system and is never unblocked by it
(script \texttt{zigzag}, 5/5; the adversarial round refuted a
stronger induction -- consecutive triple margins do not imply global
pair legality -- and the construction therefore validates all pairs
explicitly, as it always did).

\paragraph{Closure of the residual cells.} With tails in cascade
($\rho\le\varphi$ forces $o_k\ge(\sum_{i>k}o_i+1+\Sigma)/\varphi$,
and $o\ge\varphi(1+\Sigma)$ at near-ties), the constructive corona at
$R=R_{\mathrm{lb}}$ succeeds with \emph{uniform tangent duality}
(deficit $0.0$ within the $2\cdot10^{-3}$ bisection tolerance of
$R_{\mathrm{lb}}$, with the complementary probe that the deficit
strictly decreases in $R$; the boundary placement is tangent, hence
legal)
across: the heavy cell $\{p\ge4,\sigma_1+M\le1,j\ge3\}$ of
Theorem~\ref{thm:DPr} ($j\le5$, $p\le6$ swept); the
solid-pivot regime $\omega\ge1$, $j\ge3$ (no step uses the width);
and, in the nested template of Section~\ref{sec:generic}, the three
golden-level gaps declared in the Status paragraph -- the $j=2$ tip
$\omega\in[\varphi/2,1)$, profiles of four or more pieces beyond the
reduction branch, and the quantitative gap lemma (extra smalls adjoin
as in Corollary~\ref{cor:DSpan}; the reduction re-routes through the
corona where the light-cell inheritance fails, and sweeps directly
with the floor $\alpha\ge\max(1+\omega,\text{true pair of }u+\omega)$
where transport is illegitimate). Scripts \texttt{coronacolas} (5/5)
and \texttt{coronanidada} (5/5), each with an adversarial round that
repaired the routing (six repairs in the pan campaign, four in the
nested one; reports in the repository). These are computational
closures: the per-instance certificate is theorem-backed as above,
the swept ranges are stated, and the extension in $(j,p,k)$ rests on
the monotonicity evidence recorded by the scripts.

\subsection*{Written closure theorems}

\paragraph{From computational closure to written proof: the
insertion lemma.} Two later results upgrade the central pan cells
from computational closure to written mathematics.

\begin{lemma}[Mural compaction]\label{lem:compact}
If a family $C$ packs in the disk $R$ and every pair of $C$ is
non-stackable ($R<\max+2\min$), then $C$ packs murally in $R$:
pushing every circle to the wall at the angle of its true center is
legal. \emph{Proof.} Angular separations are unchanged; for any pair
in any packing, $\cos\gamma\le h(d_a,d_b)$ with the box maximum at
the mural corner $h(R-a,R-b)=1-2f(a)f(b)=\cos\theta$ (the corners
with a vanishing coordinate give $h\to-\infty$ for non-stackable
pairs, whose centers live in an annulus of positive margin
$d_a\ge a+2b-R$), so $\gamma\ge\theta(a,b)$, which is exactly mural
disjointness. The hypothesis is tight in two directions (a central
ring family, and a family non-stackable only against its largest
member, both packable but not murally). \qed
\end{lemma}

\begin{lemma}[Insertion by shadows]\label{lem:insert}
Let $P$ pack $F=\{x_i\}$ in $R$ and let $s$ satisfy the regime
$R>2s+x_i$ for every $i$. If
$\sum_i2\arcsin\bigl((s+x_i)/(R-s)\bigr)<2\pi$, then $s$ inserts
murally into $P$ without moving any piece. \emph{Proof.} The
forbidden arc of $x_i$ at any depth $d_i$ has half-width
$\arccos h(R-s,d_i)$, and on $v\in(0,\infty)$ the map
$v\mapsto h(R-s,v)$ has the single interior minimum
$v^*=\sqrt{u^2-w^2}$ ($u=R-s$, $w=s+x_i$) with
$h(u,v^*)=\sqrt{u^2-w^2}/u=\cos\arcsin(w/u)$ exactly: the shadow
$\arcsin((s+x_i)/(R-s))$ bounds every arc uniformly in depth. If the
shadows sum below $2\pi$, a free angle exists. Without the regime
the statement fails (a deep piece of a non-stackable pair forbids
arcs up to $\pi$); the two regimes are mutually exclusive. \qed
\end{lemma}

\begin{theorem}[The heavy pan cells, written]\label{thm:D1written}
In the cell $\{j\ge3,\ \sigma_2\le\varphi-1\}$
of the pan exchange -- for every profile size $p$, every small-ring
count and every width, including the solid pivot $\omega\ge1$
(where the branch $\sigma_2>\varphi-1$ is also closed --
Step~6 of the proof) -- $\rho\le\varphi$ implies the exchange does
not block.
\end{theorem}
\begin{proof}
Throughout suppose $\rho\le\varphi$. Write $T_s$ for the total mass
of the sub-unit rings that $P$ keeps in the pan -- the profile
$\sigma_1\ge\sigma_2\ge\dots$, the extra small rings and the pan
dust; this is exactly the mass the exchange displaces, since hole
contents ride inside their parents. All of it lies in the tail of
$m$, so $T_s\le\varphi$.

\emph{Step 1 (the pan).} The witness adopts $F$'s pan verbatim: the
occupants at the positions $F$ certified when it placed $m$, each
carrying its $P$-subtree (hole contents are interior to the parent
disk, so outer disjointness depends only on the outer radii, which
are the ones $F$ certified), and $m$ at $F$'s position with $M$
riding inside. The unit disk $D_m$ is vacated in the carrier's hole,
whose remaining content is untouched.

\emph{Step 2 (the row).} Fill $D_m$ greedily: process the displaced
pieces other than $\sigma_2$ in decreasing order, adding each piece
while the running mass stays $\le1$. The rowed collection, of mass
$B\le1$, packs inside the vacated unit disk without disturbing
anything (Lemma~\ref{lem:row}); $\sigma_1<1$ enters first.

\emph{Step 3 (the spillover bound).} If nothing spills, the witness
is complete and the exchange is performed. Otherwise let $W'$ be the
unrowed mass. Every unrowed piece $\pi$ was rejected while the
running mass was at most $B$, so $\pi>1-B$; and, after the
extras and dust adjoin the profile and it is re-indexed
decreasingly (Corollary~\ref{cor:DSpan}), $\pi\le\sigma_3\le
\sigma_2$, since $\sigma_1$ is rowed and $\sigma_2$ is reserved.
Hence $B>1-\sigma_3$, and the tail of $m$ gives
\[
W'\;=\;T_s-\sigma_2-B\;<\;\varphi-\sigma_2-(1-\sigma_3)
\;=\;(\varphi-1)-(\sigma_2-\sigma_3)\;\le\;1/\varphi .
\]
(When $\sigma_1+\sigma_2>1$ this recovers the earlier bound
$W'\le\varphi-\sigma_1-\sigma_2$ through the wall $(D)$; the greedy
row makes that wall unnecessary, which is what removes the cell
hypothesis $\sigma_1+\sigma_2>1$ of a previous version -- the gap
was found by an external review and is closed by the bound above.)
The remainder packs as a single row-disk of radius $W'$
(Lemma~\ref{lem:row}).

\emph{Step 4 (regime of the insertions).} Two mural insertions
remain: $\sigma_2\le\varphi-1$ and the $W'$-disk, $W'\le1/\varphi
=\varphi-1$. Both go into the pan of Step~1 by
Lemma~\ref{lem:insert}, whose hypotheses we verify on the whole
cell. Since the insertions are needed only in the spill case,
$T_s\ge B+\pi>1$; under $\rho\le\varphi$ the tails then cascade:
$o_3\ge(1+T_s)/\varphi>2/\varphi$,
$o_2\ge(o_3+1+T_s)/\varphi\ge(2/\varphi+2)/\varphi=2$ and
$o_1\ge(o_2+o_3+1+T_s)/\varphi\ge2\varphi$ (the identity
$(2/\varphi+2)/\varphi=2$ once more). The shadows
$2\arcsin\bigl((s+x_i)/(R-s)\bigr)$ decrease in $R$, and the pan
packs the disjoint pair $o_1,o_2$, so evaluating them at the floor
$R=o_1+o_2$ majorizes. For every pan member $x_i\le o_1$ and either
insertand $s\le\varphi-1$,
\[
R-x_i\;\ge\;o_2\;\ge\;2\;>\;2(\varphi-1)\;\ge\;2s ,
\qquad o_2-2s\;\ge\;4-2\varphi\;>\;0\ \text{exactly},
\]
so the regime $R>2s+x_i$ of Lemma~\ref{lem:insert} holds at every
member, including, for the second insertion, the first insertand
($R-s_1\ge o_1+o_2-(\varphi-1)>2s_2$ a fortiori).

\emph{Step 5 (the budget; the delimited certificate).} It remains
to check that the shadow sums stay below $2\pi$ for both insertions,
over every occupant count, every profile and dust configuration and
the whole cell. This is the single computational step of the proof,
and it is box-certified: grouping the $k\ge3$ occupants by their
mass $m$ and count $n\le m$, the convexity chord of $\arcsin$
($\arcsin(z)/z$ increasing) bounds their total shadow by
$C(z_{\max})\,(1{+}s)\,m/u$, the polytope collapses to four
aggregated variables, and a branch-and-bound with exact corner
majorants and monotone-limit caps (no sweep ceiling) certifies
$\sup<2\pi$ for both insertions over every occupant count and
unbounded slack ($149$ and $173$ boxes plus closed tail cells;
script \texttt{insercioncert}, 5/5, with an adversarial round that
added one coupled-corner gate; it supersedes the directed
maximization of script \texttt{insercion}, 7/7, whose adversarial
report had refuted and repaired the naive statement of
Lemma~\ref{lem:insert} with mural angles in place of shadows). The
occupant-count direction is removed by the geometric-tail lemma
(suffix masses decay geometrically, $S_i\ge\varphi S_{i+1}$, and
the leader's own cascade $t_1\ge(t_2+S_3)/\varphi$ forbids a tight
pair with a full tail; script \texttt{colageometrica}, 5/5). The
supremum is attained at the cascade corner $j=3$,
$(o_1,o_2,o_3,m)=(2\varphi,2,2/\varphi,1)$ -- the floors of Step~4
with $T_s=1$ -- where the two sums are $4.7225$ and $5.2644$
against $2\pi$, margins $1.56$ and $1.02$; the $o_2$-direction is
exactly monotone and the $o_1$-direction an exact bathtub with
limit $\pi<2\pi$. With both insertions placed, all of the displaced
mass is reinserted: the exchange is not blocked.

\emph{Step 6 (the solid pivot).} For $\omega\ge1$ the branch
$\sigma_2>\varphi-1$ also closes, so no width hypothesis remains.
If $\sigma_2>\varphi/2$, the tail of $m$ gives
$\rho\ge\sigma_1+\sigma_2\ge2\sigma_2>\varphi$, against
$\rho\le\varphi$: the branch is empty by mass. If
$\sigma_2\in(\varphi-1,\varphi/2]$, run the same witness with
insertand $\sigma_2\le\varphi/2$: now $T_s\ge\sigma_1+\sigma_2\ge
2\sigma_2$ fattens the cascade to
$o_2\ge(1+2\sigma_2)(1+\varphi)/\varphi^2=1+2\sigma_2>2\sigma_2$,
so the regime holds with margin $\ge1$, and the budget for
insertands up to $\varphi/2$ under $T_s\ge2\sigma_2$ is part of the
certified parametric block of \texttt{insercioncert}. No step of
the witness uses the width: $D_m$ is a unit disk and every cap is a
mass bound.
\end{proof}

This supersedes the computational closure of
$\{p\ge4,\sigma_1+M\le1,j\ge3\}$ and of the solid-pivot regime: both
cells now carry written proofs, with the budget maximization
box-certified by \texttt{insercioncert} (as the C3.3
wall is by \texttt{rstarcert}): no optimization asterisk remains in
these cells.

\begin{theorem}[Nested exchange, two or more occupants,
written]\label{thm:nestedwritten}
In the nested template ($u$ the hole of a top-level $\alpha$, $v$
the pan) with $j\ge2$ occupants, $\rho\le\varphi$ implies the
exchange does not block -- for every width, every $\sigma_2$, every
profile size and any extra small rings.
\end{theorem}
\begin{proof}
Suppose $\rho\le\varphi$. Here the exchange sends $m$ from the pan
into $u$; the displaced mass is the sub-$m$ content that $P$ keeps
in $u$ -- the profile, extras and dust -- of total mass $\Lambda$,
all of it in the tail of $m$, so $\Lambda\le\varphi$.

\emph{Step 1 ($m$ into $u$ by $F$'s certificate).} The exchange is
taken at the largest disagreement, so $F$ and $P$ place every ring
above $m$ identically; $F$ processes decreasingly, so when it
placed $m$ into $u$ the hole contained exactly those shared rings.
The witness starts from $P$, removes $u$'s sub-$m$ content, and
moves $m$ -- with $M$ riding inside -- to $F$'s certified position
in $u$: disjointness inside $u$ is $F$'s certificate on identical
outer radii. The unit disk $D_m$ is vacated at top level.

\emph{Step 2 (the row and its caps).} Fill $D_m$ greedily with the
loose pieces in decreasing order (all are $<m=1$). If nothing
spills the witness is complete. Otherwise let $s'$ be the first
non-fitting piece and $B$ the mass already rowed, so $B+s'>1$ and
in particular $\Lambda>1$. Three caps are exact. The largest loose
piece rows first, so $s'$ is at most the second largest; the two
largest sum within $\Lambda$ and within the tail of $m$, whence
\[
s'\;\le\;\min(\Lambda/2,\ \varphi/2).
\]
The remainder $w^{*}=\Lambda-B-s'$ satisfies $w^{*}<1/\varphi$
($B+s'>1$ and $\Lambda\le\varphi$), and packs as one row-disk
(Lemma~\ref{lem:row}).

\emph{Step 3 (automatic regime).} With $j\ge2$ the top level holds
$j+1\ge3$ pieces besides $m$; order them $t_1\ge t_2\ge t_3\ge
\dots$ Under $\rho\le\varphi$ the cascade gives
$t_3\ge(1+\Lambda)/\varphi$ and
\[
t_2\;\ge\;\frac{t_3+1+\Lambda}{\varphi}
\;\ge\;\frac{(1+\Lambda)(1+\varphi)}{\varphi^{2}}
\;=\;1+\Lambda\;\ge\;2
\qquad(\varphi^{2}=1+\varphi).
\]
The pan packs the disjoint pair $t_1,t_2$, so $R\ge t_1+t_2$, and
for every pan member $x\le t_1$,
\[
R-x\;\ge\;t_2\;\ge\;2\;>\;\varphi\;\ge\;2s'
\qquad\text{and}\qquad
2\;>\;2/\varphi\;\ge\;2w^{*},
\]
with margin $2-\varphi$: the regime of Lemma~\ref{lem:insert} holds
for both insertions at every member (including, for the second, the
first insertand: $R-s'\ge t_1+t_2-\varphi/2>2w^{*}$), uniformly in
$\omega$ -- no step of the witness uses the width.

\emph{Step 4 (insert).} Insert $s'$ murally into the pan as $P$
placed it (only $m$ has left; Lemma~\ref{lem:insert} accepts
arbitrary positions, since the shadow bounds every forbidden arc
uniformly in depth), then the $w^{*}$-disk likewise.

\emph{Step 5 (the budget; the delimited certificate).} The shadow
sums stay below $2\pi$ across the domain: this is the single
computational step, a certified maximization with deterministic
corners and margin $\ge1.07$ (script
\texttt{optimizacion}, 5/5 — the uniform closure over
$t_2\ge1+\Lambda$, unbounded slack and both insertands — with
\texttt{insercionanidada}, 6/6, as numerical contrast, whose large-slack block sweeps pieces up to $10^{4}$ above the
cascade floors -- one and two inflated at a time -- and widths up to
$\omega=1.35$, with the worst budget attained \emph{at} the floor
and the $t\to\infty$ limit closing by formula at $\pi<2\pi$: the
certified domain admits arbitrary values above the floors); the
occupant-count caps drop by the geometric-tail lemma
(script \texttt{colageometrica}, 5/5). With both insertions placed
the displaced mass is reinserted: the exchange is not blocked. For
$j\le1$ the second insertion dies exactly on the golden razor
($o_1\to2/\varphi$ against $2w^{*}=2/\varphi$, shadow ratio
identically $1$), which is why the hypothesis $j\ge2$ is stated:
that strip is Theorem~\ref{thm:gapwritten}.
\end{proof}

Theorem~\ref{thm:nestedwritten} supersedes the computational
closures of the golden-level nested gaps with two or more occupants
-- in particular the $j=2$ tip $\omega\in[\varphi/2,1)$ entirely, and
profiles of four or more pieces beyond the reduction branch (profile
size is mass here).

\begin{theorem}[Gap lemma, nested $j\le1$, written]\label{thm:gapwritten}
In the nested template with $j\le1$ occupants, $\rho\le\varphi$
implies the exchange does not block, for every width (solid pivot
included), every profile and any extra small rings.
\end{theorem}
\begin{proof}
Steps~1--2 are those of Theorem~\ref{thm:nestedwritten}: $m$ moves
into $u$ by $F$'s certificate, $D_m$ is filled greedily, and either
nothing spills (done) or the spill produces $s'$ and $w^{*}$ with
the exact ligatures
\[
s'\le\min(\Lambda/2,\ \varphi/2),
\qquad
w^{*}<\min(1/\varphi,\ \Lambda-1),
\qquad w^{*}\le\Lambda-2s',
\]
the last two because $B+s'>1$ and because the rowed mass contains
the largest loose piece, which is $\ge s'$. With $j\le1$ the
insertion route of Theorem~\ref{thm:nestedwritten} is not available
(the golden razor), and the witness instead re-packs the whole top
level from scratch as a corona of at most five circles:
$\alpha$, the optional occupant $o_1$, a unit disk carrying the
$D_m$ row, and the two circles $s'$ and $w^{*}$ -- a bounded
family, decided exactly.

\emph{The available radius.} The true pan radius $R$ packs $P$'s
top level, which contains $\{\alpha,o_1,m\}$ (or $\{\alpha,m\}$),
so $R$ is at least every pair sum; moreover any packing of a trio
$\{a,b,c\}$ satisfies the written trio bound
$R\ge\min(R_3,\,M)$, where $R_3$ solves
$\theta(a,b)+\theta(b,c)+\theta(c,a)=2\pi$ and
$M=\min(\max+2\min)$ over the pairs is the stackability threshold:
if $R<M$, every pair is non-stackable, so by the corner lemma of
the duality certificates each angular separation of centres is at
least the mural angle $\theta$, and the three central separations sum to at most $2\pi$ (twice
the extreme separation when the three centres lie in a half-plane,
exactly $2\pi$ in the surrounding case) -- $\theta(a,b)+\theta(b,c)+\theta(c,a)\le2\pi$, i.e.\
$R\ge R_3$ ($\theta$ decreases in $R$). The certificate below therefore works
throughout with the proved floor
$R_0=\max\{\text{pair sums},\ \min(R_3,M)\}$; at the golden point
$(\alpha,o_1)=(2,2/\varphi)$, pair sums, $R_3$ and $M$ all
collapse to $1+\sqrt5$ and the trio sum equals $2\pi$ exactly.

\emph{The corona fits.} For $j=0$ this is an exact algebraic
theorem: at $R=\alpha+1$ the Descartes discriminant of the tangent
pair $\{\alpha,1\}$ vanishes identically, the wall pocket is
$p(u)=u(u+1)/(u^{2}+u+1)$, increasing. Parametrize the cascade floor by
$u=(1+\Lambda)/\varphi$, so $\Lambda=\varphi u-1$ and
$2/\varphi<u\le\varphi$ on the domain $\Lambda\in(1,\varphi]$. With
the coupled mass caps ($s'\le\Lambda/2$, $w^{*}\le\Lambda-1$) the
sub-pocket condition for $s'$ reads $2p(u)\ge\Lambda$, and
\[
2p(u)-\Lambda=\frac{(\varphi-u)\bigl[\varphi u^{2}+2(\varphi-1)u
+(\varphi-1)\bigr]}{u^{2}+u+1}\ge0
\]
since the bracket has positive coefficients: the condition holds
on all of $\Lambda\le\varphi$, with equality exactly at
$\Lambda=\varphi$ and $\alpha$ on its cascade floor -- tangent to
the golden pocket $b_2(\varphi,1)=\varphi/2$ of
Theorem~\ref{thm:DP}'s counterexample. For $w^{*}$ the condition
is $p(u)\ge\varphi u-2$, and $g(u)=p(u)-(\varphi u-2)$ is
decreasing on the domain with
$g(\varphi)=(3-\sqrt5)/4>0$: it holds throughout. When
$\alpha$ exceeds its floor the pocket only widens ($p$
increasing) and both conditions relax; the non-consecutive pair
$(s',w^{*})$ is separated by the unit disk in the cyclic order,
so its own tangency imposes nothing new: the quartet cell is a
theorem. For
$j=1$ the quintet is the single delimited computational step: it is
box-certified over the entire domain by script
\texttt{code/quintetocert.py} (4/4): a branch-and-bound over
$(\Sigma,\alpha,o_1)$ with the exact ligatures of this proof, the
radius handled through the necessity bracket
$R_0=\max\{\text{pair sums},\min(R_3,M)\}$ in \emph{both}
branches (the $R_3\le M$ safeguard is bracketed, not sampled),
five closed per-box certificates -- pocket assignments, a
row-disk merge of the two smalls, the diametral band with a
closed-form $4$-cycle test and a range dichotomy for $w^{*}$, a
\emph{stacked witness} that closes the golden point
$(\varphi,\varphi,\varphi)$ algebraically (exact margin
$1/\varphi^{3}$ in $\mathbb Q[\sqrt5]$), and a rational
\emph{semiarc threshold} $x\le R/(1+f(a)+f(o))$ that collapses to
exactly $1$ at the critical corner (the configuration of
Theorem~\ref{thm:golden} reappearing as a tangent-legal witness)
-- with every acceptance decided by exact rational comparisons
(no tolerances; angle sums by outward directed intervals with
every operation rounded outward, and the arcsines evaluated by
high-precision arithmetic on exact rational arguments, so no
hypothesis on the platform's math library remains), together
with a written trio lemma showing that the trio fits within the
proved floor $R_0$: directly at the pair-sum floor when
$\max(\alpha,o_1)\ge2$, and through $R_3\le M$ when
$\max(\alpha,o_1)\le2$ -- both anchored at the kernel-checked
mirror corner $b_2(2,\sqrt5-1)=1$ -- and
the $\alpha\to\infty$ tail closed by
the exact pocket infimum $ab/(\sqrt a+\sqrt b)^{2}$; the binding
tangent curve of the trio sum $=2\pi$ from $(2,2/\varphi)$
through $(3/2,3/2)$ is covered by the same subdivision with no
special neighbourhoods. It supersedes, for this quintet, the
earlier arc-LP and pocket sweeps (scripts \texttt{code/arcolp.py}
and \texttt{code/bolsillos.py}, 5/5 each, three adversarial
rounds), which in turn superseded the directed maximization of
script \texttt{code/gaplemma.py}. No step uses the width -- the solid
pivot is included -- and the caps are mass bounds, so every
profile and any extra small rings are covered. With the corona
placed, the displaced mass is reinserted: the exchange is not
blocked.
\end{proof}

With Theorems~\ref{thm:nestedwritten} and~\ref{thm:gapwritten}, the
\emph{entire nested template} of case~(b) carries written proofs for
every $j$, every width, every profile and extras -- at the same
epistemic standard as the pan side (the central pan maximization is
box-certified by \texttt{rstarcert}, the D1 budget by
\texttt{insercioncert}). For $j=0$ the domain sweep has since been upgraded to an
\emph{exact algebraic pocket certificate}: at $R=\alpha+1$ the
Descartes discriminant vanishes identically and the wall pocket is
$p(u)=u(u+1)/(u^2+u+1)$, increasing; with the coupled mass caps
($s'\le\Sigma/2$, $w^*\le\Sigma-1$) the sub-pocket condition
factors as $(\varphi-u)\,r(u)\ge0$ with
$r(u)=\varphi u^2+2(\varphi-1)u+(\varphi-1)$
positive-coefficient, so equality holds exactly at $\Sigma=\varphi$
with $\alpha$ on its cascade floor -- where the cap
$s'=\varphi/2$ is \emph{tangent} to the golden pocket
$b_2(\varphi,1)=\varphi/2$ of Theorem~\ref{thm:DP}'s
counterexample; the second cap closes with exact minimum
$(3-\sqrt5)/4$ (the quartet cell is a theorem, its critical corner
the golden pocket itself). The remaining bounded-corona domains --
the $j=1$ quintet and the $k\le2$ hole cells -- are certified by
subdivision as well, on two new exact tools: an \emph{arc-LP
characterization} of small mural coronas (feasibility of the
per-arc requirement system is equivalent to the corona, with the
pairs-fit precondition; its disjoint-family dual serves as pruning
and an exact base-enumeration primal as the criterion -- closed
inequalities, so it certifies \emph{at} tangency), and a
closed-form feasibility test for the critical $4$-cycle
($-\sigma\ge\max(B_1,B_2)$: the two diagonal NS-$2$ deficits
against the slack). The binding frontier turned out to be an
entire \emph{tangent curve}: the trio sum
$\theta(a,o)+\theta(o,1)+\theta(1,a)$ at $R=a+o$ equals $2\pi$
exactly along a curve running from the golden point $(2,2/\varphi)$
through $(3/2,3/2)$ -- where $\theta(3/2,1;3)=\pi/2$ exactly -- to
its mirror image, the boundary of the trio-necessity band; on the
curve the corona is handled \emph{by construction} (the trio is
exactly tangent and the small pieces insert sub-pocket at zero
angular cost), with the golden endpoints -- where pair sums, $R_3$
and the stackability threshold all collapse -- excluded into
certified neighborhoods. Behind it all sits one more exact golden
identity whose algebraic core is kernel-checked in Lean:
$\theta(\varphi,1/\varphi)+\theta(1/\varphi,1)+\theta(1,\varphi)
=\pi$ at $R=2\varphi$, which makes the dangerous $4$-cycle of the
$j=1$ cell exactly tangent and ties its $5$-cycle to the
shared-root corner constant $\pi+4\arcsin(1/\sqrt3)$ (scripts
\texttt{code/arcolp.py} and \texttt{code/bolsillos.py}, 5/5 each,
three adversarial rounds -- one of which refuted and repaired the
arc-LP statement itself). For the $j=1$ quintet this closure is
now superseded by the tolerance-free box certificate of script
\texttt{code/quintetocert.py} (see the proof of
Theorem~\ref{thm:gapwritten}): the tangent curve is covered by
the same subdivision with no special neighbourhoods, and the
golden endpoints are no longer excluded; \texttt{code/arcolp.py}
and \texttt{code/bolsillos.py} remain the certificates of the
other small-corona cells they back.

\subsection*{The exchange assembly and the container port}

\begin{figure}[ht]
\centering
\begin{tikzpicture}[every node/.style={font=\small, align=center},
  level distance=16mm, sibling distance=52mm]
\node[draw, rounded corners, inner sep=5pt]
  {exchange step: destination $u$ of $m$\\ (containers are the pan or
   holes of rings larger than $m$; $u\ne v$)}
  child {node[draw, rounded corners, inner sep=4pt]
    {(a) $u$ the pan\\ \footnotesize
     Thms.~\ref{thm:DP}--\ref{thm:DPr},\\
     \footnotesize Cor.~\ref{cor:DSpan},
     Thm.~\ref{thm:D1written}}}
  child {node[draw, rounded corners, inner sep=4pt]
    {(b) $u$ the hole of a\\ top-level $\alpha$, $v$ the pan\\
     \footnotesize Thms.~\ref{thm:nestedwritten}--\ref{thm:gapwritten}}}
  child {node[draw, rounded corners, inner sep=4pt]
    {(c) the container port\\
     \footnotesize written as a reduction\\
     \footnotesize (this appendix)}};
\end{tikzpicture}
\caption{The case partition of the exchange assembly. Cases (a) and
(b) are covered by the pan and nested theorems; case (c), the
container port, is written as a reduction whose remaining domains are
inventoried at the close of this appendix.}
\label{fig:casetree}
\end{figure}
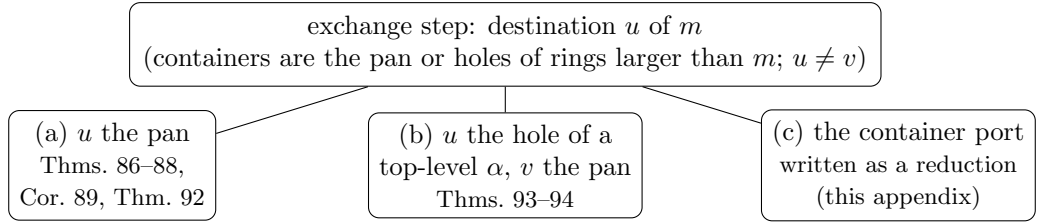

\paragraph{The exchange assembly.} The exchange step of
Theorem~\ref{thm:oblivious} is assembled into a case partition
(a theorem: containers are the unique pan or holes of rings larger
than $m$, and $u\ne v$): (a)~$u$ is the pan -- the template of
Theorems~\ref{thm:DP}--\ref{thm:DPr} plus the closures above;
(b)~$u$ is the hole of a top-level $\alpha$ and $v$ is the pan -- the
nested templates plus the closures above; (c)~the \emph{container
port}. Assembly facts (theorems): $|S|\le1$ never blocks (a single
$\sigma<r_m$ rows into $D_m$; with $S=\emptyset$, $m$ moves under
$F$'s own certificate, containers of larger rings being shared);
feasibility is per-container and positions are existential
(Section~2) -- and this is now a proved \emph{realization-and-repack
lemma}, not a gloss: any forest whose containers each carry a
placement of their children as balls with disjoint interiors is
realized globally by composing translations along root-to-leaf
paths (inductive invariant: a ring's subtree stays in its ball, its
strict descendants in its hole ball; the annulus material is
disjoint from the hole interior; translations preserve distances,
and any isometry of a container placement is just another
placement), so replacing one container's placement -- subtrees
rigid -- yields another realization of the same forest, and a
witness with a \emph{new} assignment composes the per-container
certificates (vacated ball, row, corona, pocket, shadow insertion)
through the same lemma (script \texttt{repack}, 5/5, adversarial
round: exact tangencies, solid rings and subtree isometries all
verified, and the checker proved equivalent to material
disjointness). Repacking any container is therefore a legal
resource -- the pan repack of
Theorem~\ref{thm:DP} is the precedent; and every geometric wall
descends by antitone containment to the intrinsic disk of its
members, so walls port to any container with the corresponding
capacity. Porting \emph{programs} additionally needs their
cohabitation hypotheses, which yields the closure split: if $\alpha$
is a direct member of $v$, case~(b) is inherited verbatim with
$R\mapsto Y-\omega$ (theorem); an extension lemma transports the
$(N)$-heritage walls with the load travelling inside $\sigma_1$
(monolith for placements not using $\sigma_1$'s hole, exact
absorption of the fattened $B_3'$ in the golden-line chain). The
genuinely new configurations -- $\alpha$ outside $v$ -- were closed
by a dedicated campaign (script \texttt{puertocii}, 7/7, two hostile
rounds): exact pincers reduce the light profile to a window that is
empty for $\omega_{\mathrm{ef}}\le3/(2\varphi)$ (the corner
degenerates at $\sigma_2=\tfrac12$ exactly), and inside the window
the pan repack places $\sigma_2$ in the mirror pocket of the
top-level pair, by $\alpha>2$, $T>\sqrt5-1$ and monotonicity of
$b_2$ against the corner $b_2(2,\sqrt5-1)=1$; the heavy profile is
closed by an exact partition pincer ($b_2(4/\varphi,2/\varphi)
=12/(7\varphi)>1$); towers, the $Y\ge\alpha$ branch and the
shared-root configurations (where the top-level pair degenerates)
are closed by a dedicated mural-trio analysis: the exchange lives in
the hole of $\alpha$ with capacity $c=\alpha-\omega\ge\Sigma_S+Y$
(hole tariff), the trio sum decreases in $c$ (exact), and its
supremum over the legal window is the certified corner
$\pi+4\arcsin(1/\sqrt3)\approx5.60<2\pi$ (margin $0.68$; lightness
with $\sigma_1\ge\sigma_2$ forces $\sigma_2\le\tfrac12$, while the
analytic danger threshold is $\sigma_2=\tfrac23$), with the heavy
branch by an exact partition and the deeper/mirrored configurations
by direct sweeps with derived tariffs.

The port itself is now \emph{written as a reduction}: every case
below is a theorem/corollary, lands in a computationally
certified cell, or belongs to one of the residual domains
inventoried honestly in the closing subsection. When $v$ is the
pan and
$\alpha$ sits nested in a tower whose root $t$ is top-level, the
case is a \emph{corollary} of
Theorems~\ref{thm:nestedwritten}--\ref{thm:gapwritten}: those
theorems use, of the configuration, exactly four things -- a top-level
pan family $\{t,o_1,\dots,o_{j'},m\}$ in $P$, the vacated unit disk
$D_m$ at top level ($v$ the pan makes $m$ top-level in $P$), the
legality of $m$'s destination by $F$'s certificate with rigid
subtrees (agnostic to \emph{which} hole receives $m$; the tower is
shared between $F$ and $P$), and the mass ligatures -- and all four
hold with the tower root $t$ in the role of ``$\alpha$'', the
per-rank floors only rising (the tail of $t$ contains the whole
tower; $t\ge\alpha+\omega\ge1+2\omega$ by two circles level by
level; the hole-mass floor is automatic by an exact two-branch
argument), while both certified domains admit arbitrary values above
their floors (large-slack sweeps to $10^4$ with the worst budget at
the floor, and the $t\to\infty$, $\alpha\to\infty$ limits closing by
formula -- an adversarial round of this very reduction demanded and
obtained the large-slack block). When $v$ is the hole of a top-level
$Y$, the case assembles into a theorem from the exact pincers above
together with the two \emph{hole-corona} theorems, which run the
nested template \emph{inside} the hole: the cascade chain among a
hole's occupants yields the same automatic regime ($x_2\ge1+\Sigma$
for $k\ge3$ occupants, $\varphi^2=1+\varphi$; tails are global), and
for $k\le2$ the family is bounded and the armored trio bound
applies, on the hole's own low floors (occupants $\ge1$ with no wall
relation -- a domain the pan boxes do not cover, swept separately
with $10^4$ slack and the $x_1\to\infty$ limit by formula). The
breathing branch $Y\ge\alpha$ needs a \emph{strong breathing lemma}
since the hole mass includes dust: the full cross-tail pincer plus
the wall $\Sigma_S>1$ force $X_Y+\omega>\varphi(3-\varphi)
=2\varphi-1=\sqrt5$, dust is at most $\varphi-\Sigma_S<\varphi-1$ by
the tail of $m$, and $\sqrt5-(\varphi-1)=\varphi$ exactly: the
above-$m$ mass satisfies $X_{>m}+\omega>\varphi$, so $k\ge1$ -- the
golden threshold reappears after discounting dust. The
$\alpha$-corona branch is lighter: lightness gives
$\Sigma_S-\sigma_2<1$ exactly (the rest of $S$ rows into $D_m$ and
only $\sigma_2$ needs mural insertion), and the capacity window
$[\Sigma_S+X_\alpha,\,1+\sigma_2+X_\alpha)$ is exact with $X_\alpha$
cancelling (script \texttt{coronaagujero}, 5/5, adversarial round
with zero counterexamples under 55k independent adversarial
instances).

A \emph{geometric-tail lemma} then makes every one of these shadow
budgets \emph{uniform in the occupant count}. For any cascade
family $t_1\ge t_2\ge\dots\ge t_n$ (global tails, $\rho\le\varphi$)
with $R\ge t_1+t_2$, regime $2s<t_2$ and $t_2\ge1+\Sigma$, the
budget is dominated term by term by an explicit finite series: the
suffix masses $S_i$ decay geometrically
($S_i\ge\varphi S_{i+1}$, again $1+1/\varphi=\varphi$), each tail
piece obeys three exact caps
($t_{3+r}\le\min(t_2,\ S_3/\varphi^r-(1+\Sigma),\
(S_3-1-\Sigma)/(r+1))$), a piece at rank $3+r$ drags a suffix of
mass at least $1+\Sigma+p_{\min}$ so the series has finitely many
terms, and -- the load-bearing link -- the cascade of $t_1$ itself
forces $t_1\ge(t_2+S_3)/\varphi$: a heavy tail forces a large
leader, so the tight pair and the full tail are \emph{jointly
infeasible} (allowing both would push the majorant to $6.37>2\pi$;
with the link its supremum over the whole box is $5.2115<2\pi$,
attained with \emph{zero} gap at the real family
$\{2\varphi,2,2/\varphi\}$ plus the unit disk -- the very corner
every sweep had found). The hypothesis $t_2\ge1+\Sigma$ is forced
by three-or-more occupants via $\varphi^2=1+\varphi$ and is exactly
the lemma's frontier: at $n=2$ the golden razor
($t_2\to2/\varphi$, shadow ratio identically $1$) breaks the
budget, which is the known boundary where the bounded-family
closures take over. Consequently the sweep caps ($j\le6$, $k\le14$,
$k\le12$) drop out of the written theorems' statements: shadow
coverage holds for \emph{every} occupant count (script
\texttt{colageometrica}, 5/5, adversarial round: domination
unbroken under ${\sim}45$k families including the theorems' own
cascade generators). Finally, a
\emph{pocket-$\varphi$ lemma} removes the profile- and dust-count
directions of the scaling question exactly: at $R=o_1+o_2$ the
Descartes discriminant of the tangent pair vanishes identically, the
pocket radius $1/(1/o_1+1/o_2-1/(o_1+o_2))$ is monotone in
$o_1,o_2,R$ and equals $\varphi$ \emph{exactly} at the cascade
minimum $(2\varphi,2,2\varphi+2)$, while the total small mass is
$\le\varphi$ by the tail of $m$: the entire profile and dust row
into a single pocket ($p$ and $k$ free; script \texttt{escala}, 5/5,
with the nested template verbatim and a geometric non-invasion
check).

\subsection*{The multipiece campaign and the vacuity closures}

The shared-root \emph{multipiece} sweeps are certified by a
pessimistic-corner engine over the
arc-LP's sufficiency direction (valid for every $k$): a
term-by-term majorant matrix (each wall angle bounded at its own
corner, participants at their ceilings \emph{including} their
contribution to the capacity, the rest at their floors -- valid
since $\partial\theta/\partial a$ carries the sign of
$R-a-b=\Sigma_S+\text{rest}>0$), a maximum-slack LP whose returned
gap vector is re-verified in plain floating point, and an
\emph{antipodal} criterion for the $\Sigma_S\to1$ edge ($Y$ and $m$
at distance exactly $\pi$, strict by algebra, with the remaining
pieces as an interval -- hence totally unimodular, hence exact --
path system on one half-circle, where the circular LP's margin
degenerates even at real points). This certifies the explicit
extra occupants ($1$--$3$ corona pieces up to size $Y$, coronas of
$4$--$6$ pieces, entire $\sigma_2<1$) over the sweep box, and the
mirrored branch beyond its $X=0$ cut in the $X_Y=0$ direction
(shifted windows with $X_\alpha,X_z,X_m>0$, dual per-box bounds
coupling the diametral floor $c'\ge1+z$, certified at margins both
$0.3$ and $0.35$). The heavy-profile partition -- multipiece $A$
with \emph{unbounded} cardinality -- is closed by a
\emph{reduction lemma with a golden threshold}, exact in
$\mathbb{Q}(\sqrt5)$: taking $B^{*}$ the best subset of $S$ of mass
$\le1$ (rowed into $D_m$) and $A=S\setminus B^{*}$ with
$\beta=\Sigma B^{*}$, maximality gives $a+\beta>1$ for every
$a\in A$, $\beta\ge\sigma_1\ge\max A$ (hence $a\le\varphi/2$), and
\emph{at most four} pieces of $A$ exceed
$t_0=(\varphi-1)/4=1/(4\varphi)$ -- a dichotomy at
$\beta^{*}=(9-\sqrt5)/8$ whose critical identity
$5t_0=\varphi-\beta^{*}$ is exact, with the twin identity
$4t_0=\varphi-1$ making four large pieces and dust
\emph{incompatible} -- while the remaining dust (pieces $\le t_0$,
present only for $\beta>\beta^{*}$, mass $\mu<5t_0$) is paid by
\emph{mass, not count}: convexity of $\arcsin$ gives
$\theta(a,b)\le\arcsin f(a)+\arcsin f(b)$, so a monotone dust chain
costs at most $\pi\mu/(R-t_0)$ regardless of how many pieces carry
the mass. The heavy mural is then $\{Y\text{ or }z,m\}$ plus at
most four large pieces plus one aggregated dust block, and the
two-sided antipodal engine certifies both heavy branches over
their sweep boxes -- the direct one entirely; the mirrored one on
its $X=0$ cut -- and the slot-reduction closure \texttt{lemaA3}
now re-certifies both heavy murals with $Y$, $z$ and $|A|$
unbounded (the golden-chain lemma settling the tangency corner); and now, closing the mirrored cell, \emph{beyond}
that cut: within the dust convention the global pivot tail caps
\emph{every} extra mass at once,
$\Sigma_S+X_m+X_\alpha+X_z+\Sigma X_Y\le\varphi$, so each dust
mass is below $\varphi-1$ -- the old sweep ceilings
$X_\alpha\le1.5$, $X_z\le1$ cease to be ceilings -- and a
frozen-criterion sweep certifies the heavy mirrored branch with
\emph{all} dusts free: $X_m$ eliminated as a dimension (its
majorant only ever uses its lower endpoint), $X_z$ eliminated by
credit segments, the dust block made \emph{splittable} across the
two half-circles (a greedy halving of any multiset of pieces
$\le$cap leaves $|m_1-m_2|\le$cap, so two sub-blocks of mass
$M/2+\text{cap}/2$ cover every real split -- the atomic block
saturated one side), and the costly $\Sigma_S\to1$ edge collapsed
by a \emph{degeneracy reduction}: there the heavy wall forces
$\sigma_2\le\varepsilon_0$, the best-subset row fills
$\beta\ge1-\varepsilon_0$, and both windows have width
$\le\varepsilon_0$, so fourteen dimensions reduce to four and the
whole strip $\Sigma_S\in(1,1.016]$ certifies in $85$ boxes
(six $\Sigma_S$-bands, $\sim32\,100$ boxes, cover the rest).
What remains sampled there: only the model-conditional occupant
channel -- the width cap $\omega\le1.6$ has since been removed: the
tail $\omega>1.6$ is box-certified at once for every width by
monotone-limit caps (the fast pieces grow linearly in $\omega$, the
receiving hole satisfies $c\ge1+z$, and each corona term is bounded
uniformly in $\omega$; script \texttt{espomegacola}, 5/5,
adversarial round: confirmed with no soundness gap)
(scripts \texttt{esppesada} and \texttt{espfinal}, 5/5 each,
adversarial rounds; the latter's referee found one surgical
soundness gap -- $\sigma_1$ can be excluded from the best subset,
so the strip's dust cap is $2\varepsilon_0$, not $\varepsilon_0$
-- repaired with identical counts, and proved the greedy-halving
invariant). Also settled: the $X_Y>0$ mirrored corona -- whose exact danger
geometry is now settled by a two-round exploration: at the
coexistence floor $c'=1+z$ the pair $(z,m)$ is exactly diametral,
and an extra piece $x$ of $X_Y$ fits beside $m$ precisely when
$x\le p(z,1;1+z)=z(z+1)/(z^2+z+1)$, the \emph{Descartes pocket of
the diametral pair} (degenerate form, exact at $R=a+b$; the
would-be sliver $x\in(p,1)$ has width $1/(z^2+z+1)$ --
$1/(2\varphi^2)$ at $z=\varphi$, where also
$p(\varphi,1;\varphi+1)=\varphi/2$ exactly, the golden pocket of
Theorem~\ref{thm:DP}'s counterexample, already met as the critical
point of the duality certificates above); a first adversarial round
confirmed this algebra by hand, and a second one then proved the
sliver \emph{vacuous}: the global tail of the pivot counts every
smaller piece, so $\Sigma_S+X_m+\Sigma X_Y\le\varphi$, which is
incompatible with the light-profile wall $\Sigma_S>1$ and
$x>p\ge0.91$ ($100\%$ of the previously sampled ``variety''
violates it, minimum $1.96$ against $\varphi$), and independently
the floor itself is \emph{rigid} -- the diametral pair is forced
tangent and the largest hole left in the container is exactly the
pocket, so no legal configuration ever placed such an $x$ murally
at the floor. The positive consequence: every \emph{legal} $X_Y$
piece measures at most $\varphi-\Sigma_S<\varphi-1<p(z)$
throughout the domain -- \emph{always sub-pocket} -- so the
light mirrored $X_Y>0$ cell \emph{closes} by vacuity plus a
certified $k$-piece corona: a nine-dimensional
branch-and-bound over the sweep box with the dust mass
$\mu\le\varphi-\Sigma_S-X_m$ as its own dimension (each piece
$\le\mu$ by pigeonhole) certifies the corona
$\{z,D_m,\sigma_2\}$ plus the dust block -- paid by mass, monotone
chain at most $\pi\mu/(c'-\mu)$, any $k$ -- throughout the domain
($1\,329$ boxes; end-to-end sanity with explicit $1$--$4$-piece
partitions, $250/250$ across seeds), all \emph{within the dust
convention} $X_Y<r_m$. The ``occupant $\ge r_m$ inside $Y$''
channel, long model-conditional with its tariff underived, now
has a \emph{derived} tariff -- two node walls: in the light
profile blocking implies $x<\sigma_2+\omega+X_x$
(Lemma~\ref{lem:DBo} ported to the mirrored cell: the opposite
disk packs $\{\sigma_2\}\cup\mathrm{children}(x)$ into $x$'s
hole and the row $S\setminus\{\sigma_2\}<1$ fills $m$'s vacated
disk; nothing else moves), and in \emph{every} profile
$x<\omega+\Sigma_S-1+\sigma_1+X_x$ (a greedy partition of $S$
sends $B$ into $x$'s hole and $A\le1$ to the row) -- and the
light channel is \emph{closed entirely on $v$ and on the depth-one
tower cut}: the exact tie $x=r_m$
is $\rho$-illegal (the first-copy convention makes the first
copy's tail collect the twin \emph{plus} all of $S$ and the
dust, $\ge1+\Sigma_S>2>\varphi$ -- the referee's find, which
dissolved the residual ``twin sheet'' the campaign had first
delimited: one more boundary vacuity), the band
$x\in(r_m,2/\varphi)$ is $\rho$-illegal by the tail of $x$
itself, the high band certifies by subdivision ($79\,277$ boxes,
using the trio floor of $\{z,m,x\}$ from a new \emph{crescent
lemma} -- exact angular windows around a near-mural giant,
matching the classical three-circle optimum $1+2/\sqrt3$ to
$10^{-13}$ -- and a \emph{dust-pooling} witness: the blocking
walls with the $X$-masses in situ are only one assignment, and
the witness repools $\sigma_2$ and the obstructing dust across
the free disks of $u$, of $z$'s hole and of $m$'s vacated disk),
and above the node wall the cell unblocks outright (script
\texttt{espcanal}, 5/5, adversarial round run on rigor criteria
alone); the heavy profile with $x$ explicit is certified as
well, by a \emph{dimension reduction} that the extra ring itself
makes possible: with $x\ge(1+\Sigma_S)/\varphi$ forced by $x$'s
tail, the capacity floor is far above espfinal's $1+z$, so the
whole partition load $A=S\setminus B^{*}$ is paid \emph{by mass}
as a splittable block with the exact cap
$\min(\beta,\varphi/2,\Sigma A)$ (every piece of $A$ is bounded
by each term: $\beta\ge\sigma_1\ge\max A$ even when $\sigma_1$
is excluded from $B^{*}$, maximality plus $\Sigma_S\le\varphi$
gives $a\le\varphi/2$, and a piece cannot outweigh its
multiset), the six explicit-partition dimensions collapse to
$\beta$ alone, and a new exact prune $\beta>\tfrac12$ (the
best-subset greedy stops only when every remaining piece exceeds
$1-\beta$, so $\sigma_1>1-\beta$ with $\sigma_1\le\beta$)
dissolves a phantom fine-dust band: five $\Sigma_S$-bands
($144\,337$ boxes) plus the tower cut ($15\,571$ boxes) certify
the heavy channel entirely (script \texttt{espcanalp}, 5/5,
adversarial round: the referee re-derived the reduction theorem
independently and probed it with $20\,000$-instance fuzz, and found the uncovered
tower cut and a mislabeled band edge, both repaired -- the tower
cut is now \emph{certified}, not declared); what remains of the
channel is the mid band at tower depth $\ge2$ and the structural
exclusion of $u$'s content; and
the heavy mirrored branch's dust is now \emph{merged} with the
reduction-lemma block as well -- on the $X=0$ cut, the $X_Y$ dust
(mass $\le\varphi-\Sigma_S$) joins the partition dust in a single
mass-paid chain, the hidden $\mu_Y$ range is handled by an
eight-segment ceiling/tail-credit split (a full waiver of the
dust's own tail term creates a \emph{phantom} tangency exactly at
$\Sigma_S\to1$, where that tail is the rescue), and the heavy wall
$\Sigma_S\ge1+\sigma_2$ prunes the light region already certified
-- whose exact boundary $\Sigma_S=1+\sigma_2$ is \emph{vacuous}
($S\setminus\{\sigma_2\}$ has mass exactly $1$, so $\beta=1$ and
the $\alpha$-window is empty): the mirrored cell with $X_Y>0$ is
closed in both profiles on its declared cuts, certified over six
$\Sigma_S$-bands ($\sim34\,800$ boxes) (scripts \texttt{espxy},
\texttt{espvals}, \texttt{espkp} and \texttt{esppesada}, 5/5 each,
four adversarial rounds; the second refuted a proposed
vacated-ball witness -- the
two modes of the repack lemma, positional and fresh-certificate,
must not be mixed in one container -- and found the missing
global-tail wall itself; the third confirmed the certificate with
no soundness gap and forced the convention to be declared); the caps of the sweep
boxes themselves have since been removed by the tail certificates
\texttt{espomegacola} and \texttt{r2bcolas} (scripts \texttt{r2bmulti} and \texttt{areduccion},
5/5 each, adversarial rounds: no soundness gap -- $1\,460$ directed
tests for the engine, hand-redone algebra and adversarial profiles
for the lemma; the rounds corrected domain claims to the sweep
boxes, replaced the solver-trusting margin by a float
re-verification of the returned gap vector, fixed a degenerate
half-circle that silently re-introduced the excluded antipodal
pair, and contributed the $4t_0$ incompatibility). The near-equal-tops
duality gap ($R_{\mathrm{fit}}/R_{\mathrm{lb}}\le1.0116$) was
first re-delimited by a \emph{conditional duality lemma}: for families whose pairs are all
non-stackable at the radius in play, the arc-LP radius of
Appendix~\ref{app:campaign} is the exact minimal mural-corona radius
on both sides, by compactification for necessity and the closed
arc-LP for sufficiency, reproducing the classical equal-circle
optima $1+2/\sqrt3$ and $1+\sqrt2$ at tangency; on the synthetic
non-stackable domain the true load-versus-tops interval has corner
supremum $1.082$, evaluated over a deterministic corner mesh --
now an \emph{abstract} arc-LP statement: a subsequent global-tail
audit showed that every sampled gap instance of the realistic
generator is $\rho$-illegal ($0/30$ under the full tail
legality), by a one-line \emph{forbidden-triple} fact: three
pieces whose second and third are $\ge(\varphi/2)\cdot$largest
already violate $\rho\le\varphi$ (their tail sum is
$\ge2\cdot(\varphi/2)=\varphi$ times the largest), while the
duality gap needs at least three near-equal tops (ratio
$0.9>\varphi/2$) -- and the complementary case, a third piece too
small for the forbidden triple, is closed by a \emph{forced
sub-pocket} lemma: full tail legality gives
$t_3<(\varphi-r_2)t_1-2$, the dominant pair's pocket at its pairs
floor is $q(r_2)t_1$ with $q(r)=r(1+r)/(1+r+r^2)$ (the same
degenerate-pocket function, scaled by the smaller piece), and
$(\varphi-r_2)t_1-2\le q(r_2)t_1$ holds throughout
$t_1\le2/(\varphi-r_2-q(r_2))$ -- at the cell's edge $r_2=0.9$ the
ceiling is $23.0$, growing to infinity at the real root
$r^{*}=0.9637$ of the golden cubic
$r^3+(2-\varphi)r^2+(2-\varphi)r=\varphi$, while the generator's
true domain is $t_1\le6.64$ -- so the third piece always inserts
freely, the trio sits at its pairs floor, and (checked against the
confinement strengthening of the lower bound, $100\%$ of $5\,300$
cell instances) $R$ certificates match: the cell has \emph{no}
legal instance of the phenomenon ($0$ gaps across $\sim11\,000$
adversarial legal samples, including sub-$m$ third pieces and
four-top profiles). The old $1.0116$ residue is therefore
\emph{withdrawn}: $\rho\le\varphi$ itself forbids the
configurations that carried it, and only the abstract arc-LP
statements (the conditional duality lemma and the synthetic
$1.082$ corner) remain; and the converse ``gap $\Rightarrow$
cell'', long declared open with only the empirical boundary
$0.9$, is now closed in an \emph{exact} form: a certified-gap
family, after an exact \emph{peeling} step (any piece below the
straight-wall pocket $p_\infty(a,b)=ab/(\sqrt a+\sqrt b)^2$ of
the core's smallest pair inserts interior to every wall corona
at every radius -- the Descartes pocket \emph{shrinks} as $R$
grows, so the flat-wall limit is the uniform floor -- leaving
the wall radius unchanged), must either retain a
\emph{stackable} pair at the arc-LP radius (some $u\ge t$ with
$R^{*}\ge u+2t$: the cyclic certificate loses that pair's angle)
or, with four or more core pieces, have a non-adjacent pair
dominating both complementary arcs of the minimal-sum cyclic
order at an intermediate radius; the complement is gap-free by
two small theorems (a three-piece non-stackable family reduces
the arc-LP to boxes, $\Sigma\theta\le2\pi$ exactly the cyclic
condition; and without stackable pairs the corner analysis of
the cyclic engine gives $\gamma_{\min}=\theta_w$ exactly, so
non-domination would make the system feasible below $R^{*}$).
Across $4\,000$ sampled families every one of the $1\,266$
certified gaps (gap-certification threshold $2\cdot10^{-6}$; the
band $(10^{-9},2\cdot10^{-6}]$ is declared unswept) lands in the
exact cell -- all of them via the
stackable-pair route: the true anatomy of the duality gap is a
stackable grain that fails to peel, not the near-equal tops
themselves (a peeled four-equal-tops family has \emph{no} gap;
script \texttt{f3converso}, 5/5).
The scripts
\texttt{f3cierre}, \texttt{auditcolas} and \texttt{f3vacio} (5/5
each) passed adversarial rounds: the first refuted the original
``closure'' claim and its sampled $1.030$ bound, the second
refuted a partial-tail audit and supplied the forbidden-triple
mechanism, the third corrected the declared generator domain
($6.64$, not $5.1$) and verified the confinement attack. The
shared-root \emph{trio} itself is certified by subdivision over its entire legal domain in the direct
orientation -- light interior, both frontiers and every depth at
once, via the free-leader reduction -- and over the $X=0$ cut of
the mirrored one, where the same branch-and-bound also proves the
mass wall $\Sigma_S\le\varphi$ \emph{necessary}: without it the
corner $\sigma_2\to1$ is legal and its corona genuinely fails, the
$(z,m)$ pair going exactly diametral (script \texttt{r2bcert}, 5/5,
adversarial round; the pan-repack legality behind the
mirror-pocket pincers is likewise no longer on this list: it is the
realization-and-repack lemma above). The optimization asterisk of
the \emph{bounded-family} domain sweeps is, for the shadow budgets
themselves, closed: since every term of a
shadow budget is monotone in each piece radius and in $s$, and
antitone in $R$, corner evaluation gives a valid upper bound on any
box, and an exhaustive branch-and-bound over the majorant's
hypothesis box (with the exact prunings $t_2\ge1+\Sigma$,
$u\le\varphi t_2$ and the $t_1$-link, the $\pi$-cap for
$t_1\to\infty$, and a normalized form with an analytic geometric
remainder for $t_2>10^3$) \emph{certifies}
$\sup G\le5.25$ on the principal box and $<2\pi-0.05$ globally --
margins $0.98$ and $0.19$ radians, twelve orders of magnitude above
IEEE rounding (script \texttt{optimizacion}, 5/5, adversarial round
with a repaired normalized-link bound and a certified-honesty
control: tightening the target below the true supremum $5.2115$
makes the branch-and-bound stall at exactly that value).

\subsection*{The extra-ring cells of the occupant channel}

The occupant channel ($\ge r_m$ inside $Y$'s container) has its
tariff \emph{derived} (the two node walls) and is closed in both
profiles on $v$ and on the depth-one tower cut; the mid band
$x\in[2/\varphi,\sigma_2+\omega+X_x)$ at tower depth $\ge2$ is closed
for chain towers in both profiles (scripts
\texttt{code/esptorre.py}, \texttt{code/esptorrep.py}: the tower
enters as mass on favourable floors only, and the naive reduction
to depth one fails because $t_1$'s children are not dust). The
inherited width ceiling $\omega\le1.6$ of the channel cells is
removed: the six criteria are $\omega$-generic, the mid range
$\omega\in[1.6,40]$ re-runs them over extended roots ($\sim$5M
boxes), and the tail $\omega\ge40$ closes by self-contained cap
arguments with margins $0.4$--$0.8$~rad
(\texttt{code/espomegacanal.py}).

The $k\ge2$ extra-ring cells are certified on their declared
domains (scripts \texttt{code/lemaA4.py} and
\texttt{code/lemaA5.py}; nine adversarial rounds, four
intermediate refutations repaired): \emph{light leaf extras} for
$\omega\le1.25$ with the full nested-mass range -- $W_z\le34$ by
domain, and the $W_z>34$ tail by a mode in which the nested mass
counts toward $z$'s tail, so $\rho\le\varphi$ both bounds the
uncovered sector by infeasibility and couples $z\le C_0+W_z$; the
piecewise-linear lower bound on the hosting capacity then has an
upper piece of slope $2/\varphi>1$ where the interior critical
points of the pair bound are \emph{maxima} -- the reverse of the
generic-slope case -- with a closed-form critical point, so the
supremum over the unbounded $z$-ray is computed from finitely many
candidates. \emph{Parent extras} (nested inside other extras of
$v$) are certified on $\omega\le1.05$, $W_v\le8$, $W_z\le34$
(extended node walls $T+W_z$, positional caps for $j\ge6$). The
width range $\omega\in[1.6,2]$ is certified for $j_v\le1$
configurations via a movable antipodal exemption in the placement
engine (any single-clamped pair $(z,\cdot)$ is exemptible by
coexistence), with the finding that the saturated-diametral
obstruction family is \emph{width-invariant}: the $j_v\ge2$
residues of all width ranges are one object. The \emph{heavy}
profile ($\Sigma_S\ge1+\sigma_2$) with $k\ge2$ extras is
certified on $\omega\le1.15$ over a 16-band manifest (the A7
heavy node wall applied per-extra, the partition collapsed to its
best-subset mass, exact ladder variants up to $j\le8$). The
light-leaf width cut rose from $1.15$ to $1.25$ via a
\emph{necessity} wall: the floor capacity $c^*$ of the mandatory
corona set, computed by bisection over the repository's audited
order-sum refuter (radial stacking handled). Every mirrored-cell
closure stays within the dust convention (every extra piece is
dust, of size $<r_m$; the extra \emph{masses} are bounded by the
global tail, not by $r_m$).

\subsection*{Historical coverage limits of the computational campaign}

The list below records the coverage limits of this particular campaign,
not remaining cases of the universal disk theorem.
Theorem~\ref{thm:golden-global} now settles the golden guarantee by a
different argument, independently of every computational closure below.
Each item retains its original epistemic label; bounded or sampled
claims are not upgraded to universal certificates by the new theorem.
\begin{enumerate}
\item[(i)] \emph{Swept:} the occupant-count direction $j$ of the
\emph{computational} scaling closures, swept to $j\le9$ (pan) and
$j\le8$ (nested) with uniform tangent duality and the exact
geometric growth $T_{k-1}\ge\varphi\,T_k$ of the cascade; the
\emph{written} theorems' occupant counts are closed by the
geometric-tail lemma.
\item[(ii)] \emph{Box-certified, on declared boxes:} every
subdivision certificate is rigorous over its declared box; the
former width, mass and piece-count ceilings are removed by the
tail certificates and the slot-reduction engine
(\texttt{code/espomegacola.py}, \texttt{code/r2bcolas.py},
\texttt{code/lemaA.py}--\texttt{code/lemaA3.py}). What remains:
the engine's mirrored-ESP corona boxes hold on their swept ranges
only.
\item[(iii)] \emph{Declared residues of the extra-ring cells:}
the band $\omega\in[1.25,1.6]$ and the width-invariant
$j_v\ge2$ family for $\omega>1.6$; parent extras with $W_v>8$ or
$\omega>1.05$ or $W_z>34$; the heavy profile outside
$\omega\le1.25$; branched towers (chain towers are closed in
both profiles); and the structural exclusion of $u$'s content (the
one model-conditional piece). The high-width, large-$W_z$ leaf
branch is certified on $\omega\in[1.6,6]$, $W_z>34$, excluding
the declared $j_v\ge2$ family; its $\omega>6$ continuation remains
open. The low-$W_z$, $j_v\le1$ branch is certified only through
$\omega=2$. Recorded truth probes found no violations; they do
not establish universal coverage.
\item[(iv)] \emph{Sampled:} the supporting sweep of the exact
``gap $\Rightarrow$ cell'' converse (closed in exact form;
the domination branch is an unexercised prediction; the realistic
cell itself carries no legal instance).
\item[(v)] \emph{Exposition program:} analytic proofs replacing the
remaining computational certificates. Their optimization budgets
are already certified, as recorded in Appendix~\ref{app:verifmap}.
A rigorous certificate on its
declared domain is already a proof; replacing it is not itself a
missing mathematical case. This is an exposition direction of Open
Problem~\ref{op:assembly}, not a gap in the global golden proof.
\end{enumerate}
Everything else in this appendix is a written proof, an exact
identity in $\mathbb{Q}(\sqrt5)$, a vacuity, or a subdivision
certificate recorded in the verification map
(Appendix~\ref{app:verifmap}).

\section{Verification details for Theorem~\ref{thm:twins}}\label{app:verif}

All feasibilities in the twin instances are two-circle-exact except three
facts, proved here; every numerical claim is also reproduced by the scripts in
the repository (\texttt{code/gemelas.py}).

\paragraph{(V1) $\{10,\,4.99,\,4.50\}$ does not pack in a disk of radius 15.}
As in Lemma~\ref{lem:cap}: $|c_A|\le5$, $|c_X|\le10.01$, $|c_Y|\le10.5$,
$|c_A-c_X|\ge14.99$, $|c_A-c_Y|\ge14.5$, $|c_X-c_Y|\ge9.49$; the first two
give $a:=|c_A|\ge4.98$. With $u=-c_A/|c_A|$,
$c_X\!\cdot\!u\ge(124.5-a^2)/(2a)\ge9.95$ and
$c_Y\!\cdot\!u\ge(100-a^2)/(2a)\ge7.5$ (both decreasing in $a$, evaluated at
$a=5$), with transverse norms at most
$\sqrt{10.01^2-9.95^2}\le1.095$ and $\sqrt{10.5^2-7.5^2}\le7.349$.
Hence $|c_X-c_Y|^2\le10.01^2+10.5^2-2(9.95\cdot7.5-1.095\cdot7.349)
\le77.30$, so $|c_X-c_Y|\le8.80<9.49$.

\paragraph{(V2) With the pan pair $\{10,5\}$ placed, no third circle of radius
$z\in\{4.74,4.76\}$ fits.}
From $|c_A|\le5$, $|c_B|\le10$ and $|c_A-c_B|\ge15$ the triangle inequality
forces equalities: the pair is diametrically rigid, $c_B=10u$ for a unit
vector $u$ and $c_A=-5u$. For a third circle of radius $z$:
$c_z\!\cdot\!u\ge\bigl((10{+}z)^2-(15{-}z)^2-25\bigr)/10$, which equals $8.8$
for $z=4.76$ and $8.7$ for $z=4.74$; then
$|c_z-c_B|^2=|c_z|^2-20\,c_z\!\cdot\!u+100\le(15-z)^2-20\,c_z\!\cdot\!u+100$,
giving $|c_z-c_B|\le5.38<9.76$ and $\le5.60<9.74$ respectively.

\paragraph{(V3) A feasible triple.}
The circles of radii $10$, $4.76$, and $4.74$ pack in a disk of
radius $15$. An exact boundary configuration is $c_A=(-5,0)$,
\[
c_X=\bigl(8.8,\ \sqrt{27.4176}\,\bigr),\qquad
c_Y=\bigl(8.7,\ -\sqrt{29.5776}\,\bigr).
\]
Then $|c_A|=5$; $|c_X|^2=77.44+27.4176=104.8576=10.24^2$ and
$|c_Y|^2=75.69+29.5776=105.2676=10.26^2$, so $|c_X|=15-4.76$ and
$|c_Y|=15-4.74$ (both on the boundary; the abscissas are the exact
tangency values $10.24\cdot\tfrac{55}{64}=8.8$ and
$10.26\cdot\tfrac{145}{171}=8.7$);
$|c_A-c_X|^2=13.8^2+27.4176=217.8576=14.76^2$ and
$|c_A-c_Y|^2=13.7^2+29.5776=217.2676=14.74^2$ (exact tangencies); and
$|c_X-c_Y|^2=0.01+(\sqrt{27.4176}+\sqrt{29.5776})^2>113.9>9.50^2=90.25$
with ample margin.

\paragraph{(V4) Hole facts.} The hole of the $10$ has radius $9.495$;
$4.99+4.50=9.49\le9.495$ (the pair fits, two-circle exact) while
$4.76+4.74=9.50>9.495$ and $5+4.50=9.50>9.495$ (blocked, two-circle exact).

\end{document}

%% file: golden_global.tex
\subsection{The uniform golden guarantee}\label{sec:goldenglobal}
\begin{theorem}[Global golden guarantee]\label{thm:golden-global}\label{conj:golden}
For every finite inventory of strictly decreasing positive outer radii
in a disk, every descending greedy execution returns the lexicographically
maximum feasible subinventory when $\rho\le\varphi$. The assertion
also holds with independent hole radii $0\le h_i<r_i$.
\end{theorem}
The greedy semantics are those of Section~\ref{sec:model}: previously
chosen parents stay fixed, but sibling coordinates may be rearranged.
A rejection means that no container admits the extended forest.
The proof works for an arbitrary finite inventory; it does not enumerate
successive sizes. Write $\beta(a,b)=ab(a+b)/(a^2+ab+b^2)$.

\subsection{Tail bounds and a geometric pocket}

A list of disks whose radii sum to $s$ fits in a disk of radius $s$:
place them consecutively along a diameter. This row construction may
replace a disk envelope without changing its surrounding packing.

\begin{lemma}[Two tail bounds]\label{gt:tails}
If $p+U\le\varphi m$ and $U\le\varphi p$, then $U\le m$.
If $a>b>0$, $T\le\varphi b$, and $b+T\le\varphi a$, then
$T\le2\beta(a,b)$.
\end{lemma}
\begin{proof}
For the first assertion,
$(1+\varphi)U\le\varphi p+\varphi U\le\varphi^2m=(1+\varphi)m$.
For the second, normalize $b=1$, put $x=a/b>1$,
$g=x(x+1)/(x^2+x+1)$, and write $t=T/b$. We have
$t\le\min(\varphi,\varphi x-1)$. The identities
\begin{align*}
 2x(x+1)-\varphi(x^2+x+1)
   &=(x-\varphi)((2-\varphi)x+1),\\
 (x^2+x+1)+2x(x+1)-\varphi x(x^2+x+1)
   &=(\varphi-x)\bigl(\varphi x^2+(2\varphi-2)x+\varphi-1\bigr)
\end{align*}
give $2g\ge\varphi$ when $x\ge\varphi$, and
$1+2g\ge\varphi x$ when $1<x\le\varphi$. Thus $t\le2g$ in
both cases. Rescale by $b$.
\end{proof}

\begin{lemma}[Balanced rows]\label{gt:split}
Let a nonempty finite list of positive radii have total $T$ and largest
member $p$. After reserving $p$, its remaining members can be divided
into two groups, each of total at most $T/2$.
\end{lemma}
\begin{proof}
Assign each remaining member to the currently lighter group. Both load
differences remain at most $p$, since the added member is at most $p$.
At the end the loads $u,v$ satisfy $u+v=T-p$ and $|u-v|\le p$;
hence $2u,2v\le T$. Each group fits its row envelope.
\end{proof}

Place disks $a\ge b>0$ in a disk of radius $a+b$, with centers
$(-b,0)$ and $(a,0)$. The upper and lower pockets have inscribed
disks of radius $\beta(a,b)$, each tangent to the two disks and the wall.
Their centers are
\[
 \left(\frac{(a-b)(a+b)^2}{a^2+ab+b^2},\ \pm2\beta(a,b)\right).
\]
These coordinates verify all three tangencies and separation of the two
pockets directly. We need a stronger fact about a single pocket.

\begin{lemma}[Two unequal disks in one pocket]\label{gt:arbelos}
If $0<r,s\le\beta(a,b)$ and $r+s\le b$, the two disks of radii
$r,s$ fit simultaneously in the upper pocket. The lower pocket remains
available. A zero radius may be omitted.
\end{lemma}
\begin{proof}
Put $L=a+b$. A disk of radius $x\le\beta(a,b)$ tangent to the wall
has its center in the upper half-plane at an admissible angle exactly in
\[
 \left[2\arcsin\sqrt{\frac{bx}{a(L-x)}},\quad
 \pi-2\arcsin\sqrt{\frac{ax}{b(L-x)}}\right].
\]
The law of cosines gives the two endpoints. The interval is nonempty
exactly when $x(a^2+b^2)\le ab(L-x)$, equivalent to $x\le\beta(a,b)$.
Place $r$ at its lower endpoint and $s$ at its upper endpoint, and set
\[
 A^2=\frac{as}{b(L-s)},\qquad B^2=\frac{br}{a(L-r)},\qquad A,B\ge0.
\]
Their angular separation is $\pi-2\arcsin A-2\arcsin B$; the
required separation is $2\arcsin(AB)$. It suffices that
$A^2+B^2+3A^2B^2\le1$: this implies
$\sqrt{1-A^2}\sqrt{1-B^2}\ge2AB$ and therefore
$\arcsin A+\arcsin B+\arcsin(AB)\le\pi/2$.
After multiplication by the positive denominator $ab(L-r)(L-s)$,
the margin is
\[
 L^2\bigl[a(b-r-s)+(a-b)r\bigr]+(a-b)^2rs\ge0.
\]
The disks lie wholly above the diameter: its intersection with the
container is covered by $a,b$, so a disk disjoint from them cannot
cross it. The lower pocket is consequently free.
\end{proof}

\subsection{The minimal triple and its opposite gap}

\begin{lemma}[A fully tangent minimal triple]\label{gt:minimal}
For $a\ge b\ge c>\beta(a,b)$, the three disks admit a smallest
containing disk of radius $R_0>a+b$, in which they are pairwise tangent
and all tangent to the wall. With
\[
 A=1/a,\quad B=1/b,\quad C=1/c,\quad t=\sqrt{AB+AC+BC},
\]
its radius satisfies $1/R_0=2t-A-B-C$.
\end{lemma}
\begin{proof}
Existence of a minimum follows by compactness, using a row as an upper
bound. In any feasible placement of three disks, each disk can be moved
to the wall while the other two are fixed. For its center $x$, choose
a nonzero direction $v$ such that $(x-x_j)\cdot v\ge0$ for the
other two centers. Two closed half-planes through the origin have such
a common direction in the plane. Distances to both other centers are
nondecreasing along the ray. Move to the wall and repeat for the others.

For wall-tangent disks with pair sums at most $R$, let
$\theta_{ij}\in[0,\pi]$ denote their minimum angular separation.
Three such positions exist exactly when
\[
 \theta_{ab}+\theta_{ac}+\theta_{bc}\le2\pi.
\]
Necessity follows by summing the three directed cyclic gaps, each at
least its corresponding minimum. Conversely choose gaps in
$[\theta_{ij},\pi]$ summing to $2\pi$; their endpoints have sums
at most $2\pi$ and $3\pi$, so this is possible. Each gap is then
the smaller angle for its pair.

Two-disk necessity gives $R_0\ge a+b$. At equality the two largest
are forced into the diametral arrangement. Move $c$ to the wall while
fixing them and use the interval in Lemma~\ref{gt:arbelos}; it would
give $c\le\beta(a,b)$, a contradiction. Thus $R_0>a+b$. All three
minimum angles are continuous and strictly between $0$ and $\pi$
near $R_0$. Their sum equals $2\pi$, since a strict inequality would
permit decreasing $R_0$. Taking these gaps realizes every tangency.
Descartes' equation for the oriented curvatures $A,B,C,-1/R_0$ gives
the stated formula; its other root has $1/R_0<0$.
\end{proof}

\begin{lemma}[Compensation by the other gap]\label{gt:opposite}
In the placement of Lemma~\ref{gt:minimal}, there is a further disk
of radius $d>0$, disjoint from all three disks and contained in the
same container, with $c+d\ge2\beta(a,b)$.
\end{lemma}
\begin{proof}
First justify the geometric existence. Invert about the tangency point
of $a,b$. Their boundaries become parallel lines and their interiors
become the exterior half-planes of a strip. The container interior
becomes the exterior of a circle inscribed in the strip. A disk inside
the strip tangent to both lines has radius equal to half its width;
its center lies on the midline. Tangency to the image of the container
leaves exactly two such disks, on opposite sides of that circle. They
have disjoint interiors. The inversion center lies inside the image
of the container circle, hence outside both solution disks; their inverse
images are bounded disks. Inverting back gives two admissible disks in
the original container, one being $c$ and the other denoted $d$.
This argument identifies the two components of the free region; a
height inequality alone would not establish their disjointness.

Set $S=A+B$, $P=AB$, and $Q=S^2-P$. Then $t^2=P+SC$ and
$c>\beta(a,b)=S/Q$ imply $SC<Q$, hence $t<S$ and $C<S$.
The classical Descartes reflection formula for the two solutions
tangent to $a,b$ and the wall \cite[Section 3.3]{GrahamEtAl2005} gives
\[
 D=\frac1d=2(A+B-1/R_0)-C=4S+C-4t>C>0.
\]
The roots are distinct because $t<S$, and inversion has exhibited
exactly two admissible disks, so the reflection identifies the other one.
Using $t^2=P+SC$ and the displayed expression for $D$ yields
\[
 Q(C+D)-2SCD=2(S-t)^2(2S+2t-C)\ge0.
\]
Divide by $QCD>0$ to obtain $c+d\ge2S/Q=2\beta(a,b)$.
\end{proof}

\subsection{Three largest disks determine local feasibility}

\begin{theorem}[Three largest disks]\label{gt:three}
Let $a>b\ge c\ge r_4\ge\cdots\ge r_n>0$, $n\ge3$, and
suppose every tail is at most $\varphi$ times its corresponding radius.
The entire list fits a disk of radius $R$ if and only if $a,b,c$ fit.
Equality between the smaller radii is permitted.
\end{theorem}
\begin{proof}
Only sufficiency needs proof. Let $U=\sum_{j\ge4}r_j$.
Lemma~\ref{gt:tails}, applied to $a,b$, gives
$c+U\le2\beta(a,b)$.

If $c>\beta(a,b)$, place the three largest disks in their minimum
container $R_0\le R$. Lemma~\ref{gt:opposite} provides a disk
of radius $d\ge2\beta(a,b)-c\ge U$ in the remaining gap. Replace
it by a row containing all the remaining disks.

If $c\le\beta(a,b)$, place $a,b$ diametrally in radius $a+b\le R$
and place $c$ in the lower pocket. For a nonempty remaining list,
write $e=r_4$ and $V=\sum_{j\ge5}r_j$. The two-tail assertion
of Lemma~\ref{gt:tails}, first at $b,c$ and then at $c,e$, gives
\[
 e+V\le b,\qquad V\le c,\qquad e\le c\le\beta(a,b).
\]
Thus Lemma~\ref{gt:arbelos} places the two envelope disks $e,V$
in the upper pocket. Replace $V$ by a row of its actual members;
omit it when $V=0$. The empty remaining list requires no extra step.
These envelopes construct positions; they do not replace the inventory
in the definition of $\rho$ or in a greedy execution.
\end{proof}

\subsection{A uniform exchange of parents}

Suppose a descending execution differs from the lexicographic maximum.
At the first differing radius it must omit a ring of that maximum:
the opposite difference would contradict lexicographic maximality of
the latter. Restrict to the already admitted set $G$ and this omitted ring.
This restricted inventory is feasible and still has $\rho\le\varphi$.
Let $F$ be its greedy forest on $G$. Among all complete witness forests
$P$, choose one having the longest prefix of parents in agreement
with $F$; this maximum exists because there are finitely many parent
assignments. If all parents agree, the omitted ring can be accepted
in $F$, contradicting rejection.

Otherwise let $m$ be the largest disagreement, with destination $u$
in $F$ and origin $v\ne u$ in $P$. Both containers are the root or
holes of rings larger than $m$. All larger parents agree. All smaller
rings may be reconstructed as leaves, since we only need to align
the prefix through $m$.

Write $T$ for the total radius below $m$. If $T\le m$, remove the
smaller rings, put $m$ in $u$ using the prefix certificate from $F$,
and place all smaller rings in a row inside the outer disk vacated
by $m$ in $v$, preserving the larger siblings there. This aligns $m$.
Hence assume $T>m$. Let $p<m$ be the next radius and $U=T-p$.
The first assertion of Lemma~\ref{gt:tails} gives $U\le m$.

Consider the rooted tree formed by the larger rings and the virtual
root. The degree of a container is its number of immediate larger
children. The following three cases cover every such tree.

\paragraph{A branching container with at least three larger children.}
Let $w$ have children $a_1>\cdots>a_k>m$, $k\ge3$. The auxiliary
disk list
\[
 a_1,\ldots,a_k,m,m
\]
satisfies all golden tail bounds: for each $a_i$ its auxiliary tail
is at most its original tail, because $2m<m+T$; the last two tail
ratios are $1,0$. Its three largest disks already coexist in $w$.
Theorem~\ref{gt:three} therefore gives two radius-$m$ slots alongside
all the larger children.

If $w\ne u$, put $p$ in one slot, a row of total $U$ in the other,
and $m$ in $u$ using $F$. This includes $w=v$; no third slot in $w$
is needed. If $w=u$, put $m,p$ in its slots and the row of total $U$
in the outer disk vacated by $m$ in $v$, using $P$ there.

\paragraph{A container with exactly two larger children.}
Let its children be $a>b>m$. Their original tails imply
$m+T\le\varphi b$ and $b+m+T\le\varphi a$; thus
$m+T\le2\beta(a,b)$ by Lemma~\ref{gt:tails}. Since $T>m$,
$\beta(a,b)>m$. The container has radius at least $a+b$, so a diametral
placement of $a,b$ provides two radius-$m$ slots in its opposite pockets.
Apply the preceding slot assignment. This also proves the degree-two
criterion without a separate case enumeration.

\paragraph{No branching container.}
The larger tree is a chain with exactly one leaf container. Since
$u\ne v$, at least one of $u,v$ has exactly one larger child $a$.
If it is $u$, $F$ certifies the pair $a,m$ there; if it is $v$, $P$
does so. In either case that container has radius at least $a+m$.
The original tails give $T\le\varphi m$ and $m+T\le\varphi a$, hence
$T\le2\beta(a,m)$. Reserve $p$; Lemma~\ref{gt:split} divides
the others into two rows of total at most $T/2\le\beta(a,m)$.

When using $u$, put $a,m$ there diametrally, the two rows in its
pockets, and $p$ in the outer disk vacated by $m$ in $v$. When using
$v$, keep $a$ there, put $p$ inside the envelope reserved for a phantom
disk $m$, and put the two rows in the pockets; the real $m$ goes to
$u$ using $F$.

In every case the parents of all larger rings are unchanged. Local
interior packings move with their owners. Those owners have radius
larger than $m$, so the assembly creates no cycle even when one chosen
container is an ancestor of another. The complete new witness agrees
with $F$ through $m$, contrary to the maximal choice of $P$. This proves
Theorem~\ref{thm:golden-global}.

\begin{corollary}[Exact global threshold]\label{cor:golden-global}
The disk threshold is $\tau=\varphi$, and its infimum is not attained.
This holds both for common positive width and for independent positive
widths (equivalently, independent holes smaller than the outer radii).
\end{corollary}
\begin{proof}
Theorem~\ref{thm:golden-global} rules out every failure at
$\rho\le\varphi$. The common-width four-ring family of
Theorem~\ref{thm:golden} has failures tending to $\varphi$ from above
and belongs to both classes.
\end{proof}
The algebra of the angular margin, the opposite-gap identity and the
row-envelope bounds is certified in \texttt{Calamares/ThreeCore.lean}.
Inversion, compactness, Descartes and forest assembly remain written
proofs. The specialized exchange program that follows supplies stronger
information about particular templates; its coverage limits are not
premises of this global result.

%% file: generalizations_v2.tex
\subsection{Container shape and dimension}\label{sec:dimension}

Theorems~\ref{thm:selection} and~\ref{thm:oblivious}, and now
Proposition~\ref{prop:n3}, hold for arbitrary compact containers in
$\R^d$. For spherical containers the negative results also transfer
to higher dimensions, for a different reason: each relevant sibling
query involves at most three balls.

\begin{lemma}[Dimension reduction for sibling queries]\label{lem:dimreduce}
Let $k\ge2$ and $d\ge k-1$. Balls of positive radii $a_1,\ldots,a_k$
pack in a ball of radius $R$ in $\R^d$ if and only if they pack in
one of the same radius in $\R^{k-1}$. A single ball fits if and only
if its radius is at most $R$, in every positive dimension.
\end{lemma}
\begin{proof}
The forward embedding from lower to higher dimension preserves all
center distances and norms. Conversely, center the container at zero
and let $c_1,\ldots,c_k$ be the centers of a packing. Their affine hull
$A$ has dimension at most $k-1$. Let $q$ be the orthogonal projection
of zero onto $A$. Since $q$ is orthogonal to $A-q$,
\[
 \|(c_i-q)-(c_j-q)\|=\|c_i-c_j\|\ge a_i+a_j,
 \qquad \|c_i-q\|^2=\|c_i\|^2-\|q\|^2\le(R-a_i)^2.
\]
Here $R-a_i\ge0$ by the original containment. A common translation
of the centers therefore preserves separation and does not increase
their norms. Identify $A-q$ isometrically with a subspace of
$\R^{k-1}$. This translates the whole affine hull; it does not project
the individual centers onto a predetermined lower-dimensional space.
\end{proof}

\begin{corollary}[Transfer of the sharpness results]\label{cor:dimsharp}
For spherical pans in every dimension $d\ge2$, the following hold with
the same radii and widths as in the plane:
\begin{enumerate}
\item placement obliviousness holds for at most three pieces and fails
  at four, as in Theorem~\ref{thm:n4};
\item the twin instances of Theorem~\ref{thm:twins} exclude every
  deterministic state-based rule, and every randomized state-based
  rule fails on some instance with probability at least $1/2$;
\item the rigid family of Theorem~\ref{thm:rigidfloor} has infimum
  exactly $T$, not attained at positive width;
\item the dimension-dependent geometric threshold satisfies
  $1\le\tau_d\le\varphi$, and its restriction to inventories of
  at most four rings satisfies $\tau_{4,d}=\varphi$, unattained.
\end{enumerate}
\end{corollary}
\begin{proof}
By Lemma~\ref{lem:dimreduce}, every query of at most three siblings
in a ball has the same answer in all dimensions $d\ge2$. A fixed
assignment forest whose containers have at most three children can
be realized by assembling these local packings from the leaves:
each child's entire subtree stays inside its outer ball.

The positive part of (1) is Proposition~\ref{prop:n3}. Its negative
part uses the root pair $\{10,5\}$, the hole pair $\{4.9,4.8\}$,
and the forbidden root triple $\{10,4.9,4.8\}$. All are preserved.
The four executions and the two witnesses of Theorem~\ref{thm:twins}
also use only individual, pair and triple queries. The state at the
arrival of the $5$ is identical in the two inventories, and the
necessary choices are still opposite. If the probability of choosing
the root is $p$, the failure probabilities are $1-p$ and $p$.

In (3), conditions (F1) and (F2) are scalar, and (F3) is precisely
a triple query. Thus the entire admissible parameter family and its
function $\rho$ are unchanged, including the approximating sequence.
Finally, the witness of Theorem~\ref{thm:golden} has the root triple
$\{\varphi,s_1,s_2\}$ and the nested pivot $1$. Its failing execution
only queries pairs and triples as well. It gives the same failures
with $\rho=\varphi+3\varepsilon\to\varphi$. The lower bound in (4)
is Theorem~\ref{thm:oblivious} under weak superincreasingness.
For the four-ring lower bound, the forest reduction in
Theorem~\ref{thm:fourfloor} is dimension-independent. Each of its
two remaining cases uses only pairs and triples in spherical
containers, so the rigid floor and the diametral-pocket argument
apply unchanged. Every four-ring failure therefore still has
$\rho>\varphi$; the golden family gives the matching infimum.
\end{proof}

The corollary does not imply $\tau_d=\tau_2$ or $\tau_d=\varphi$:
other counterexamples may involve four or more siblings. Nor does
monotonicity of packing feasibility in the dimension, by itself,
give monotonicity of the threshold for failure of a greedy execution.

\begin{proposition}[The first dimensional separation]\label{prop:dimfour}
There are four distinct sibling radii that fit in the unit ball
in $\R^3$ but not in the unit disk in $\R^2$, namely
$\{441,440,439,438\}/1000$.
\end{proposition}
\begin{proof}
Set $u=8/25$ and use the tetrahedral centers
\[
(u,u,u),\quad(u,-u,-u),\quad(-u,u,-u),\quad(-u,-u,u).
\]
Each norm squared is
$3u^2$ and each pair distance squared is $8u^2$. Even with all
radii enlarged to $441/1000$, containment and separation follow from
\[
 (1-441/1000)^2-3u^2=5281/10^6>0,\qquad
 8u^2-(882/1000)^2=10319/250000>0.
\]
For planar infeasibility, shrink all radii to $r=438/1000$.
The centers would lie in radius $1-r$. A center at zero is impossible
since $1-r<2r$. Otherwise two consecutive centers in angular order
have angular separation at most $\pi/2$, so their squared distance
is at most $2(1-r)^2<(2r)^2$; the difference is $16961/125000$.
This contradicts separation. This is a difference of sibling
feasibility, not a new failure of placement obliviousness.
\end{proof}

\subsection{Square pans: confinement and twin instances}\label{sec:square}

Only the root changes to $[0,s]^2$; every ring's hole is still a disk.
The capacity used by best fit at the root is its inradius $s/2$.
The following sufficient criterion excludes all placements of a trio,
without assuming an optimal tangency pattern.

\begin{lemma}[Quadrant confinement]\label{lem:squareQ}
Let $a\ge b\ge c>0$, $a\le s/2$, and $p\ge q\ge0$. If
\begin{align}
 p&\ge s/2-b, & p^2+(s-a-b)^2&<(a+b)^2,\label{eq:squareQb}\\
 q^2+(s-a-c)^2&<(a+c)^2, &p+q&\ge s-b-c,\label{eq:squareQc}\\
 L:=s-c-a-q&\ge0, &2L^2&<(b+c)^2,\label{eq:squareQbox}
\end{align}
then disks of radii $a,b,c$ do not pack in the square of side $s$.
\end{lemma}
\begin{proof}
Reflect the whole packing about either midline as needed, so the
largest center $A$ has both coordinates in $[a,s/2]$. The coordinates
of $B,C$ lie in $[b,s-b]$ and $[c,s-c]$, respectively. Each absolute
coordinate difference between $A$ and $B$ is at most $s-a-b$.
If $|B_x-A_x|\le p$, the first quadratic inequality contradicts
separation. The negative alternative $B_x-A_x<-p$ is excluded by
$A_x-B_x\le s/2-b\le p$. Thus $B_x>A_x+p$, and likewise
$B_y>A_y+p$, whence $A_x,A_y<s-b-p$.

The same argument with the other quadratic inequality gives
$|C_x-A_x|>q$. Its negative alternative is impossible because
$A_x-C_x<s-b-p-c\le q$. Therefore $C_x>A_x+q$ and
$C_y>A_y+q$. Since $p\ge q$ and $b\ge c$, both small centers lie
in the box $[a+q,s-c]^2$. Their squared distance is at most $2L^2$,
contradicting~\eqref{eq:squareQbox}. All separations used for
packings are non-strict; tangencies have not been excluded.
\end{proof}

\begin{proposition}[An exact square counterexample]\label{prop:squareD}
With side $s=4.8568$, width $w=0.16$ and radii
$\{a,1,b,c\}=\{1.845,1,0.844,0.841\}$, all four rings fit but
best fit places exactly three. The dominant tail is $\rho=337/200$.
Every decimal here and below denotes the exact terminating rational.
\end{proposition}
\begin{proof}
Apply Lemma~\ref{lem:squareQ} with $p=1.591$, $q=1.581$.
The margins in its five inequalities are, in their displayed order,
\[
 \frac{33}{5000},\quad\frac{2079}{25000000},\quad
 \frac{66559}{25000000},\quad\frac1{5000},\quad
 \frac{53587423}{25000000},\qquad L=\frac{2949}{5000}.
\]
The witness places $a$ at $(a,a)$ and $1$ at $(s-1,s-1)$:
$2(s-a-1)^2-(a+1)^2=16337/25000000>0$. In the hole of $a$,
of radius $H=a-w=b+c=1.685$, place the disk of radius $b$ at
relative center $(-c,0)$ and that of radius $c$ at $(b,0)$.

Best fit instead nests $1$ in $a$, since $1\le H<s/2$.
The disk $b$ then fits only at the root: $1+b>H$ and $b>1-w$.
Root feasibility follows by shrinking the disk $1$ in the witness.
At $c$, the root triple is excluded by the lemma, and all holes are
excluded by $1+c>H$, $c>1-w$ and $c>b-w$. The tails are
$179/123$, $337/200$, and $841/844$, so the middle one dominates.
Together with Proposition~\ref{prop:n3}, this establishes the same
three/four-piece transition for a square.
\end{proof}

\begin{theorem}[Square twins]\label{thm:squaretwins}
Let $s=5.1214$, $w=0.301$, and
\[
 I_1=\{2,1,0.850,0.849\},\qquad
 I_2=\{2,1,0.950,0.750\}.
\]
Both full inventories are feasible. In $I_1$ the pivot $1$ must be
placed at the root for the greedy to succeed; in $I_2$ it must be
nested in $2$. Hence no state-based rule succeeds universally, and
every randomized state-based rule has failure probability at least
$1/2$ on one of these instances.
\end{theorem}
\begin{proof}
The hole of $2$ has radius $H=1.699$. Lemma~\ref{lem:squareQ}
excludes $\{2,0.850,0.849\}$ using $(p,q)=(1.72,1.718)$, and
excludes $\{2,1,0.750\}$ using $(p,q)=(2.12,1.39)$; direct rational
substitution verifies all its inequalities. Shrinking the last disk
therefore excludes $\{2,1,z\}$ for every $z\ge0.750$.
On the other hand $\{2,0.950,0.750\}$ has the explicit centers
\[
 (2.96,2),\qquad(0.95,4.1714),\qquad(4.3714,4.3714).
\]
Containment is immediate. The three positive separation margins are
\[
\frac{1314449}{25000000},\qquad\frac{663599}{12500000},\qquad
\frac{221399449}{25000000}.
\]

The root pair $\{2,1\}$ fits in opposite corners since
$2(s-3)^2>9$. Shrinking its unit disk at the same center also
packs $\{2,z\}$ for every $0<z<1$, justifying admission of the
first small ring when the pivot is nested.
This pair, with the two small rings inside $2$,
witnesses $I_1$, as $0.850+0.849=H$.
For $I_2$, use the displayed root triple and nest $1$ in $2$.
Every small ring fits alone in $H$, but none fits there beside $1$,
since even $1+0.750>H$. None fits in the pivot's hole of radius
$0.699$, and the last cannot nest in the preceding small ring:
$0.849>0.850-w$ and $0.750>0.950-w$.

If the pivot is nested, both small rings are forced to the root,
where the first triple is infeasible and the second feasible.
If the pivot is at the root, both are excluded from the root and
must use $H$, where only the first pair fits
($0.950+0.750=1.700>H$). All subsequent decisions are therefore
forced. At the pivot, the state, side, width and inventory size are
identical. If the root is chosen with probability $p$, the failure
probabilities are $1-p$ and $p$. The dominant tails are $1.699$ and
$1.700$, respectively.
\end{proof}

\subsection{An algebraic upper bound for the square threshold}
\label{sec:squarelimit}

Define $\tau_{\square}$ as the infimum of $\rho$ over square-pan
instances admitting a failing greedy execution, with side and positive
width allowed to vary, as in the definition of $\tau$.

\begin{theorem}[An approximating square family]\label{thm:squarelimit}
Let $Y$ be the unique positive root of
\[
 (17+10\sqrt2)Y^2+(72+16\sqrt2)Y-(112+96\sqrt2)=0.
\]
Then
\[
 1\le\tau_{\square}\le Y=1.684487745872346\ldots.
\]
There are four-ring counterexamples to best fit with strictly
decreasing rational radii, positive rational width and rational side,
whose values of $\rho$ tend to $Y$ from above. No optimality of $Y$
is asserted, either globally or for all parameters of the confinement
criterion.
\end{theorem}
\begin{proof}
Put $k=\sqrt2$, $C=1+k/2$, and consider first the balanced data
$a_0=1+t$, $b_0=c_0=t$, $s_0=C(2+t)$,
$p_0=q_0=s_0/2-t$. The common quadratic margin is
\[
 G(t)=(a_0+t)^2-p_0^2-(s_0-a_0-t)^2
 =\frac{(17+10k)t^2+(36+8k)t-(28+24k)}8.
\]
It is strictly increasing for $t\ge0$ and has exactly one positive
root $t_0=Y/2$. Indeed,
\[
 G(21/25)=\frac{8897}{5000}-\frac{639k}{500}<0,
 \qquad
 G(17/20)=\frac{5953}{3200}-\frac{399k}{320}>0,
\]
using $7/5<k<10/7$. For $21/25<t<17/20$,
\[
 L_0=s_0/2-a_0=k/2+(k-2)t/4,\qquad
 \tfrac12<\tfrac{229}{400}<L_0<\tfrac57,
 \qquad 2L_0^2<\tfrac{50}{49}<4t^2.
\]
Thus containment of $a_0$, positivity of $p_0$, and the final
confinement inequality have strict margins.

Fix any $t_0<t<17/20$ and perturb by $\eta>0$:
\begin{gather*}
 a=1+t+2\eta,\quad m=1,\quad b=t+\eta,\quad c=t-\eta,\\
 w=1-t+2\eta,\quad H=a-w=2t,\quad
 s=C(a+1),\quad p=q=s/2-t.
\end{gather*}
At $\eta=0$ both quadratic margins equal $G(t)>0$. The strict
inequalities just established, including $a<s/2$ and $p>0$,
are continuous in $\eta$; finitely many such inequalities remain
strict on some interval $0<\eta<\eta_0(t)$. The two remaining
walls hold identically:
\[
 p-(s/2-b)=\eta>0,\qquad p+q-(s-b-c)=0.
\]
Choose also $\eta<t$, $\eta<1-t$, and $3\eta<2t-1$.
Then $a>1>b>c>w>0$, and Lemma~\ref{lem:squareQ} excludes the
root triple $\{a,b,c\}$.

The witness places $\{a,1\}$ in opposite corners and $\{b,c\}$
on a diameter of the large hole. Both are legal because
$2(s-a-1)^2=(a+1)^2$, both root radii are at most $s/2$, and
$b+c=H$. Best fit nests $1$, since $1<H<a<s/2$.
The differences excluding the alternative holes for the two
small rings are
\begin{gather*}
 1+b-H=1-t+\eta>0,\quad 1+c-H=1-t-\eta>0,\\
 b-(1-w)=3\eta>0,\quad c-(1-w)=\eta>0,\quad
 c-(b-w)=1-t>0.
\end{gather*}
The first small ring fits at the root by shrinking the pivot in
the witness; the last is excluded everywhere. Best fit places
exactly three rings. Moreover
\[
 \frac{1+b+c}{a}=\frac{1+2t}{1+t+2\eta}<2t,
 \qquad \frac cb<1<2t,
\]
where the first inequality uses $1<2t^2+4t\eta$ and $t>21/25$.
Thus $\rho=2t\to2t_0=Y$ from above, proving the upper bound.
The lower bound follows from Theorem~\ref{thm:oblivious}.

For entirely rational instances choose rational $t\downarrow t_0$
and then rational $\eta$ in the nonempty interval above. Replace
$s=C(a+1)$ by a slightly larger rational and reset $p=q=s/2-t$.
All strict inequalities persist by continuity in $s$, the two
displayed wall identities persist exactly, and the root pair gains
strict separation. Density of the rationals completes the construction.
The unperturbed balanced limit itself is not a strict-radius
counterexample and is not used as one.
\end{proof}

Eliminating the radical shows that $Y$ also satisfies
\[
 89Y^4+1808Y^3+4704Y^2-9984Y-5888=0.
\]
The positive branch is specified by the quadratic over $\mathbb Q(\sqrt2)$,
not by an arbitrary choice of a positive root of this quartic.
Exact rational examples in the accompanying code attain
$\rho=1.684488$, $1.684487746$, and $1.6844877458724$.
These finite checks certify only those instances; the limiting bound
uses the continuity proof above.

\paragraph{Formal verification scope.}
\texttt{Square.lean} proves Lemma~\ref{lem:squareQ} as a universal
Cartesian exclusion over linearly ordered commutative rings, including
the reflections. \texttt{SquareTwins.lean} proves the required twin
exclusions and the explicit feasible triple with the same Cartesian
predicate. \texttt{SquareLimit.lean} checks the balanced-gap and norm
polynomial identities. These files contain 24 theorems, in addition
to the 57 earlier algebraic certificates and the three in
\texttt{FourRing.lean}. The Euclidean interpretation,
assignment forests, algorithmic arguments, dimension reduction and
continuity remain written proofs; they are not claimed as fully
formalized Lean theorems. All new numerical checks use exact rational
arithmetic, and all displayed decimals specifying instances are exact.

%% file: extras_v2.tex
\subsection{Up to five rings in every spherical dimension}\label{sec:fivedim}

The global proof above is stated for the plane. The following separately
verified extension covers up to five rings in every spherical dimension;
it uses explicit pocket constructions, not a dimension reduction for
four or five arbitrary siblings.

\begin{theorem}[Five-ring golden threshold]\label{thm:five-dim}
In a spherical pan of any dimension $d\ge2$, with strictly decreasing
positive outer radii and independent holes $0\le h_i<r_i$, every
descending greedy execution on at most five rings returns the lex-max
set when $\rho\le\varphi$. Thus $\tau_{\le5,d}=\varphi$, with
unattained infimum. The threshold is also $\varphi$ for exactly five rings.
\end{theorem}
\begin{proof}
The diametral pair and both pocket centers preceding
Lemma~\ref{gt:arbelos} embed in every dimension $d\ge2$: all norms
and pairwise distances remain unchanged. Two scalar consequences of
Lemma~\ref{gt:tails} suffice. First, for $a>m>p\ge q>0$, the tail
bounds $p+q\le\varphi m$ and $m+p+q\le\varphi a$ give
$p+q\le2\beta(a,m)$, hence $q\le\beta(a,m)$. Second, for
$A>B>m$ and $T>m$, the bounds $m+T\le\varphi B$ and
$B+m+T\le\varphi A$ give $m+T\le2\beta(A,B)$, hence
$\beta(A,B)>m$.

Restrict a first disagreement with the lex-max set to the admitted
prefix and the omitted ring, as in the proof of
Theorem~\ref{thm:golden-global}, and align a complete witness with the
greedy forest. A disagreement at a ring $m$ with at most one smaller
ring is resolved by the row exchange. Only the second and third rings
can therefore remain. Write $u$ for the destination of $m$ in the
greedy forest and $v\ne u$ for its parent in the witness.

At the second ring, write the inventory as $A>m>p>q>t$, omitting
absent terms. The disagreement certifies both $R\ge A+m$ and
$h_A\ge m$. If there are fewer than two smaller rings, use the row
exchange. Otherwise the first scalar consequence above gives
$q,t\le\beta(A,m)$. If $u$ is the root, place the diametral pair
$A,m$ there, $p$ in the hole of $A$, and $q,t$ in its two pockets.
If $u$ is the hole of $A$, use a phantom pair $A,m$ at the root:
put $p$ in the phantom $m$ disk and $q,t$ in the pockets, while the
real $m$ goes inside $A$. All smaller rings may be made leaves.

At the third ring, a blocked row exchange requires exactly five
rings $A>B>m>p>q$, with $p,q$ both immediate children of $u$ in
the witness and $p+q>m$. In particular $m$ has no smaller descendants.
The parents of $A,B$ already agree. If both are at the root, then
$R\ge A+B$ and $\beta(A,B)>m$. When $u$ is the root, place $m,p$
in its two pockets and $q$ in the disk vacated by $m$ in $v$.
Otherwise put $p,q$ in the pockets and $m$ in its certified hole $u$.

It remains that $B$ is inside $A$. If $u$ is the root, its certificate
admits the pair $A,m$; put $q$ in one pocket and $p$ in the disk
vacated in $v$, keeping $B$ inside $A$. If $u$ is the hole of $A$,
do the same with the pair $B,m$ inside $A$. In both cases the first
scalar consequence bounds $q$ by the relevant pocket radius.
If $u$ is the hole of $B$, then $v$ is either the root or the hole
of $A$. Use respectively the phantom pair $A,m$ in the root or
$B,m$ in the hole of $A$, put $p$ in the phantom disk and $q$ in
a pocket, and place the real $m$ in $B$. These cases exhaust $u\ne v$.
All constructions preserve the larger parents and assemble locally,
so each step increases the aligned prefix. This proves the guarantee.

The four-ring family of Theorem~\ref{thm:golden} supplies the upper
bound in every $d\ge2$, by Corollary~\ref{cor:dimsharp}. For exactly
five rings, add a radius $\delta>0$ in the empty hole of the radius-$1$
ring in its witness, with $\delta<\min(w,1-w,s_2)$. The fifth ring
is processed after the existing first failure, which therefore persists.
Its addition raises each old tail ratio by at most $\delta/s_2$.
Let $\delta$ tend to zero together with the golden-family parameter.
The common-width example belongs also to the independent-hole class.
The lower bound rules out attainment in either class.
\end{proof}

\subsection{Independent holes and the exact area guarantee}\label{sec:variablewidth}

Let ring $i$ have outer radius $r_i$ and inner radius $h_i$, with
$0\le h_i<r_i$. Its area divided by $\pi$ is
$a_i=r_i^2-h_i^2>0$. There need not be a single value function of
outer radius representing these areas.

\begin{proposition}[Independent-hole placement]\label{prop:independent}
For arbitrary compact pans in any positive dimension, independent holes
preserve placement obliviousness under $\rho\le1$, including equality.
The unconditional guarantee for at most three rings also remains valid.
\end{proposition}
\begin{proof}
The parent-alignment proof of Theorem~\ref{thm:oblivious} uses only
outer radii in the row exchange. At a pivot $m$, all displaced smaller
children have total outer radius at most $r_m$, so they and their
subtrees fit in the outer disk or ball vacated by $m$. A hole that
previously contained $m$ certifies enough capacity for the replacement;
no comparison of two widths is needed. The chosen destination admits
$m$ with its larger siblings under the greedy certificate. Parent radii
strictly increase, so the assembly is acyclic. The same downward closure
and maximal-prefix argument applies at equality. For at most three
rings, the proof of Proposition~\ref{prop:n3} similarly only uses
$r_3<r_2\le h_1$ when replacing one child by another.
\end{proof}

In a disk, Theorem~\ref{thm:golden-global} strengthens the placement
guarantee to $\rho\le\varphi$. Neither placement statement by itself
implies optimality of the independent-hole contact areas.

\begin{lemma}[Tail capture]\label{lem:tailcapture}
Let $L$ be the lex-max feasible set and $T_i=\sum_{j>i}r_j$.
If $i\in L$ and $h_i\ge T_i$, then every $j>i$ belongs to $L$.
\end{lemma}
\begin{proof}
In a witness of the selected prefix through $i$, its hole is empty.
Place all later rings there as a row of outer disks. This produces a
feasible completion containing every later index, without adding any
previously omitted index. Lexicographic maximality gives the claim.
\end{proof}

\begin{theorem}[Sharp area factor]\label{thm:variable-area}
Fix $0<\kappa<1$. For independent holes in any compact planar pan
and $\rho\le\kappa$, every descending greedy execution satisfies
\[
 A_{\rm greedy}\ge c(\kappa) A_{\rm opt},\qquad
 c(\kappa)=\min(1,\kappa^{-2}-1).
\]
For $\kappa\le1/\sqrt2$, the lex-max set is the unique area-optimal
set. For $1/\sqrt2<\kappa<1$, the constant is best possible even for
two rings in a disk; its infimum over positive-optimum instances is
not attained.
\end{theorem}
\begin{proof}
The greedy set is $L$ by Proposition~\ref{prop:independent}. Compare
it with any distinct feasible set $S$. Their first difference is an
index $i\in L\setminus S$. Let $C$ be the normalized area of their
common prefix and $B$ the normalized area of members of $S$ after $i$.
If $h_i\ge T_i$, Lemma~\ref{lem:tailcapture} gives $S\subsetneq L$,
so positivity gives $A(L)>A(S)$. Otherwise
$0\le h_i<T_i\le\kappa r_i$, and
\[
 B\le\sum_{j\in S,\ j>i}r_j^2\le T_i^2,\qquad
 a_i>r_i^2-T_i^2\ge(\kappa^{-2}-1)T_i^2\ge c(\kappa)B.
\]
Since $C\ge0$ and $0<c(\kappa)\le1$, this yields
$A(L)\ge C+a_i>c(\kappa)(C+B)=c(\kappa)A(S)$, with the common
factor $\pi$ restored if desired. For $c=1$ this proves uniqueness;
otherwise compare with an area optimum. If that optimum is $L$, the
guarantee is immediate and its ratio is $1>c$.

For sharpness when $\kappa>1/\sqrt2$, take a radius-$1$ pan and
\[
 (r_1,h_1)=(1,\kappa-\varepsilon),\qquad
 (r_2,h_2)=(\kappa,0),\qquad
 0<\varepsilon<\kappa-\sqrt{1-\kappa^2}.
\]
The first outer disk fills the root, and its hole is too small for the
second ring; no set containing both is feasible. The greedy takes the
first, while the optimum is the second. Here $\rho=\kappa$ and
\[
 \frac{A_{\rm greedy}}{A_{\rm opt}}
 =\kappa^{-2}-1+\frac{2\varepsilon}{\kappa}
                  -\frac{\varepsilon^2}{\kappa^2}
 \longrightarrow\kappa^{-2}-1
\]
from above. Both widths are positive. The strict comparison already
proved excludes attainment of this infimum.
\end{proof}

Consequently the universal area-optimality threshold for independent
holes is exactly $1/\sqrt2$, including its good endpoint. Under strict
superincreasing radii alone there is no positive uniform area factor: with
$0<e<1/4$, use pan radius $1+e$ and
$(r_1,h_1)=(1+e,1-e)$, $(r_2,h_2)=(1,0)$. The greedy-to-optimum
ratio is $4e$, whereas $\rho=1/(1+e)<1$. The seven theorems of
\texttt{Calamares/VariableWidth.lean} certify the algebra, not the
capture lemma or the geometric assembly. The five-ring proof uses the
pocket certificates of \texttt{FiveRing.lean}; all its placements above
are explicit.

%% file: main.bbl
\begin{thebibliography}{99}

\bibitem{GrahamEtAl2005}
R.~L. Graham, J.~C. Lagarias, C.~L. Mallows, A.~R. Wilks, C.~H. Yan.
Apollonian Circle Packings: Geometry and Group Theory I. The Apollonian Group.
\emph{Discrete \& Computational Geometry} 34:547--585, 2005.
\url{https://arxiv.org/abs/math/0010298}.


\bibitem{PedrosoCunhaTavares2016}
J.~P. Pedroso, S.~Cunha, J.~N. Tavares.
Recursive circle packing problems.
\emph{International Transactions in Operational Research} 23(1--2):355--368,
2016.

\bibitem{ChenEtAl2018}
M.~Chen, X.~Tang, T.~Song, Z.~Zeng, X.~Peng, S.~Liu.
Greedy heuristic algorithm for packing equal circles into a circular
container.
\emph{Computers \& Industrial Engineering} 119:114--120, 2018.

\bibitem{Gleixner2020}
A.~Gleixner, S.~J. Maher, B.~M\"uller, J.~P. Pedroso.
Price-and-verify: a new algorithm for recursive circle packing using
Dantzig--Wolfe decomposition.
\emph{Annals of Operations Research} 284(2):527--555, 2020.

\bibitem{CoffmanGareyJohnson1987}
E.~G. Coffman~Jr., M.~R. Garey, D.~S. Johnson.
Bin packing with divisible item sizes.
\emph{Journal of Complexity} 3(4):406--428, 1987.

\bibitem{DemaineFeketeLang2010}
E.~D. Demaine, S.~P. Fekete, R.~J. Lang.
Circle packing for origami design is hard.
In \emph{Origami$^5$: Proc.\ 5th International Conference on Origami in
Science, Mathematics and Education}, A K Peters, 2011, pp.\ 609--626.

\bibitem{FeketeKeldenichScheffer2019}
S.~P. Fekete, P.~Keldenich, C.~Scheffer.
Packing disks into disks with optimal worst-case density.
In \emph{Proc.\ 35th International Symposium on Computational Geometry (SoCG)},
LIPIcs 129, 35:1--35:19, 2019. Journal version:
\emph{Discrete \& Computational Geometry} 69:51--90, 2023.

\bibitem{AbrahamsenMiltzowSeiferth2020}
M.~Abrahamsen, T.~Miltzow, N.~Seiferth.
Framework for $\exists\mathbb R$-completeness of two-dimensional packing
problems.
In \emph{Proc.\ 61st IEEE Annual Symposium on Foundations of Computer
Science (FOCS)}, pp.\ 1014--1021, 2020.

\bibitem{FeketeMorrScheffer2019}
S.~P. Fekete, S.~Morr, C.~Scheffer.
Split packing: algorithms for packing circles with optimal worst-case density.
\emph{Discrete \& Computational Geometry} 61(3):562--594, 2019.

\bibitem{GLNO1998}
R.~L. Graham, B.~D. Lubachevsky, K.~J. Nurmela, P.~R.~J. \"Osterg\aa rd.
Dense packings of congruent circles in a circle.
\emph{Discrete Mathematics} 181:139--154, 1998.

\bibitem{Pirl1969}
U.~Pirl.
Der Mindestabstand von $n$ in der Einheitskreisscheibe gelegenen Punkten.
\emph{Mathematische Nachrichten} 40:111--124, 1969.

\bibitem{Melissen1994}
H.~Melissen.
Densest packings of eleven congruent circles in a circle.
\emph{Geometriae Dedicata} 50:15--25, 1994.

\bibitem{Fodor1999}
F.~Fodor.
The densest packing of 19 congruent circles in a circle.
\emph{Geometriae Dedicata} 74:139--145, 1999.

\bibitem{EkanayakeLaFountain2024}
D.~B. Ekanayake, D.~J. LaFountain.
Tight partitions for packing circles in a circle.
\emph{Italian Journal of Pure and Applied Mathematics} 51:115--136, 2024.

\bibitem{Edmonds1971}
J.~Edmonds.
Matroids and the greedy algorithm.
\emph{Mathematical Programming} 1:127--136, 1971.

\bibitem{KorteHausmann1978}
B.~Korte, D.~Hausmann.
An analysis of the greedy heuristic for independence systems.
\emph{Annals of Discrete Mathematics} 2:65--74, 1978.

\bibitem{Gupte2016}
A.~Gupte.
Convex hulls of superincreasing knapsacks and lexicographic orderings.
\emph{Discrete Applied Mathematics} 201:150--163, 2016.

\bibitem{Litvinchev2014}
I.~Litvinchev, L.~Infante, E.~L. Ozuna Espinosa.
Approximate circle packing in a rectangular container: integer programming
formulations and valid inequalities.
In \emph{Computational Logistics (ICCL 2014)}, LNCS 8760, Springer, 2014,
pp.\ 47--60.

\end{thebibliography}
